%% file: semi-simplicity.tex
\documentclass[11pt]{amsart}
\usepackage{etex}
\input defs.tex
\begin{document}
\vbadness=10001
\title[Semisimplicity of Hecke and (walled) Brauer algebras]{Semisimplicity of Hecke and (walled) Brauer algebras}

\author{Henning Haahr Andersen}
\address{H.H.A.: Centre for Quantum Geometry of Moduli Spaces, Aarhus University, Ny Munkegade 118, building 1530, room 327, 8000 Aarhus C, Denmark}
\email{mathha@qgm.au.dk}

\author{Catharina Stroppel}
\address{C.S.: Mathematisches Institut, Universit\"at Bonn, Endenicher Allee 60, Room 4.007, 53115 Bonn, Germany}
\email{stroppel@math.uni-bonn.de}

\author{Daniel Tubbenhauer}
\address{D.T.: Mathematisches Institut, Universit\"at Bonn, Endenicher Allee 60, Room 1.003, 53115 Bonn, Germany}
\email{dtubben@math.uni-bonn.de}

\thanks{H.H.A. was supported by the center of excellence grant ``Centre for Quantum Geometry of Moduli Spaces (QGM)'' from the ``Danish National Research Foundation (DNRF)'', C.S. by a Hirzebruch professorship of the Max-Planck-Gesellschaft and D.T. by a research funding of the ``Deutsche Forschungsgemeinschaft (DFG)'' during this work.}

\vbadness=10001
\begin{abstract}
We show how to use Jantzen's sum formula for Weyl modules
to prove semisimplicity criteria for endomorphism algebras of 
$\Uq$-tilting modules 
(for any field $\K$ and any parameter $q\in\K-\{0,-1\}$). As an application, 
we recover the semisimplicity criteria for the Hecke 
algebras of types $\Am$ and $\Bm$, the walled Brauer algebras and the Brauer algebras 
from our more general approach.
%
%
\end{abstract}

\vspace*{-1cm}

\maketitle

\tableofcontents

\vspace*{-1cm}
%
%
\section{Introduction}\label{sec-intro}
\input{res/intro.tex}
\paragraph*{\textbf{Acknowledgements}} We like to thank Michael Ehrig, Steffen K\"onig, 
Jonathan Kujawa, Gus Lehrer, 
Andrew Mathas and Antonio Sartori for helpful suggestions, comments and discussions, and 
the referee for further helpful comments.
We would also like to thank the Institut Mittag-Leffler: a major part of the 
research for this paper was done while the authors enjoyed the 
hospitality and the excellent working 
conditions of the Institut Mittag-Leffler. C.S. and D.T. want to thank 
the Max-Planck Institute in Bonn for the 
extraordinary working 
conditions and for sponsoring some research visits. D.T. likes to thank the 
Belgian Lambic for providing a refreshment during the summer 2015.
%
\section{\texorpdfstring{$\Uq$}{Uq}-tilting modules and semisimplicity}\label{sec-prelim}
\input{res/prelim.tex}
\section{Several versions of Schur-Weyl dualities}\label{sec-schur}
\input{res/schur.tex}
%
\section{Some kernels of Schur-Weyl actions}\label{sec-schuridem}
\input{res/schuridem.tex}
%
\section{Semisimplicity: the Hecke algebras of types \texorpdfstring{$\Am$}{A} and \texorpdfstring{$\Bm$}{B}}\label{sec-hecke}
\input{res/hecke.tex}
\section{Semisimplicity: the walled Brauer algebra}\label{sec-walled}
\input{res/walled-brauer.tex}
\section{Semisimplicity: the Brauer algebra}\label{sec-brauer}
\input{res/brauer.tex}
%
\appendix
\renewcommand{\thesection}{\Alph{section}}
\section{From positive characteristic to characteristic zero}\label{sec-app2}
\input{res/appendix2.tex}
\section{Root systems of types \texorpdfstring{$\AM,\BM,\CM$}{Am-1,Bm,Cm} and \texorpdfstring{$\DM$}{Dm}}\label{sec-app}
\input{res/appendix.tex}


\bibliographystyle{plainurl}
\bibliography{semi-simplicity}
\vspace*{-.1cm}
\end{document}

%% file: defs.tex
\usepackage{latexsym,exscale,enumerate,amsfonts,amssymb,mathtools}
\usepackage{amsmath,amsthm,amsfonts,amssymb,amscd,stmaryrd,textcomp}
\usepackage[normalem]{ulem}
\usepackage{young}
\usepackage{thmtools,xcolor}
\usepackage{easybmat}

\definecolor{colormy}{rgb}{0.8,0.05,0.05}
\definecolor{mycolor}{rgb}{0.25,0.99,0.25}

\usepackage[all]{xy}
\SelectTips{cm}{}

\usepackage{tikz}
\usetikzlibrary{decorations.markings}
\usetikzlibrary{decorations.pathreplacing}
\usetikzlibrary{arrows,shapes,positioning}
\tikzstyle directed=[postaction={decorate,decoration={markings,
    mark=at position #1 with {\arrow{>}}}}]
\tikzstyle rdirected=[postaction={decorate,decoration={markings,
    mark=at position #1 with {\arrow{<}}}}]

\usepackage{hyperref}

\def\cal#1{\mathcal{#1}}%

\newcommand{\sll}[1]{\mathfrak{sl}_{#1}}
\newcommand{\gll}[1]{\mathfrak{gl}_{#1}}

\newcommand{\soo}[1]{\mathfrak{so}_{#1}}
\newcommand{\spo}[1]{\mathfrak{sp}_{#1}}
\newcommand{\fg}{\mathfrak{g}}

\newcommand{\Endrr}{\mathrm{End}_{\Uq}(T)}
\newcommand{\Mod}[1]{{#1}\text{-}\textbf{Mod}}

\newcommand{\AM}{\textbf{A}_{m-1}}
\newcommand{\BM}{\textbf{B}_{m}}
\newcommand{\CM}{\textbf{C}_{m}}
\newcommand{\DM}{\textbf{D}_{m}}
\newcommand{\Am}{\textbf{A}}
\newcommand{\Bm}{\textbf{B}}
\newcommand{\Bmm}{\textbf{B}_d}
\newcommand{\Cm}{\textbf{C}}
\newcommand{\Dmm}{\textbf{D}}

\newcommand{\Uu}{\textbf{U}}

\newcommand{\Uq}{\textbf{U}_q}
\newcommand{\Uo}{\textbf{U}_1}
\newcommand{\Uoo}{\tilde{\textbf{U}}_1}

\newcommand{\End}{\mathrm{End}}

\newcommand{\Char}{\mathrm{char}}
\newcommand{\ord}{\mathrm{ord}}

\newcommand{\Dl}{\Delta_q(\lambda)}
\newcommand{\Nl}{\nabla_q(\lambda)}
\newcommand{\Ll}{L_q(\lambda)}

\newcommand{\T}{\boldsymbol{\mathcal{T}}}

\def\C{{\mathbb C}}
\def\N{{\mathbb{Z}_{\geq 0}}}
\def\Z{{\mathbb Z}}
\def\Zg{{\mathbb{Z}_{>0}}}

\def\K{{\mathbb K}}
\def\R{{\mathbb R}}

\def\F{{\mathbb F}_p}

\theoremstyle{definition}
\newtheorem{thm}{Theorem}[section]
\newtheorem{cor}[thm]{Corollary}
\newtheorem{lem}[thm]{Lemma}
\newtheorem{prop}[thm]{Proposition}

\theoremstyle{definition}
\newtheorem*{thmm}{Theorem}

\numberwithin{equation}{section}

\declaretheorem[style=definition,name=Example,qed=$\blacktriangle$,numberlike=thm]{ex}
\declaretheorem[style=definition,name=Definition,qed=$\blacktriangle$,numberlike=thm]{defn}
\declaretheorem[style=definition,name=Remark,qed=$\blacktriangle$,numberlike=thm]{rem}
\declaretheorem[style=definition,name=Conventions,qed=$\blacktriangle$,numberlike=thm]{nota}



%% file: res/intro.tex
Fix a reductive Lie algebra 
$\fg$, a field $\K$ and any 
$q\in\K^{\ast}$, where $\K^{\ast}=\K-\{0,-1\}$ if $\Char(\K)>2$ and 
$\K^{\ast}=\K-\{0\}$ otherwise. Let $\Uq=\Uq(\fg)$ be the 
$q$-deformed enveloping algebra of $\fg$ over $\K$
and let $T$ be a $\Uq$-tilting module.

In this paper we give a semisimplicity 
criterion for the algebra $\Endrr$ which only 
relies on the combinatorics of the root and weight data associated 
to $\fg$. 
The crucial observation 
we use here is that $\Endrr$ is 
semisimple if and only if all Weyl factors of $T$ are simple $\Uq$-modules -- 
a property which can be checked using 
(versions of) Jantzen's sum formula.

We apply our methods to four explicit examples: the Hecke 
algebras of types $\Am$ and $\Bm$, the walled Brauer algebras and 
the Brauer algebras. For all of these we obtain full
semisimplicity criteria 
by using the corresponding 
combinatorics of roots and weights. In all of these 
cases the semisimplicity criteria were obtained before, but using 
specific properties of the algebras in question, see 
Remarks~\ref{rem-semi1},~\ref{rem-semi2} and~\ref{rem-semi3}.
However, our approach has the advantage that it provides 
a quite general method to deduce semisimplicity criteria.  
The necessary calculations to prove these are always the same 
(\textit{mutatis mutandis}, depending on the associated root and weight data). 
Hence, our approach unites the known 
semisimplicity criteria of these algebras in our 
more general 
framework.

\subsection{The setup}\label{sub-intropart1}

The category $\Mod{\Uq}$ of finite-dimensional
representations of $\Uq$ (of type $1$)
provides an interesting example of a tensor 
category. The structure of $\Mod{\Uq}$ 
heavily depends on the field $\K$ 
and on $q\in\K^{\ast}$. 
If $\Char(\K)=0$ and $q=1$, then we 
are in the \textit{classical case} where $\Mod{\Uq}$ 
behaves like the category $\Mod{\fg}$ of complex, 
finite-dimensional representations of $\fg$, and hence, 
is in particular semisimple. But $\Mod{\Uq}$ is 
non-semisimple in case $\Char(\K)>0$ and $q=1$, or in case 
$\Char(\K)\geq 0$ and $q\in\K^{\ast},q\neq 1$ is a root of unity.

In this paper we like to consider an arbitrary 
field $\K$ and arbitrary $q\in\K^{\ast}$ 
and study particular pieces of the 
category $\Mod{\Uq}$ in more detail. 
To be more specific, we show how \textit{Jantzen's sum formula} 
can be used to deduce the semisimplicity of modules in $\Mod{\Uq}$.

As an application, we provide semisimplicity criteria 
for well-known algebras $\cal{A}$ arising in invariant theory, 
namely for \textit{Hecke algebras} $\cal{H}^{\Am}_d(q)$ and 
$\cal{H}^{\Bm}_d(q)$ \textit{of types} $\Am$ \textit{and} $\Bm$ 
(see Theorem~\ref{thm-heckesemi}), for the 
\textit{walled Brauer algebra} $\cal{B}_{r,s}(\delta)$ (see Theorem~\ref{thm-walledbrauer}) 
and for the \textit{Brauer algebra} $\cal{B}_{d}(\delta)$ (see Theorem~\ref{thm-brauer}). 
These examples are however just the tip of an iceberg: our 
approach should work to provide semisimplicity criteria for a 
big class of algebras (see also Remark~\ref{rem-generalization}). 
But in this paper we restrict to these example, and 
we obtain explicit necessary and 
sufficient conditions for the 
semisimplicity of these algebras $\cal{A}$ (over any field $\K$ and 
any $q\in\K^{\ast}$). For instance, 
when $\cal{A}=\cal{H}^{\Am}_d(q)$ is the Hecke 
algebra of the symmetric group $S_d$ in $d$ letters, we get:

\begin{thmm}
(\textbf{Semisimplicity criterion for the Hecke algebra of type $\Am$})
\newline
$\cal{H}^{\Am}_d(q)$ 
is semisimple if and only if one of the following conditions hold:
\begin{enumerate}
\item $\Char(\K)>d$ and $q=1$.
\item $\Char(\K)=0$ and $q=1$.
\item $q\in\K^{\ast},q\neq 1$ is a root of unity with $\ord(q^2)>d$.
\item $q\in\K^{\ast},q\neq 1$ is a non-root of unity.
\end{enumerate}
\end{thmm}

The Hecke algebra of type $\Am$ and its semisimplicity criterion stated 
above is a particular nice example of our general approach, since the corresponding 
combinatorics is very easy in this case.

To explain our methods in more detail, we 
consider the full, additive tensor subcategory 
$\T$ of $\Mod{\Uq}$ given by all $\Uq$\textit{-tilting modules} 
(a notion that we recall in Section~\ref{sec-prelim}). 
For any $\Uq$-tilting module $T\in\T$ we have,
as observed in~\cite[Theorems~4.11 and 5.13]{ast}, that
\begin{equation}\label{eq-semiintro}
\Endrr\text{ is semisimple if and only if }T\text{ is a semisimple }\Uq\text{-module.}
\end{equation}

Moreover, $T$ is 
a semisimple $\Uq$-module if and only if all Weyl 
modules $\Dl$ appearing in the Weyl filtration of 
$T$ are simple $\Uq$-modules, 
see Lemma~\ref{lem-cellsemisimple}. 
The important step here is now to 
use (a version of) 
Jantzen's sum formula for the Weyl modules $\Dl$, see Theorem~\ref{thm-jsum}, 
to translate the semisimplicity problem 
into a purely algorithmic problem in 
terms of roots, weights and the combinatorics of the (affine) Weyl group $W$:
\begin{equation}\label{eq-jantzenintro}
\Dl\text{ is simple if and only if Jantzen's sum formula of }\Dl\text{ vanishes.} 
\end{equation}

To state some explicit consequences, 
let us restrict ourselves to 
Lie algebras $\fg$ of type $\AM,\BM,\CM$ or $\DM$. We then have 
the quantum analogue $V\in\T$ of the 
\textit{vector representation}
of $\fg$ and its dual $V^*\in\T$ 
(which are 
isomorphic in types $\BM,\CM$ and $\DM$).

Let $n=\dim(V)$ and take the $\Uq$-module 
$T_{n}^{r,s}=V^{\otimes r}\otimes (V^*)^{\otimes s}$. 
Since $V\in\T$ and hence, $T_{n}^{r,s}$ 
is a tensor product of 
$\Uq$-tilting modules (except if $\Char(\K)=2$ in type $\BM$), 
it is itself a $\Uq$-tilting module, see Proposition~\ref{prop-prop}. 
Thus,~\eqref{eq-semiintro} and~\eqref{eq-jantzenintro} apply.

By (generalized versions of) \textit{Schur-Weyl duality}, 
see Section~\ref{sec-schur}, the above mentioned algebras $\cal{A}$ 
arise, for 
suitable choices of $\fg,n,r,s$, as endomorphism algebras 
of the form $\End_{\Uq}(T_{n}^{r,s})$. 
Hence, our method implies directly 
explicit semisimplicity criteria 
as long as $\cal{A}\cong\End_{\Uq}(T_{n}^{r,s})$.

It  
remains to deal with the cases where the 
algebras $\cal{A}$ 
do not appear as such endomorphism algebras. 
This could 
happen because of the following 
reasons:
\begin{itemize}
\item The natural map 
from $\cal{A}$ to $\End_{\Uq}(T_{n}^{r,s})$ is not 
injective 
(this happens in case $r+s$ is large compared to $n$)
or not surjective 
(this happens in case $\fg=\mathfrak{so}_{2m}$).
\item The algebra $\cal{A}$ does not 
appear as an algebra of the form $\End_{\Uq} (T_{n}^{r,s})$ at all 
(this happens for $\cal{B}_d(\delta)$ in case 
$\Char(\K)=0$ and $\delta\in\Z_{<0}$ is odd).
\end{itemize}
We note that, by quantum Schur-Weyl duality as in Theorems~\ref{thm-schur-weyl1} 
and~\ref{thm-schur-weylb},  
Hecke algebras of type $\Am$ or $\Bm$ can always 
be obtained as endomorphism algebras of 
some $\Uq$-tilting module.

To deal with these cases for the (walled) Brauer algebras, 
we first observe that passing to a 
field $\K$ with $\Char(\K)=p>2$ 
has several advantages (our approach 
is in fact easier in positive characteristic). 
First of all, the (walled) Brauer 
algebra for parameter $\delta$ equals 
the (walled) Brauer algebra for parameter 
$\delta\pm ap$ (for any $a\in\Z$) which 
allows us to pass from even values of $\delta$ to odd values of $\delta$.
Second, since 
under the corresponding Schur-Weyl duality $n$ depends on $\delta$, we can avoid 
that $r+s$ is large compared to $n$ by adding $p$ to $n$ often enough. 
Using both observations we can always achieve $\cal{A}\cong\End_{\Uq}(T_{n}^{r,s})$.
However, adding $p$ makes Jantzen's sum formula more involved. We therefore 
prefer to argue differently: in some `boundary cases' we can 
determine the kernel of the action of 
$\cal{A}$ on $\End_{\Uq}(T_{n}^{r,s})$ explicitly, see Section~\ref{sec-schuridem}, 
and deduce in this way the (non-)semisimplicity of $\cal{A}$ from the (non-)semisimplicity of 
$\End_{\Uq}(T_{n}^{r,s})$.

Finally it remains to 
treat the case $\Char(\K)=0$. 
We observe that 
the algebra $\cal{A}$ in question is 
semisimple if and only if it is semisimple over fields of 
large enough characteristic. 
One way to pass at least to the complex numbers 
is to use the theory of \textit{ultraproducts}, see for example~\cite[Chapter~2]{sch}, 
and realize the complex numbers as an \textit{ultralimit} of fields of positive characteristics.
Since the semisimplicity can 
be described by an integral polynomial 
expression (namely a determinant), the algebra $\cal{A}$ is semisimple over the complex numbers if and only if 
it is semisimple over fields of large enough characteristics.
Instead of the (way more powerful) theory 
of ultraproducts, we use the probably 
more common tool of \textit{trace forms} to pass from positive 
characteristic to characteristic zero, see~\ref{sec-app2}. 
Note that both arguments rely on the fact that our algebras 
$\cal{A}$ in question can be defined over $\Z$.

\begin{rem}\label{rem-generalization}
Our methods 
to deduce semisimplicity criteria 
work more generally and not just for the category $\Mod{\Uq}$. 
Our arguments in~\cite{ast} (which are the basis of 
the criterion from~\eqref{eq-semiintro}) do depend 
on the existence of weight spaces such that~\cite[Lemma~4.5]{ast} makes sense, 
and the semisimplicity 
criterion itself relies on the existence of Jantzen's sum formula, which 
also involves weight computations and is not available in general. 
But as long as these notions are available, our method works. 
As an explicit generalization: 
we could 
for instance work with category $\boldsymbol{\cal{O}}$, its tilting theory 
(see for 
example~\cite[Chapter~11]{hum}) and the corresponding 
Jantzen's sum formulas (see for 
example~\cite[Chapter~5, Section~3]{hum}). For brevity, we stay with $\Mod{\Uq}$
in this paper.
\end{rem}

\subsection{Outline of the paper}\label{sub-intropart2}

The paper is organized as follows.

\begin{itemize}

\item In Section~\ref{sec-prelim} we recall some 
facts about $\Uq$-tilting 
modules. Moreover, 
we recall the two main ingredients for our proofs of semisimplicity:
\begin{itemize}
\item The semisimplicity criterion of endomorphism algebras of 
$\Uq$-tilting modules.
\item Jantzen's sum formula which provides a method to check whether 
a given Weyl module $\Dl$ is a simple $\Uq$-module.
\end{itemize}

\item In Section~\ref{sec-schur} we list, for the convenience of the reader, 
some Schur-Weyl like dualities which we need
in a rather complete form. In Section~\ref{sec-schuridem} 
we additionally describe in some `boundary cases' the kernels
of the homomorphisms appearing in the Schur-Weyl like dualities. We 
need the explicit description in some of these cases for 
our proof, but the explicit descriptions 
are interesting in their own right.

\item In Sections~\ref{sec-hecke},~\ref{sec-walled} 
and~\ref{sec-brauer} we give the 
semisimplicity criteria for the Hecke algebras of types 
$\Am$ and $\Bm$, the walled Brauer algebras and the Brauer algebras.

\item In~\ref{sec-app2} we describe in detail some 
tools to compare semisimplicity in characteristic $p$ and in 
characteristic zero. Moreover, in~\ref{sec-app} 
we recall the root and weight data in types $\AM,\BM,\CM$ and $\DM$ 
that we use in this paper.

\end{itemize}

\begin{nota}\label{nota-field}
Throughout: we denote by $\K$ an arbitrary field, by 
$q$ any element 
in $\K^{\ast}$ and by $p\in\Zg$ a 
prime number (usually $p=\Char(\K)$).
We call the case of $\Char(\K)=0$ and $q=1$ 
the \textit{classical case}. We exclude 
the quasi-classical case $q=-1$ for technical reasons in 
case $\Char(\K)>2$ 
(the notion quasi-classical was coined in~\cite[Section~33.2]{lu}, 
where Lusztig also proves that, if $\Char(\K)=0$, 
then the $q=-1$ case is equivalent to the $q=1$ case).

Let $\ord(q^2)=\ell$ with $\ell\in\N$ 
be the order of $q^2$, that is, 
the smallest integer $\ell\in\N$ such that $q^{2\ell}=1$ (or $\ell=0$ 
if no such number exists). In case $q\neq 1$ and $\ell\neq 0$, we 
say that $q$ is a \textit{root of unity}. If $\ell=0$, then 
we call $q$ a \textit{non-root of unity}.

By an algebra $\cal{A}$ we 
always mean a unital, associative algebra over 
$\Z$ or $\K$.
All modules are finite-dimensional, left $\cal{A}$-modules throughout the paper. 
As usual in the case $\cal{A}=\Uq$, 
we consider only $\Uq$-modules of \textit{type} $1$ 
(see~\cite[Chapter~5, Section~2]{ja}).
\end{nota}

%% file: res/prelim.tex
We start by briefly recalling some notions from the theory 
of $\Uq(\fg)$-tilting modules. 
The reader unfamiliar with these is referred 
to~\cite{and},~\cite{ast},~\cite{astproofs},~\cite{jarag} 
or~\cite{saw} (and the references therein).

Here we denote by $\Uq(\fg)$ 
the \textit{quantized enveloping algebra} specialized at $q\in\K^{\ast}$ for a 
reductive Lie algebra $\fg$ with a fixed triangular decomposition 
$\fg=\fg^+\oplus\fg^0\oplus\fg^-$ attached to a choice of
\textit{positive roots} $\Phi^+\subset\Phi$ inside all roots $\Phi$.
Let 
$\Pi\subset\Phi^+$ be the set of \textit{simple roots}, $X$ the \textit{integral weight lattice}
and $X^+$ the set of \textit{dominant integral weights}.

For the main calculations 
in the paper it is enough to restrict ourselves to the classical Lie algebras
$\fg=\gll{m}$, $\fg=\soo{2m+1}$, 
$\fg=\spo{2m}$ or $\fg=\soo{2m}$ for some fixed $m\in\Zg$. We usually let $n$ 
denote the dimension of the corresponding (quantized) vector representation 
$V=\Delta_q(\omega_1)$ (that is $n=m$ for $\gll{m}$, 
$n=2m+1$ for $\fg=\soo{2m+1}$ and $n=2m$ for $\fg=\spo{2m}$ respectively $\fg=\soo{2m}$).
For convenience, we have listed in~\ref{sec-app} 
the for our purpose necessary explicit root and weight data in 
the Dynkin types $\AM,\BM,\CM$ and $\DM$ 
(together with some standard notations that we use throughout).
We study 
the category $\Mod{\Uq}$ of finite-dimensional
representations of $\Uq$ (of type $1$) in what follows.

\begin{rem}\label{rem-roots}
In the `generic' cases (e.g. $q=\pm 1,\Char(\K)=0$), 
$\Mod{\Uq}$ is semisimple and behaves combinatorially as $\Mod{\fg}$ 
for the corresponding Lie algebra $\fg$ over $\C$.  
(For non-roots of unity 
or $q=\pm 1,\Char(\K)=0$ see~\cite[Theorem~9.4]{apw}
and~\cite[Section~33.2]{lu} for $q=-1$.)
\end{rem}

The algebra $\Uq$ has a 
triangular decomposition $\Uq=\Uq^+\Uq^0\Uq^-$.
This gives for each $\lambda\in X^+$ 
a \textit{Weyl} $\Uq$\textit{-module} $\Dl$ and a \textit{dual Weyl} 
$\Uq$\textit{-module} $\Nl$.
The $\Uq$-module $\Dl$ has a unique simple head $\Ll$ 
which is the unique simple socle of $\Nl$.
Let $\mathrm{ch}(M)$ denote the \textit{(formal) 
character of} $M\in\Mod{\Uq}$, that is,
\[
\mathrm{ch}(M)=\sum_{\lambda\in X}(\dim(M_{\lambda}))e^{\lambda}\in\Z[X],
\]
where $M_{\lambda}=\{m\in M\mid um=\lambda(u)m, u\in\Uq^0\}$ is the 
$\lambda$-weight space of $M$ (here we regard $\lambda$ as a 
character of $\Uq^0$), 
and $\Z[X]$ is the group algebra of the additive group $X$.

The following is crucial in the non-semisimple cases:

\begin{prop}\label{prop-classchar}
The characters $\mathrm{ch}(\Dl)$ and $\mathrm{ch}(\Nl)$ are independent of 
$\Char(\K)$ and of $q\in\K^{\ast}$. In particular, 
they are given as in the classical case.
\end{prop}

\begin{proof}
The statement follows directly 
from the definitions and the $q$-version of Kempf's vanishing theorem 
which can be found in~\cite[Theorem~5.5]{rh2}.
\end{proof}

We say $M\in\Mod{\Uq}$ has a $\Delta_q$\textit{-filtration} 
if there exists $i\in\N$ and 
a descending sequence
\[
M=M_{0}\supset M_1\supset\cdots\supset M_{i^{\prime}}\supset\dots\supset M_{i-1}\supset M_i=0,
\]
such that for all $i^{\prime}=0,\dots,i-1$ 
we have $M_{i^{\prime}}\in\Mod{\Uq}$, 
$M_{i^{\prime}}/M_{i^{\prime}+1}\cong \Delta_q(\lambda_{i^{\prime}})$ 
with $\lambda_{i^{\prime}}\in X^+$.

A $\nabla_q$\textit{-filtration} is defined similarly, 
but using $\nabla_q(\lambda)$ instead of $\Delta_q(\lambda)$ and 
an ascending sequence of $\Uq$-submodules, that is,
\[
0=M_{0}\subset M_1\subset\cdots\subset M_{i^{\prime}}\subset\dots\subset M_{i-1}\subset M_i=M,
\]
such that for all $i^{\prime}=0,\dots,i-1$ we have 
$M_{i^{\prime}+1}/M_{i^{\prime}}\cong \nabla_q(\lambda_{i^{\prime}})$ 
with $\lambda_{i^{\prime}}\in X^+$.

A $\Uq$\textit{-tilting module} is a $\Uq$-module $T\in\Mod{\Uq}$ 
which has both, a $\Delta_q$- 
and a $\nabla_q$-filtration. 
These filtrations are unique up to reordering of factors, 
(this can be verified using standard arguments, 
see for example~\cite[Proposition~A2.2]{don1} or~\cite[4.16~Remark~(4)]{jarag}) 
and we henceforth 
call the appearing factors \textit{Weyl} or \textit{dual Weyl factors} of $T$ respectively.

The \textit{category} $\T$ \textit{of} $\Uq$\textit{-tilting modules} is the full 
subcategory $\T\subset\Mod{\Uq}$ with objects consisting of all $\Uq$-tilting modules.
The category $\T$ is an additive Krull-Schmidt category, closed under direct sums, duality 
and finite tensor products. The latter is in general non-trivial to prove. 
Apart from type $\BM$, the following has an elementary proof.

\begin{prop}\label{prop-prop}
Let $\Uq=\Uq(\fg)$ with $\fg$ of 
type $\AM$, $\BM$, $\CM$ or $\DM$. Then the vector 
representation $V$ of $\Uq$ is a 
$\Uq$-tilting module\footnote{Here we need that $\Char(\K)\neq 2$ in type $\BM$ and we assume this throughout if we work in this type.}. Moreover, 
$T^d_n=V^{\otimes d}\in\T$ 
is a 
$\Uq$-tilting module for all $d\in\Z_{\geq 0}$ as well. The dimension 
$\dim(\End_{\Uq}(T_n^d))$ only depends on $\fg$ and $d$.
\end{prop}

\begin{proof}
For the types $\AM$, $\CM$ and $\DM$, 
see~\cite[Propositions~3.10]{astproofs} for an elementary proof.
In type $\BM$ it was observed in~\cite[Page~20]{ja1} that $V$ 
is a $\Uq$-tilting module as long as $\Char(\K)\neq 2$. By~\cite[Theorem~3.3]{par}, 
it follows that $V^{\otimes d}\in\T$. 
To see that $\dim(\End_{\Uq}(T_n^d))$ only depends on $\fg$ and $d$
first note that $\dim(\End_{\Uq}(T_n^d))=\sum_{\lambda\in X^+}(T:\Dl)^2$ 
(which can be derived from the Ext-vanishing, see for example~\cite[Theorem~3.1]{ast}).
Now use the fact that $\mathrm{ch}(\Dl)$ is as in the classical case 
which implies the statement.
\end{proof}

\begin{lem}\label{lem-cellsemisimple} 
A $\Uq$-tilting module $T\in\T$ is a semisimple $\Uq$-module 
if and only if all Weyl factors $\Dl$ of $T$ are simple $\Uq$-modules 
if and only if all dual Weyl factors $\Nl$ of $T$ are simple $\Uq$-modules.
\end{lem}

\begin{proof}
The second equivalence is evident. If $T$ is semisimple, then clearly 
all Weyl factors $\Dl$ of $T$ are simple $\Uq$-modules. 
If all Weyl factors are simple, hence $\Dl\cong\Nl$, 
then the statement follows
by using Ext-vanishing (see 
for example~\cite[Theorem~3.1]{ast}) and 
induction on the length of a $\Delta_q$-filtration 
of $T$.
\end{proof}

\begin{thm}(\textbf{Semisimplicity criterion for $\Endrr$})\label{thm-cellsemisimple}
Let $T\in\T$ be a $\Uq$-tilting module. 
Then the algebra $\Endrr$ is semisimple if and only if $T$ is a semisimple $\Uq$-module.
\end{thm}

\begin{proof}
This is a consequence of~\cite[Theorems~4.11 and 5.13]{ast}.
\end{proof}

Thus, by Lemma~\ref{lem-cellsemisimple} and Theorem~\ref{thm-cellsemisimple}, 
the question whether $\Endrr$ is semisimple is equivalent to the question 
whether all (dual) Weyl factors of $T$ are simple $\Uq$-modules:

\begin{cor}\label{cor-cellsemisimple}
The algebra $\Endrr$ is semisimple if and only if all 
Weyl factors $\Dl$ of $T$ are simple $\Uq$-modules 
if and only if all dual Weyl factors $\Nl$ of $T$ are simple $\Uq$-modules.
\end{cor}

A method to check if a given Weyl module is a simple $\Uq$-module is 
provided by \textit{Jantzen's sum formula}. 
In order to state it, we need some preparations. First, 
for any $a\in\N$ and any $p$, we denote by $v_p(a)$ 
its $p$\textit{-adic valuation}, that is the largest 
non-negative integer such that $p^{v_p(a)}$ divides $a$. 
Second, let $W$ be the \textit{Weyl group} 
associated to $\fg$ (recall 
that $W$ is generated by the simple reflections 
$s_i=s_{\alpha_i}$ for each simple root $\alpha_i\in \Pi$) 
and let $l(w)$ 
denote the length of an element $w\in W$. The Weyl group $W$ 
acts on $X$ in two ways:
\begin{gather*}
\begin{aligned}
s_i(\lambda)=\lambda-\langle\lambda,\alpha_i^{\vee}\rangle\alpha_i,\text{ for }\lambda\in X,
\quad\text{respectively}\quad
s_i.\lambda=s_i(\lambda+\rho)-\rho,\text{ for }\lambda\in X.
\end{aligned}
\end{gather*}
Here we use the notation $\rho$ as in~\ref{sec-app}.

\begin{defn}\label{defn-singular}
Let $\lambda\in X^+,\mu\in X$ and assume $w.\lambda=\mu$ 
for some $w\in W$. Then we set
\begin{equation}\label{eq-cancel}
\chi(\lambda)=\mathrm{ch}(\Dl)\quad\text{and}\quad 
\chi(\mu)=(-1)^{l(w)}\chi(\lambda).
\end{equation} 
In particular, $\chi(\lambda)=0$ 
for all \textit{dot-singular} $\Uq$-weights $\lambda\in X$ (a $\Uq$-weight $\lambda$
is dot-singular if there 
exists $\alpha\in\Phi$ such that 
$\langle\lambda+\rho,\alpha^{\vee}\rangle=0$). 
On the other hand, any 
\textit{dot-regular} 
(that is, non-dot-singular) $\mu\in X$ is of the form $w.\lambda$ 
for some unique $\lambda\in X^+$, which makes 
the assignments well-defined.
\end{defn}

\begin{nota}\label{nota-rho}
What we call dot-singular is  
often called singular in the literature. 
In contrast, we call a $\Uq$-weight $\lambda\in X^+$ \textit{singular}, if there 
exists $\alpha\in\Phi$ with 
$\langle\lambda,\alpha^{\vee}\rangle=0$.
Similarly for regular $\Uq$-weights. 
We recall 
for the root systems of classical type equivalent criteria 
for being (dot-)singular 
(which we use for calculations) in~\ref{sec-app}.
\end{nota}

Formulas~\eqref{eq-jsuma},~\eqref{eq-jsumb} and~\eqref{eq-jsumc} 
are called \textit{Jantzen's sum formulas} because they originate 
in Jantzen's work~\cite{ja2}. We tend to write \textbf{JSF} to 
abbreviate `Jantzen's sum formula'.

\begin{thm}(\textbf{Jantzen's sum formula})\label{thm-jsum}
Let $\lambda\in X^+$. Then $\Dl$ has a filtration
\[
\Dl=\Delta_q^0(\lambda)\supset\Delta_q^1(\lambda)\supset\Delta_q^2(\lambda)\supset\dots,
\]
called \textit{Jantzen filtration}, such that 
all $\Delta_q^{k^{\prime}}(\lambda)\in\Mod{\Uq}$, 
$\Dl/\Delta_q^1(\lambda)\cong\Ll$ and: 
\begin{itemize}
\item If $\Char(\K)=0$ and $q=1$ or $q\in\K^{\ast}$ is a non-root of unity, 
then $\Delta_q^1(\lambda)=0$.
\item If $\Char(\K)=0$ and $q\in\K^{\ast}$ is a root of unity with $\ord(q^2)=\ell$, then
\begin{equation}\label{eq-jsuma}
\sum_{k^{\prime}>0}\mathrm{ch}(\Delta_q^{k^{\prime}}(\lambda))=-\sum_{\alpha\in\Phi^+}
\sum_{0<k\ell \atop <\langle\lambda+\rho,\alpha^{\vee}\rangle}\chi(\lambda-k\ell\alpha).
\end{equation}
\item If $\Char(\K)=p>0$ and $q\in\K^{\ast}$ is a root of unity with $\ord(q^2)=\ell$, then
\begin{equation}\label{eq-jsumb}
\sum_{k^{\prime}>0}\mathrm{ch}(\Delta_q^{k^{\prime}}(\lambda))=-\sum_{\alpha\in\Phi^+}
\sum_{0<k\ell \atop <\langle\lambda+\rho,\alpha^{\vee}\rangle}p^{v_p(k)}\chi(\lambda-k\ell\alpha).
\end{equation}
\item If $\Char(\K)=p>0$ and $q=1$, then
\begin{equation}\label{eq-jsumc}
\sum_{k^{\prime}>0}\mathrm{ch}(\Delta_{1}^{k^{\prime}}(\lambda))=-\sum_{\alpha\in\Phi^+}
\sum_{0<kp \atop <\langle\lambda+\rho,\alpha^{\vee}\rangle}v_p(kp)\chi(\lambda-kp\alpha).
\end{equation}
\end{itemize}
(The right-hand sums run over all 
$k\in\N$ such that the indicated inequalities hold.) 
In particular, $\Dl$ is a simple 
$\Uq$-module if and only if the corresponding \textbf{JSF} is zero. 
\end{thm}

\begin{proof}
See~\cite[Theorem~6.3]{ak},~\cite[Proposition~II.8.19]{jarag} 
respectively~\cite[Theorem~5.1]{thams}.
\end{proof}

We first show in an example how Theorem~\ref{thm-jsum} together with 
Corollary~\ref{cor-cellsemisimple} 
can be used in practice to determine whether $\Endrr$ 
is semisimple.

\begin{ex}\label{ex-JSF}
Consider $\Uo=\Uo(\gll{5})$ over $\K$ with $\Char(\K)=5$. 
Let $n=m=5$, $d=3$ and $V=\Delta_1(\omega_1)\cong\Delta_1(\varepsilon_1)$ be the vector 
representation of $\Uo$ and set $T_n^d=V^{\otimes d}$. We want to check whether $\End_{\Uo}(T_n^d)$ 
is a semisimple algebra. Since $V\in\T$, so is $T_n^d$ by Proposition~\ref{prop-prop}. By 
Corollary~\ref{cor-cellsemisimple}, it remains to check whether $T_n^d$ has  
only Weyl factors which are simple $\Uo$-modules. 
We see (using Proposition~\ref{prop-classchar}) that 
the Weyl factors of $T_n^d$ have highest weights
\begin{gather*}
\lambda=3\varepsilon_1=(3,0,0,0,0),\quad\quad
\mu=2\varepsilon_1+\varepsilon_2=(2,1,0,0,0),\\
\nu=\varepsilon_1+\varepsilon_2+\varepsilon_3=(1,1,1,0,0).
\end{gather*}
In order to see whether these Weyl factors are simple $\Uo$-modules, 
we use \textbf{JSF} from~\eqref{eq-jsumc}.
We have $\rho=(4,3,2,1,0)$ (see~\ref{sec-app}) and thus (for all $\alpha\in\Phi^+$)
\begin{gather*}
\langle\lambda+\rho,\alpha^{\vee}\rangle\leq\langle\lambda+\rho,(\varepsilon_1-\varepsilon_5)^{\vee}\rangle=7,\quad\quad
\langle\mu+\rho,\alpha^{\vee}\rangle\leq\langle\lambda+\rho,(\varepsilon_1-\varepsilon_5)^{\vee}\rangle=6,\\
\langle\nu+\rho,\alpha^{\vee}\rangle\leq\langle\lambda+\rho,(\varepsilon_1-\varepsilon_5)^{\vee}\rangle=5=\Char(\K).
\end{gather*}
\textbf{JSF} for $\nu$ is zero: it totally collapses since the second sum 
on the right-hand side of~\eqref{eq-jsumc} 
is empty.
Hence, $\Delta_1(\nu)$ is a simple 
$\Uo$-module. For $\mu$ we see that the only possible 
contribution to \textbf{JSF} comes from 
the positive root $\alpha\in\Phi^+$ of the form $\alpha=\varepsilon_1-\varepsilon_5$. 
But
\[
\mu+\rho-5(\varepsilon_1-\varepsilon_5)=(\colorbox{mycolor}{$1$},4,2,\colorbox{mycolor}{$1$},5).
\]
Hence, $\mu+\rho-5(\varepsilon_1-\varepsilon_5)$ is a singular 
$\Uo$-weight\footnote{We illustrate with green boxes the entries which make a $\Uq$-weight 
singular. Moreover, we illustrate with red boxes (and white numbers) 
the numbers relevant for the calculation in the regular cases.}.
Thus, $\chi(\mu-5(\varepsilon_1-\varepsilon_5))=0$ and so 
\textbf{JSF} is zero which again implies 
that $\Delta_1(\mu)$ is a simple 
$\Uo$-module.

For $\lambda$ the only possible 
contributions can come from the positive roots $\alpha\in\Phi^+$ of the form 
$\alpha=\varepsilon_1-\varepsilon_5$ or $\alpha=\varepsilon_1-\varepsilon_4$. We calculate
\[
\lambda+\rho-5(\varepsilon_1-\varepsilon_5)=(\colorbox{mycolor}{$2$},3,\colorbox{mycolor}{$2$},1,5),
\quad\quad
\lambda+\rho-5(\varepsilon_1-\varepsilon_4)=(\colorbox{mycolor}{$2$},3,\colorbox{mycolor}{$2$},6,0).
\]
Hence, $\Delta_1(\lambda)$ is again a simple 
$\Uo$-module which shows that all Weyl 
factors of $T_n^d$ are simple $\Uo$-modules. Thus, $\End_{\Uo}(T_n^d)$ is a 
semisimple algebra under the assumption $\Char(\K)=5$.

The situation changes if $\Char(\K)=3$: in contrast 
to the above there can now possibly be contributions to \textbf{JSF} 
of $\Delta_1(\lambda)$ for the four 
positive roots $\alpha\in\Phi^+$ given by $\alpha=\varepsilon_1-\varepsilon_5$ (for $k=1,2$), 
$\alpha=\varepsilon_1-\varepsilon_4$, $\alpha=\varepsilon_1-\varepsilon_3$ or $\alpha=\varepsilon_1-\varepsilon_2$. 
We calculate
\begin{gather*}
\begin{aligned}
\lambda+\rho-3(\varepsilon_1-\varepsilon_5)=(4,\colorbox{mycolor}{$3$},2,1,\colorbox{mycolor}{$3$}),\quad&\quad
\lambda+\rho-3(\varepsilon_1-\varepsilon_3)=(\colorbox{colormy}{\color{white}$4$},3,\colorbox{colormy}{\color{white}$5$},1,0),\\
\lambda+\rho-6(\varepsilon_1-\varepsilon_5)=(\colorbox{mycolor}{$1$},3,2,\colorbox{mycolor}{$1$},6),
\quad&\quad
\lambda+\rho-3(\varepsilon_1-\varepsilon_2)=(\colorbox{colormy}{\color{white}$4$},\colorbox{colormy}{\color{white}$6$},2,1,0),\\
\lambda+\rho-3(\varepsilon_1-\varepsilon_4)=(\colorbox{mycolor}{$4$},3,2,\colorbox{mycolor}{$4$},0).\quad&
\end{aligned}
\end{gather*}
Thus, \textbf{JSF} of $\Delta_1(\lambda)$ is non-zero. 
Hence, $\Delta_1(\lambda)$ provides a 
Weyl factor of $T_n^d$ which is a non-simple $\Uo$-module,
showing that $\End_{\Uo}(T_n^d)$ is not a 
semisimple algebra anymore.
\end{ex}

\begin{rem}\label{rem-nocan}
We deduced in Example~\ref{ex-JSF} that $\End_{\Uo}(T_n^d)$ is 
non-semisimple from 
the appearance of non-zero summands in \textbf{JSF}.
We did not pay attention to possible cancellations which could 
occur because of the sign in~\eqref{eq-cancel}. This is 
however justified:
in type $\AM$ such cancellations can not occur, 
see~\cite[Section~7.4]{ak}. 
Thus, in type $\AM$ it suffices 
to give one non-zero summand 
to conclude that the corresponding 
\textbf{JSF} is non-zero.
\end{rem}

We illustrate now non-trivial cancellations.

\begin{ex}\label{ex-JSF2}
Let 
$\Char(\K)=2$ and $\Uo=\Uo(\mathfrak{sp}_{6})$. Consider the 
$\Uo$-module $\Delta_1(\lambda)$ with 
$\lambda=\varepsilon_1+\varepsilon_2=(1,1,0)$. We claim that $\Delta_1(\lambda)$ 
is a simple $\Uo$-module. First note that we have $\rho=(3,2,1)$. 
As in Example~\ref{ex-JSF}, we see that the only positive roots $\alpha\in\Phi^+$ 
that could contribute to \textbf{JSF} from~\eqref{eq-jsumc} are
\begin{gather*}
2\varepsilon_1,\quad\quad2\varepsilon_2,\quad\quad\varepsilon_1-\varepsilon_3,\quad\quad\varepsilon_1+\varepsilon_2,\quad\quad\varepsilon_1+\varepsilon_3,\quad\quad\varepsilon_2+\varepsilon_3.
\end{gather*}
We leave it to the reader to verify that $\alpha=2\varepsilon_1$, $\alpha=2\varepsilon_2$, 
$\alpha=\varepsilon_1-\varepsilon_3$ and 
$\alpha=\varepsilon_2+\varepsilon_3$ do not contribute to \textbf{JSF} of $\Delta_1(\lambda)$. For 
the others we get 
$\langle\lambda+\rho,(\varepsilon_1+\varepsilon_2)^{\vee}\rangle=7$ 
and $\langle\lambda+\rho,(\varepsilon_1+\varepsilon_3)^{\vee}\rangle=5$. Thus, we have to deal with 
$k=1,2,3$ or $k=1,2$ in \textbf{JSF} of $\Delta_1(\lambda)$:
\begin{gather*}
\begin{aligned}
&\lambda+\rho-2(\varepsilon_1+\varepsilon_2)=(2,\colorbox{mycolor}{$1$},\colorbox{mycolor}{$1$}),\quad\quad&\lambda+\rho-2(\varepsilon_1+\varepsilon_3)=(\colorbox{colormy}{\color{white}$2$},3,\colorbox{colormy}{\color{white}$-1$}),\\
&\lambda+\rho-4(\varepsilon_1+\varepsilon_2)=(0,\colorbox{mycolor}{$-1$},\colorbox{mycolor}{$1$}),\quad\quad&\lambda+\rho-4(\varepsilon_1+\varepsilon_3)=(0,\colorbox{mycolor}{$3$},\colorbox{mycolor}{$-3$}),\\
&\lambda+\rho-6(\varepsilon_1+\varepsilon_2)=(\colorbox{colormy}{\color{white}$-2$},\colorbox{colormy}{\color{white}$-3$},1).
\end{aligned}
\end{gather*}
Permuting the two remaining regular $\Uo$-weights 
into dominant $\Uo$-weights gives different signs. Thus, 
the contributions cancel in \textbf{JSF} of 
$\Delta_1(\lambda)$ by Definition~\ref{defn-singular}. Hence,
\textbf{JSF} of $\Delta_1(\lambda)$ is zero (although not all 
summands are zero). Thus, 
$\Delta_1(\lambda)$ is a simple $\Uo$-module.
\end{ex}

%% file: res/schur.tex
In this section we recall a few known examples of Schur-Weyl like dualities.

\begin{nota}\label{nota-young}
Let
$\Lambda^+(d)=\{\lambda\in\Z_{\geq 0}^d\mid \lambda_1\geq\cdots\geq\lambda_d\geq 0, \sum_{i=1}^d\lambda_i=d\}$ 
denote the set of all 
\textit{partitions of} some $d\in\Zg$. We identify these with 
\textit{Young diagrams with} $d$ \textit{nodes}:
\[
\lambda=(\lambda_1,\dots,\lambda_d)\in\Lambda^+(d)\quad\leftrightsquigarrow\quad
\xy
(0,0)*{
\begin{tikzpicture} [scale=.33]
\draw[very thick] (0,0) rectangle (1,1);
\draw[very thick] (1,0) rectangle (2,1);
\draw[very thick] (2,0) rectangle (3,1);
\draw[very thick] (3,0) rectangle (4,1);
\draw[very thick] (0,-1) rectangle (1,0);
\draw[very thick] (1,-1) rectangle (2,0);
\draw[very thick] (2,-1) rectangle (3,0);
\draw[very thick] (0,-3.5) rectangle (1,-2.5);
\draw[very thick] (1,-3.5) rectangle (2,-2.5);
\node at (1,-1.5) {$\vdots$};
\node at (5,.5) {\tiny $\lambda_1$};
\node at (5,-.5) {\tiny $\lambda_2$};
\node at (5,-1.5) {$\vdots$};
\node at (5,-3) {\tiny $\lambda_d$};
\end{tikzpicture}
}
\endxy
\]
We always use the English notation for our Young diagrams, that is 
starting with $\lambda_1$ nodes in the top row.
Using the notation from~\ref{sec-app}, we can associate to each Young diagram of the form
$\lambda=(\lambda_1,\dots,\lambda_d)\in\Lambda^+(d)$ a $\Uq(\gll{m})$-weight 
$\lambda=\lambda_1\varepsilon_1+\dots+\lambda_d\varepsilon_d\in X^+$ and hence, 
a Weyl module $\Dl$ for $\Uq(\gll{m})$ (analogously for $\Uq(\fg)$ with
$\fg$ of types $\BM$, $\CM$ and $\DM$). 
Here, by convention, $\Dl=0$ if $\lambda_{m+1}>0$.

Similarly, given a 
\textit{pair} of Young diagrams $(\lambda,\mu)\in\Lambda^+(d_1)\times\Lambda^+(d_2)$,
then we can associate to it a dominant $\Uq(\gll{m}\oplus\gll{m})$-weight in 
the evident way and therefore a Weyl module for $\Uq(\gll{m}\oplus\gll{m})$ 
that we denote by 
$\Delta_q(\lambda,\mu)$.
Again, $\Delta_q(\lambda,\mu)=0$ if $\lambda_{m+1}>0$ or $\mu_{m+1}>0$.

We can also associate to such a pair 
$(\lambda,\mu)\in\Lambda^+(r)\times\Lambda^+(s)$ 
a $\Uq(\gll{2m})$-weight $(\lambda,\overline{\mu})$ via
\[
(\lambda,\overline{\mu})=\lambda_1\varepsilon_1+\dots+\lambda_r\varepsilon_{r}-
\mu_s\varepsilon_{r+1}-\dots-\mu_1\varepsilon_{r+s}
\]
and thus, a Weyl module $\Delta_q(\lambda,\overline{\mu})$ for $\Uq(\gll{2m})$.
Again, $\Delta_q(\lambda,\overline{\mu})=0$ if $(\lambda,\overline{\mu})_{2m+1}>0$.
\end{nota}

\subsection{Type \texorpdfstring{$\Am$}{A}: the Hecke algebra of type \texorpdfstring{$\Am$}{A} and (quantum) Schur-Weyl duality}

We fix $q\in\K^{\ast}$ in what follows. We denote by $S_d$ the symmetric group in 
$d\in\Zg$ letters.

\begin{defn}\label{defn-iwahorihecke}
Let $d\in\Zg$.
The \textit{Hecke algebra} $\cal{H}^{\Am}_d(q)$ \textit{of type} $\Am$ is the $\K$-algebra
generated by $\{H_{s_i}=H_i\mid s_i\in S_d\}$ for all transpositions $s_i=(i,i+1)\in S_d$ subject to the 
relations
\begin{equation*}
\begin{gathered}
H_i^2=(q-q^{-1})H_i+1,\quad\text{for }i=1,\dots,d-1,
\\
H_iH_j=H_jH_i,\quad\text{for }|i-j|>1,\quad\quad
H_iH_jH_i=H_jH_iH_j,\quad\text{for }|i-j|=1.
\end{gathered}
\end{equation*}
The \textit{group algebra of the symmetric group} is $\K[S_d]=\cal{H}^{\Am}_d(1)$.
\end{defn}

\begin{rem}\label{rem-iwahorihecke}
One can think of the generators $H_i$ of $\cal{H}^{\Am}_d(q)$ 
diagrammatically as 
crossings because there is a surjection
from the group algebra $\K[B_d]$ of the braid group $B_d$ in $d$ strands to $\cal{H}^{\Am}_d(q)$ 
given by sending the braid group generator $b_i$
(with strand $i$ crossing over strand $i+1$) to $H_i$. For example, the 
first relation from Definition~\eqref{defn-iwahorihecke} then reads as\footnote{We read all 
diagrams in this paper from left to right and bottom to top.}
\[
\xy
(0,0)*{
\begin{tikzpicture}[scale=.3]
	\draw [very thick, ->] (-5,-1) to (-5,3);
	\draw [very thick, ->] (-3,-1) to (-3,3);
	\draw [very thick, ->] (-1,-1) to (1,1);
	\draw [very thick, ->] (-0.25,0.25) to (-1,1);
	\draw [very thick] (0.25,-0.25) to (1,-1);
	\draw [very thick, ->] (3,-1) to (3,3);
	\draw [very thick, ->] (5,-1) to (5,3);
	\draw [very thick, ->] (-1,1) to (1,3);
	\draw [very thick, ->] (-0.25,2.25) to (-1,3);
	\draw [very thick] (0.25,1.75) to (1,1);
	\node at (-3.9,1) {\tiny $\cdots$};
	\node at (4.1,1) {\tiny $\cdots$};
\end{tikzpicture}
};
\endxy
\;\;=\;\;
(q-q^{-1})\cdot
\xy
(0,0)*{
\begin{tikzpicture}[scale=.3]
	\draw [very thick, ->] (-5,-1) to (-5,1);
	\draw [very thick, ->] (-3,-1) to (-3,1);
	\draw [very thick, ->] (-1,-1) to (1,1);
	\draw [very thick, ->] (-0.25,0.25) to (-1,1);
	\draw [very thick] (0.25,-0.25) to (1,-1);
	\draw [very thick, ->] (3,-1) to (3,1);
	\draw [very thick, ->] (5,-1) to (5,1);
	\node at (-3.9,0) {\tiny $\cdots$};
	\node at (4.1,0) {\tiny $\cdots$};
\end{tikzpicture}
};
\endxy\;\;
+\;\;
\xy
(0,0)*{
\begin{tikzpicture}[scale=.3]
	\draw [very thick, ->] (-5,-1) to (-5,1);
	\draw [very thick, ->] (-3,-1) to (-3,1);
	\draw [very thick, ->] (-1,-1) to (-1,1);
	\draw [very thick, ->] (1,-1) to (1,1);
	\draw [very thick, ->] (3,-1) to (3,1);
	\draw [very thick, ->] (5,-1) to (5,1);
	\node at (-3.9,0) {\tiny $\cdots$};
	\node at (4.1,0) {\tiny $\cdots$};
\end{tikzpicture}
};
\endxy
\]
Similarly, the algebra $\K[S_d]$ can then be thought 
of as the quotient of $\K[B_d]$ given by forgetting the information 
of over- and undercrossings and working 
with permutation diagrams.
\end{rem}

Let $\Char(\K)=0$ and $q=1$, or let 
$q\in\K^{\ast}$ be a non-root of unity. Then there are 
simple $\cal{H}^{\Am}_d(q)$-modules $D_q^{\lambda}$ for each $\lambda\in\Lambda^+(d)$, 
see for example~\cite[Chapter~3, Section~4]{mathas}. These are lifts of the classical 
Specht modules of $S_d$ to its Hecke algebra $\cal{H}^{\Am}_d(q)$.

Let $V=\Delta_q(\omega_1)$ denote the $n=m$ dimensional vector representation 
of $\Uq=\Uq(\gll{m})$. There is an 
action of $\cal{H}^{\Am}_d(q)$ on $T_{n}^{d}=V^{\otimes d}$ by so-called $R$\textit{-matrices}, 
see for example~\cite[(1.1)]{dps}.

\begin{thm}(\textbf{(Quantum) Schur-Weyl duality, type $\Am$})\label{thm-schur-weyl1}
\begin{itemize}
\item[(a)] The actions of $\Uq$ and $\cal{H}^{\Am}_d(q)$ on $T_{n}^{d}$ commute.
\item[(b)] Let $\Phi^{\Am}_{q\mathrm{SW}}$ be the algebra homomorphism induced by the action 
of $\cal{H}^{\Am}_d(q)$ on $T_{n}^{d}$. Then
\begin{equation*}
\Phi^{\Am}_{q\mathrm{SW}}\colon\cal{H}^{\Am}_d(q)\twoheadrightarrow
\End_{\Uq}(T_{n}^{d}),\quad\text{and}\quad\Phi^{\Am}_{q\mathrm{SW}}\colon
\cal{H}^{\Am}_d(q)\xrightarrow{\cong}\End_{\Uq}(T_{n}^{d}),\text{ if }n\geq d.
\end{equation*}
\item[(c)] Let $\Char(\K)=0$ and $q=1$, or let 
$q\in\K^{\ast}$ be a non-root of unity. Then there is 
a $(\Uq,\cal{H}^{\Am}_d(q))$-bimodule decomposition
\begin{equation*}
T_{n}^{d}\cong\bigoplus_{\lambda\in\Lambda^+(d)}\Dl\otimes D_q^{\lambda},
\end{equation*}
with simple $\Uq$-modules $\Dl\cong\Ll$.
\end{itemize}
\end{thm}

\begin{proof}
Part (a) and surjectivity in (b) are proven in~\cite[Theorem~6.3]{dps}. 
Injectivity in (b) follows from Proposition~\ref{prop-prop}. 
The statement (c) is the known 
$q$-analogue to the 
classical Schur-Weyl duality, 
see for example~\cite[Theorem~9.1.2]{gw} for the classical case 
(which holds almost word-by-word in the semisimple, quantized case as well).
\end{proof}

\subsection{Type \texorpdfstring{$\Am\oplus\Am$}{A+A}: the Hecke algebra of type \texorpdfstring{$\Bm$}{B} and (quantum) Schur-Weyl duality}

We fix again $q\in\K^{\ast}$ in what follows. If $q=1$, then we assume $p\not=2$.

\begin{defn}\label{defn-iwahoriheckeb}
Let $d\in\Zg$.
The (one-parameter) \textit{Hecke algebra} $\cal{H}^{\Bm}_d(q)$ \textit{of type} $\Bm$ is 
the $\K$-algebra
generated by $\{H_i\mid s_i\in S_d\}\cup\{H_0\}$ subject to the 
relations
\begin{gather*}
H_i^2=(q-q^{-1})H_i  + 1,\quad\text{for }i=0,\dots,d-1,\quad\quad
H_iH_j=H_jH_i,\quad\text{for }|i-j|>1,\\
H_iH_jH_i=H_jH_iH_j,\quad\text{for }|i-j|=1,\;i,j\neq 0,\quad\quad H_0H_1H_0H_1=H_1H_0H_1H_0.
\end{gather*}
The \textit{group algebra of the type} $\Bmm$ \textit{Weyl group} 
is $\K[S_d\ltimes(\Z/2\Z)^d]=\cal{H}^{\Bm}_d(1)$.
\end{defn}

Let $\Char(\K)=0$ and $q=1$, or let 
$q\in\K^{\ast}$ be a non-root of unity. For each 
pair of Young diagrams $(\lambda,\mu)\in\Lambda^+(d_1)\times\Lambda^+(d_2)$ 
define $D_q^{\lambda,\mu}$ via induction, see for 
example~\cite[Section~2.6]{ms} (which works in 
the semisimple, quantized case as well). That is,
\[
D_q^{\lambda,\mu}=\cal{H}^{\Bm}_d(q)
\otimes_{\cal{H}^{\Am}_{d_1}(q)\otimes\cal{H}^{\Am}_{d_2}(q)}
(D_q^{\lambda}\otimes D_q^{\mu}).
\]
Here $D_q^{\lambda}$ and $D_q^{\mu}$ are the 
quantum Specht modules of type $\Am$.

Take $\fg=\gll{m}\oplus\gll{m}$ and set $\Uq=\Uq(\fg)$. 
Let $V=\Delta_q(\omega_1)$ denote the $n=2m$ dimensional vector representation 
of $\Uq(\gll{2m})$ restricted to $\Uq$. 
There is an 
action of $\cal{H}^{\Bm}_d(q)$ on $T_{n}^{d}=V^{\otimes d}$, 
see for example~\cite[Section~1]{hu2}.

\begin{thm}(\textbf{(Quantum) Schur-Weyl duality, type $\Am\oplus\Am$})\label{thm-schur-weylb}
\begin{itemize}
\item[(a)] The actions of $\Uq$ and $\cal{H}^{\Bm}_d(q)$ on $T_{n}^{d}$ commute.
\item[(b)] Let $\Phi^{\Bm}_{q\mathrm{SW}}$ be the algebra homomorphism induced by the action 
of $\cal{H}^{\Bm}_d(q)$ on $T_{n}^{d}$. Then
\begin{equation*}
\Phi^{\Bm}_{q\mathrm{SW}}\colon\cal{H}^{\Bm}_d(q)\twoheadrightarrow
\End_{\Uq}(T_{n}^{d}),\quad\text{and}\quad\Phi^{\Bm}_{q\mathrm{SW}}\colon
\cal{H}^{\Bm}_d(q)\xrightarrow{\cong}\End_{\Uq}(T_{n}^{d}),\text{ if }{\textstyle \frac{1}{2}}n\geq d.
\end{equation*}
\item[(c)] Let $\Char(\K)=0$ and $q=1$, or let 
$q\in\K^{\ast}$ be a non-root of unity. Then there is 
a $(\Uq,\cal{H}^{\Bm}_d(q))$-bimodule decomposition
\begin{equation*}
T_{m,m}^{d}\cong\bigoplus_{d_1+d_2=d}\bigoplus_{\lambda\in\Lambda^+(d_1)\atop \mu\in\Lambda^+(d_2)}\Delta_q(\lambda,\mu)\otimes D_q^{\lambda,\mu},
\end{equation*}
with simple $\Uq$-modules $\Delta_q(\lambda,\mu)\cong L_q(\lambda,\mu)$.
\end{itemize}
\end{thm}

\begin{proof}
The statements (a) and (b), in the classical case, are 
proven in~\cite[Theorem~9]{ms} (see also~\cite[Remark~12]{ms} for the 
isomorphism criterion). The arguments 
given there go through for 
arbitrary $\K$ and $q\in\K^{\ast}$ as well. Statement (c) can be deduced from~\cite[Lemma~11]{ms} 
which again works in the semisimple, quantized case as well. See 
also~\cite[Theorem~4.3]{hu2}.
\end{proof}

\begin{rem}\label{rem-arikikoike}
The statements of Theorem~\ref{thm-schur-weylb} can be extended 
to the so-called \textit{Ariki-Koike algebras} (the Hecke algebras 
for the complex reflection groups $G(m,1,d)$), see 
for example~\cite{hs} or~\cite{sash}. Moreover, the approach 
taken in~\cite[Section~4]{ms} is setup such that it can be quantized as well. Hence,
it should give a 
quantum Schur-Weyl duality for $G(m,p,d)$ as well. 
\end{rem}

\subsection{Mixed type \texorpdfstring{$\Am$}{A}: the walled Brauer algebras and mixed Schur-Weyl duality}

The following algebra is called the \textit{walled Brauer algebra} 
(or the \textit{oriented Brauer algebra}), 
and was independently introduced in~\cite{koike} 
and~\cite{turaev}.

\begin{defn}\label{defn-walledbrauer}
Let $r,s\in\Z_{\geq 0}$ not both zero, $\delta\in\K$.
The \textit{walled Brauer algebra} $\cal{B}_{r,s}(\delta)$ is the 
$\K$-algebra
generated by $\{\sigma_i\mid i=1,\dots,r+s-1,\, i\neq r\}\cup\{u_r\}$ subject to the 
relations
\begin{equation*}
\begin{gathered}
\sigma_i^2=1,\quad\text{for }i=1,\dots,r+s-1,\;i\neq r,
\\
\sigma_i\sigma_j=\sigma_j\sigma_i,\quad\text{for }|i-j|>1,\quad\quad
\sigma_i\sigma_j\sigma_i=\sigma_j\sigma_i\sigma_j,\quad\text{for }|i-j|=1,
\\
u_r^2=\delta u_r,\quad\quad u_r=u_r\sigma_{r-1}u_r=u_r\sigma_{r+1}u_r,\quad\quad 
u_r\sigma_j=\sigma_ju_r,\quad\text{for }|r-j|>1,
\\
u_r\sigma_{r-1}\sigma_{r+1}u_r\sigma_{r-1}\sigma_{r+1}=
u_r\sigma_{r-1}\sigma_{r+1}u_r=
\sigma_{r-1}\sigma_{r+1}u_r\sigma_{r-1}\sigma_{r+1}u_r.
\end{gathered}
\end{equation*}
Note that $\K[S_r]$ is a subalgebra of $\cal{B}_{r,s}(\delta)$ 
as well as a quotient of $\cal{B}_{r,s}(\delta)$ (given by killing the 
ideal generated by $u_r$). 
Similarly for $s$ instead of $r$.
\end{defn}

\begin{lem}\label{lem-walledsecondb}
Let $0<\Char(\K)=p\leq\min\{r,s\}$. Then $\cal{B}_{r,s}(\delta)$ is non-semisimple.
\end{lem}

\begin{proof}
Assume $r\leq s$ and $0<p\leq r$. Note that $\K[S_r]$ is 
a non-semisimple quotient of $\cal{B}_{r,s}(\delta)$ by Maschke's Theorem. Since 
quotients of semisimple algebras are semisimple, $\cal{B}_{r,s}(\delta)$ 
can not be semisimple.
Dually for $r\geq s$ and $0<p\leq s$. Hence, the statement follows.
\end{proof}

\begin{rem}\label{rem-walledbrauer}
One can think of the generators of $\cal{B}_{r,s}(\delta)$ as being
generators of Kauffman's \textit{oriented} tangle algebra with $r$ left upwards
and $s$ right downwards pointing arrows as follows:
\begin{gather*}
\sigma_{<r}=
\xy
(0,0)*{
\begin{tikzpicture}[scale=.3]
	\draw [very thick, ->] (-5,-1) to (-5,1);
	\draw [very thick, ->] (-3,-1) to (-3,1);
	\draw [very thick, ->] (-1,-1) to (1,1);
	\draw [very thick, ->] (1,-1) to (-1,1);
	\draw [very thick, ->] (3,-1) to (3,1);
	\draw [very thick, ->] (5,-1) to (5,1);
	\draw [very thick, ->] (7,1) to (7,-1);
	\draw [very thick, ->] (9,1) to (9,-1);
	\node at (-3.9,0) {\tiny $\cdots$};
	\node at (4.1,0) {\tiny $\cdots$};
	\node at (8.1,0) {\tiny $\cdots$};
\end{tikzpicture}
};
\endxy\quad,\quad
\sigma_{>r}=
\xy
(0,0)*{
\begin{tikzpicture}[scale=.3]
	\draw [very thick, ->] (3,-1) to (3,1);
	\draw [very thick, ->] (5,-1) to (5,1);
	\draw [very thick, ->] (7,1) to (7,-1);
	\draw [very thick, ->] (9,1) to (9,-1);
	\draw [very thick, ->] (11,1) to (13,-1);
	\draw [very thick, ->] (13,1) to (11,-1);
	\draw [very thick, ->] (15,1) to (15,-1);
	\draw [very thick, ->] (17,1) to (17,-1);
	\node at (3.9,0) {\tiny $\cdots$};
	\node at (8.1,0) {\tiny $\cdots$};
	\node at (16.1,0) {\tiny $\cdots$};
\end{tikzpicture}
};
\endxy\\
u_r\;\;=
\xy
(0,0)*{
\begin{tikzpicture}[scale=.3]
	\draw [very thick, ->] (-1,-1) to (-1,1);
	\draw [very thick, ->] (1,-1) to (1,1);
	\draw [very thick, ->] (3,-1) to (3,1);
	\draw [very thick, directed=.55] (5,-1) to [out=90,in=180] (6,-.35) to [out=0,in=90] (7,-1);
	\draw [very thick, directed=.55] (7,1) to [out=270,in=0] (6,.35) to [out=180,in=270] (5,1);
	\draw [very thick, ->] (9,1) to (9,-1);
	\draw [very thick, ->] (11,1) to (11,-1);
	\draw [very thick, ->] (13,1) to (13,-1);
	\node at (2.1,0) {\tiny $\cdots$};
	\node at (10.1,0) {\tiny $\cdots$};
\end{tikzpicture}
};
\endxy
\end{gather*}
In this setting, the relations from Definition~\ref{defn-walledbrauer} can be interpreted in the 
usual, topological sense of Kauffman's tangle algebra (each internal circle 
can be removed and gives a factor $\delta\in\K$).
Here is an example of a typical element in $\cal{B}_{3,2}(\delta)$:
\[
\xy
(0,0)*{
\begin{tikzpicture}[scale=.3]
	\draw [very thick] (-5,-1) to (-3,1);
	\draw [very thick] (-3,-1) to (-5,1);
	\draw [very thick, ->] (-1,1) to (-1,-1);
	\draw [very thick, ->] (1,1) to (1,-1);
	\draw [very thick] (-5,1) to (-5,3);
	\draw [very thick] (-3,1) to [out=90,in=180] (-2,1.75) to [out=0,in=90] (-1,1);
	\draw [very thick] (-1,3) to (1,1);
	\draw [very thick] (1,3) to [out=270,in=0] (-1,2.25) to [out=180,in=270] (-3,3);
	\draw [very thick] (-5,3) to (-7,5);
	\draw [very thick, directed=.25] (-3,3) to [out=90,in=180] (-1,3.75) to [out=0,in=90] (1,3);
	\draw [very thick] (-1,5) to (-1,3);
	\draw [very thick] (1,5) to [out=270,in=0] (-1,4.25) to [out=180,in=270] (-3,5);
	\draw [very thick] (-7,-1) to (-7,3);
	\draw [very thick] (-7,3) to (-5,5);
	\draw [very thick, ->] (-7,5) to (-5,7);
	\draw [very thick, ->] (-5,5) to (-7,7);
	\draw [very thick, ->] (-3,5) to (-3,7);
	\draw [very thick] (-1,5) to (1,7);
	\draw [very thick] (1,5) to (-1,7);
\end{tikzpicture}
};
\endxy
\;
=
\;
\delta\cdot
\xy
(0,0)*{
\begin{tikzpicture}[scale=.3]
	\draw [very thick, ->] (-7,-1) to (-7,3);
	\draw [very thick, ->] (1,3) to (1,-1);
	\draw [very thick] (-5,-1) to (-3,1);
	\draw [very thick] (-3,-1) to (-5,1);
	\draw [very thick, ->] (-5,1) to (-5,3);
	\draw [very thick, ->] (-1,1) to (-1,-1);	
	\draw [very thick] (-3,1) to [out=90,in=180] (-2,1.65) to [out=0,in=90] (-1,1);
	\draw [very thick] (-1,2.9) to [out=270,in=0] (-2,2.25) to [out=180,in=270] (-3,2.9);
	\draw [very thick, ->] (-3,2.9) to (-3,3);
	\draw [very thick] (-1,2.9) to (-1,3);
\end{tikzpicture}
};
\endxy
\in\cal{B}_{3,2}(\delta).
\] 
A \textit{primitive (walled Brauer) diagram} is a single diagram (instead of a linear combination) 
of Kauffman's oriented tangle 
algebra without internal circles. These form a basis of 
$\cal{B}_{r,s}(\delta)$.

One could also define a \textit{quantized} walled Brauer algebra 
$\cal{B}_{r,s}([\delta])$, see~\cite[Definition~2.2]{dds}.
\end{rem}

\begin{nota}\label{nota-delta}
From now on: if we have $\Char(\K)=p$, then 
we additionally assume that $\delta\in\F\subset\K$. Here $\F$ is 
the field with $p$ elements. Hence, 
there is a minimal $\delta_p\in\N$ 
such that $\delta\equiv \delta_p\;\text{mod}\;p$.
By convention, $\delta\in\Z$ and $\delta_0=|\delta|$ if $\Char(\K)=0$.
\end{nota}

We can associate to each 
pair of 
Young diagrams $(\lambda,\mu)$ with 
$\lambda\in\Lambda^+(r-i)$ and $\mu\in\Lambda^+(s-i)$ (for $i=0,1\dots,\min\{r,s\}$) a 
$\cal{B}_{r,s}(\delta)$-module via induction, see~\cite[(2.9)]{bs5} for the 
classical case and~\cite[Theorem~2.7]{cddm} for the general case. That is,
\begin{equation}\label{eq-stilladded}
D_1^{\lambda,{\mu}}=\cal{B}_{r,s}(\delta)
\otimes_{\K(S_{r})\otimes\K(S_{s})}
(D_1^{\lambda}\otimes D_1^{\mu}).
\end{equation}
If $\K=\C$ and $r+s\leq\delta_0+1$, 
then these $\cal{B}_{r,s}(\delta)$-modules are exactly the simple 
$\cal{B}_{r,s}(\delta)$-modules, see~\cite[Theorem~2.7]{cddm}.

Let $V=\Delta_1(\omega_1)$ denote the $n=m$ dimensional vector representation 
of $\Uo=\Uo(\gll{m})$ and let $V^*$ be its dual. We set 
$T_{n}^{r,s}=V^{\otimes r}\otimes (V^*)^{\otimes s}$. 
By~\cite[Section~3]{dds}, there is an 
action of $\cal{B}_{r,s}(\delta)$ on $T_{n}^{r,s}$ for $n=\delta_p$. This 
action is given by letting $\K(S_r)$ and $\K(S_s)$ act by 
permutations. In order to explain 
the action of $u_r$ denote by $\{1,\dots,n\}$ and $\{\overline{1},\dots,\overline{n}\}$ 
a basis of $V$ and its dual basis of $V^*$ respectively. Let $v,w\in\{1,\dots,n\}$ with $v\neq w$. Then 
$u_r$ acts on the components $r$ and $r+1$ of a primitive tensor in $T_{n}^{r,s}$ by sending 
$v\otimes\overline{v}$ to 
$\sum_{i=1}^ni\otimes\overline{i}$ and $v\otimes\overline{w}$ to zero.

Recall the following `mixed' version of Schur-Weyl duality.

\begin{thm}(\textbf{Mixed Schur-Weyl duality, type $\Am$})\label{thm-schur-weyl2}
Let $n=\delta_p$.
\begin{itemize}
\item[(a)] The actions of $\Uo$ and $\cal{B}_{r,s}(\delta)$ on $T_{n}^{r,s}$ commute.
\item[(b)] Let $\Phi_{\mathrm{wBr}}$ be the algebra homomorphism induced by the action 
of $\cal{B}_{r,s}(\delta)$ on $T_{n}^{r,s}$. Then
\begin{gather*}
\begin{aligned}
\Phi_{\mathrm{wBr}}\colon\cal{B}_{r,s}(\delta)\twoheadrightarrow
\End_{\Uo}(T_{n}^{r,s}),\quad\text{and}\quad\Phi_{\mathrm{wBr}}\colon
\cal{B}_{r,s}(\delta)\xrightarrow{\cong}\End_{\Uo}(T_{n}^{r,s}),\text{ if }n\geq r+s.
\end{aligned}
\end{gather*}
\item[(c)] Let $\K=\C$ and $r+s\leq\delta_0+1$. Moreover, set $Y=\{0,1,\dots,\min\{r,s\}\}$. 
Then there is 
a $(\Uo,\cal{B}_{r,s}(\delta))$-bimodule decomposition
\begin{equation*}
T_{n}^{r,s}\cong\bigoplus_{i\in Y}\bigoplus_{\lambda\in\Lambda^+(r-i) \atop \mu\in\Lambda^+(s-i)}\Delta_1(\lambda,\overline{\mu})\otimes D_1^{\lambda,\overline{\mu}},
\end{equation*}
with simple $\Uo$-modules $\Delta_1(\lambda,\overline{\mu})\cong L_1(\lambda,\overline{\mu})$.
\end{itemize}
\end{thm}

\begin{proof}
Part (a) and (b) are proven in~\cite[Theorem~7.1 and Corollary~7.2]{dds}. The statement 
(c) can be derived from~\cite[Theorem~1.1]{koike} together with~\eqref{eq-stilladded}.
\end{proof}

\begin{rem}\label{rem-qversion}
\begin{enumerate}[(a)]
\item The assumption $r+s\leq \delta_0+1$ in (c) of Theorem~\ref{thm-schur-weyl2} 
will turn out to be necessary to ensure that $\cal{B}_{r,s}(\delta)$ is 
semisimple.
\item Note that there is also a quantized version 
of Theorem~\ref{thm-schur-weyl2}, see~\cite{dds}.
\end{enumerate}
\end{rem}

\subsection{Types \texorpdfstring{$\Bm,\Cm,\Dmm$}{B,C,D}: the Brauer algebras and Schur-Weyl-Brauer duality}

The following algebra, called the \textit{Brauer algebra}, 
goes back to work of Brauer~\cite{bra}.

\begin{defn}\label{defn-brauer}
Let $d\in\Zg$, $\delta\in\K$.
The \textit{Brauer algebra} $\cal{B}_{d}(\delta)$ is the 
$\K$-algebra
generated by $\{\sigma_i,u_i\mid i=1,\dots,d-1\}$ subject to the 
relations
\begin{gather*}
\sigma_i^2=1, \quad \quad\text{for }i=1,\dots,d-1,
\\
\sigma_i\sigma_j=\sigma_j\sigma_i,\quad\text{for }|i-j|>1,\quad\quad
\sigma_i\sigma_j\sigma_i=\sigma_j\sigma_i\sigma_j,\quad\text{for }|i-j|=1,
\\
u_i^2=\delta u_i,\quad\text{for }i=1,\dots,d-1,\quad\quad u_iu_j=u_ju_i,\quad\text{for }|i-j|>1,
\\
u_iu_{i+1}u_i=u_i,\quad\text{for }i=1,\dots,d-2,\quad\quad
u_iu_{i-1}u_i=u_i,\quad\text{for }i=2,\dots,d-1,\\
\sigma_iu_ju_i=\sigma_ju_i\quad\text{for }|i-j|=1,\quad\quad
\sigma_iu_i=u_i=u_i\sigma_i,\quad\text{for }i=1,\dots,d-1.
\end{gather*}
Note that $\K[S_d]$ is a subalgebra of $\cal{B}_{d}(\delta)$ 
as well as a quotient of $\cal{B}_{d}(\delta)$ (given by killing the 
ideal generated by the $u_i$'s).
\end{defn}

\begin{lem}\label{lem-brauersecondnew}
Let $\Char(\K)=p\leq d$. Then $\cal{B}_{d}(\delta)$ is non-semisimple.
\end{lem}

\begin{proof}
Analogously to Lemma~\ref{lem-walledsecondb}.
\end{proof}

\begin{rem}\label{rem-brauer}
One can think of the generators of $\cal{B}_{d}(\delta)$ as being
generators of Kauffman's \textit{unoriented} tangle algebra with $d$ strands as follows:
\begin{gather*}
\sigma_{i}=
\xy
(0,0)*{
\begin{tikzpicture}[scale=.3]
	\draw [very thick] (-5,-1) to (-5,1);
	\draw [very thick] (-3,-1) to (-3,1);
	\draw [very thick] (-1,-1) to (1,1);
	\draw [very thick] (1,-1) to (-1,1);
	\draw [very thick] (3,-1) to (3,1);
	\draw [very thick] (5,-1) to (5,1);
	\node at (-3.9,0) {\tiny $\cdots$};
	\node at (4.1,0) {\tiny $\cdots$};
\end{tikzpicture}
};
\endxy\quad,\quad
u_i=
\xy
(0,0)*{
\begin{tikzpicture}[scale=.3]
	\draw [very thick] (1,-1) to (1,1);
	\draw [very thick] (3,-1) to (3,1);
	\draw [very thick] (5,-1) to [out=90,in=180] (6,-.35) to [out=0,in=90] (7,-1);
	\draw [very thick] (7,1) to [out=270,in=0] (6,.35) to [out=180,in=270] (5,1);
	\draw [very thick] (9,1) to (9,-1);
	\draw [very thick] (11,1) to (11,-1);
	\node at (2.1,0) {\tiny $\cdots$};
	\node at (10.1,0) {\tiny $\cdots$};
\end{tikzpicture}
};
\endxy
\end{gather*}
Removing circles gives a factor $\delta\in\K$ again. An example is
\[
\xy
(0,0)*{
\begin{tikzpicture}[scale=.3]
	\draw [very thick] (-5,-1) to (-3,1);
	\draw [very thick] (-3,-1) to (-5,1);
	\draw [very thick] (-1,1) to (-1,-1);
	\draw [very thick] (1,1) to (1,-1);
	\draw [very thick] (-5,1) to (-5,3);
	\draw [very thick] (-3,1) to [out=90,in=180] (-2,1.75) to [out=0,in=90] (-1,1);
	\draw [very thick] (-1,3) to (1,1);
	\draw [very thick] (1,3) to [out=270,in=0] (-1,2.25) to [out=180,in=270] (-3,3);
	\draw [very thick] (-5,3) to (-7,5);
	\draw [very thick] (-3,3) to [out=90,in=180] (-1,3.75) to [out=0,in=90] (1,3);
	\draw [very thick] (-1,5) to (-1,3);
	\draw [very thick] (1,5) to [out=270,in=0] (-1,4.25) to [out=180,in=270] (-3,5);
	\draw [very thick] (-7,-1) to (-7,3);
	\draw [very thick] (-7,3) to (-5,5);
	\draw [very thick] (-7,5) to (-7,7);
	\draw [very thick] (-1,5) to [out=90,in=180] (0,5.65) to [out=0,in=90] (1,5);
	\draw [very thick] (1,7) to [out=270,in=0] (0,6.35) to [out=180,in=270] (-1,7);
	\draw [very thick] (-5,5) to [out=90,in=180] (-4,5.65) to [out=0,in=90] (-3,5);
	\draw [very thick] (-3,7) to [out=270,in=0] (-4,6.35) to [out=180,in=270] (-5,7);
\end{tikzpicture}
};
\endxy
\;
=
\;
\delta\cdot
\xy
(0,0)*{
\begin{tikzpicture}[scale=.3]
	\draw [very thick] (-3,0) to [out=90,in=180] (-1,1) to [out=0,in=90] (1,0);
	\draw [very thick] (-5,0) to [out=90,in=180] (-1,2) to [out=0,in=90] (3,0);
	\draw [very thick] (-1,0) to (-5,4);
	\draw [very thick] (-1,4) to [out=270,in=0] (-2,3.35) to [out=180,in=270] (-3,4);
	\draw [very thick] (3,4) to [out=270,in=0] (2,3.35) to [out=180,in=270] (1,4);
\end{tikzpicture}
};
\endxy
\in\cal{B}_5(\delta).
\]
A \textit{primitive (Brauer) diagram} is a single diagram (instead of a linear combination) 
of Kauffman's unoriented tangle 
algebra without internal circles. These form a basis of 
$\cal{B}_{d}(\delta)$.

There is also again a \textit{quantized} Brauer algebra 
$\cal{B}_{d}([\delta])$ (called \textit{Birman-Murakami-Wenzl or BMW algebra}), see for example~\cite[Section~2]{mw}.
\end{rem}

For the Brauer algebras we use from now on 
the same conventions for the parameter $\delta_p$ 
as for the walled Brauer algebras, see Conventions~\ref{nota-delta}.
Recall that we can associate to each 
Young diagram $\lambda$ with
$\lambda\in\Lambda^+(d-2i)$ 
(for $i=0,1,\dots,\lfloor{\textstyle \frac{1}{2}}d\rfloor$) a 
$\cal{B}_{d}(\delta)$-module $D_1^{\lambda}$. If $\K=\C$ and $2d\leq \delta_0+1$, then 
these are simple $\cal{B}_{d}(\delta)$-modules, 
see for example~\cite[Section~4]{gl}.

Let $V=\Delta_1(\omega_1)$ denote the $n$ dimensional vector representation 
of $\Uo=\Uo(\fg)$ for $\fg$ being either $\soo{2m+1}$, 
$\spo{2m}$ or $\soo{2m}$ (here $n=2m+1$ for $\fg=\soo{2m+1}$ or $n=2m$ for $\fg=\spo{2m}$ 
and for $\fg=\soo{2m}$).  
By~\cite[Theorem~3.11]{es1}, there is an 
action of $\cal{B}_{d}(\delta)$ on $T_{n}^{d}=V^{\otimes d}$ 
for $n=\delta_p$. The action is very similar to the one 
for $\cal{B}_{r,s}(\delta)$ recalled above. We point out that in type $\CM$ the 
action of $u_i\in\cal{B}_d(\delta)$ has an additional sign coming from the 
part of odd parity from the super action studied in~\cite[Theorem~3.11]{es1}.

\begin{thm}(\textbf{Schur-Weyl-Brauer duality, types $\Bm,\Cm,\Dmm$})\label{thm-schur-weyl3}
Let $n=\delta_p$.
\begin{itemize}
\item[(a)] The actions of $\Uo$ and $\cal{B}_{d}(\delta)$ on $T_{n}^{d}$ commute.
\item[(b)] Let $\Phi_{\mathrm{Br}}$ be the algebra homomorphism induced by the action 
of $\cal{B}_{d}(\delta)$ on $T_{n}^{d}$. Then
\begin{gather*}
\begin{aligned}
\BM\colon&\;\Phi_{\mathrm{Br}}\colon\cal{B}_{d}(\delta)\twoheadrightarrow
\End_{\Uo}(T_{n}^{d}),\quad\text{and}\quad\Phi_{\mathrm{Br}}\colon
\cal{B}_{d}(\delta)\xrightarrow{\cong}\End_{\Uo}(T_{n}^{d})\text{ if }n\geq d.\\
\CM\colon&\;\Phi_{\mathrm{Br}}\colon\cal{B}_{d}(\delta)\twoheadrightarrow
\End_{\Uo}(T_{n}^{d}),\quad\text{and}\quad\Phi_{\mathrm{Br}}\colon
\cal{B}_{d}(\delta)\xrightarrow{\cong}\End_{\Uo}(T_{n}^{d})\text{ if }n\geq 2d.\\
\DM\colon &\;\Phi_{\mathrm{Br}}\colon\cal{B}_{d}(\delta)\hookrightarrow
\End_{\Uo}(T_{n}^{d})\text{ if }n\geq d,\;\text{and}\;\Phi_{\mathrm{Br}}\colon\cal{B}_{d}(\delta)\xrightarrow{\cong}\End_{\Uo}(T_{n}^{d})\text{ if }n\geq 2d+1.
\end{aligned}
\end{gather*}
\item[(c)] Let $\K=\C$, $d\leq \delta_0+1$ and $\Uo=\Uo(\mathfrak{sp}_{2m})$ (thus, we have $2m=n=-\delta$). 
Moreover, set $Y=\{0,1,\dots,\lfloor\frac{1}{2}d\rfloor\}$.
Then there is 
a $(\Uo,\cal{B}_d(\delta))$-bimodule decomposition
\begin{equation*}
T_{n}^{d}\cong\bigoplus_{i\in Y}\bigoplus_{\lambda\in\Lambda^+(d-2i)}\Delta_1(\lambda)\otimes D_1^{\lambda},
\end{equation*}
with simple $\Uo$-modules $\Delta_1(\lambda)\cong L_1(\lambda)$.
\end{itemize}
\end{thm}

\begin{proof}
The parts (a),(b) are proven 
in~\cite[Theorem~5.5]{es1} (published version), but the 
criterion given there is not optimal in type $\BM$. 
The 
above bound holds in type $\BM$: note that $T_{n}^{d}\in\T$ (since we assume that 
$\Char(\K)\neq 2$). Thus, by Proposition~\ref{prop-prop}, 
$\dim(\End_{\Uo(\mathfrak{so}_{2m+1})}(T_{n}^{d}))$ is 
as in the classical case. 
Hence, the above bound follows from the classical bound 
(which can already be found implicitly in Brauer's work~\cite{bra}).
The statement 
(c) is given in~\cite[Theorem~1.1]{hu}.
\end{proof}

\begin{rem}\label{rem-qversion2}
\begin{enumerate}[(a)]
\item The assumption $d\leq \delta_0+1$ in (c) of Theorem~\ref{thm-schur-weyl3} 
again turns out to be necessary to ensure that $\cal{B}_{d}(\delta)$ is 
semisimple.
\item Note that there are also quantized versions of Theorem~\ref{thm-schur-weyl3} 
(in some cases and for appropriate parameters), see for example~\cite[Theorem~1.3]{hu} 
or~\cite[(9.6)]{lz2}.
\end{enumerate}
\end{rem}

\subsection{A slightly stronger statement in type \texorpdfstring{$\Dmm$}{D}}

Let $m\geq 1$.
Then surjectivity of $\Phi_{\mathrm{Br}}$ fails in general for $\Uo=\Uo(\soo{2m})$. 
In the remaining part of this section we will determine the image and show
$\mathrm{im}(\Phi_{\mathrm{Br}})\cong\End_{\Uoo}(T^d_n)$, see Theorem~\ref{thm-brauerimage}. 
Here, as we explain below, $\Uoo$ is obtained 
from a non-trivial symmetry of the Dynkin diagram 
of type $\DM$ as in~\eqref{eq-nontrivial}. 
The proof of Theorem~\ref{thm-brauerimage} is slightly 
involved and the main part is a counting argument comparing multiplicities of 
$\Uoo$-modules to multiplicities of $\Uo(\soo{2m^{\prime}})$-modules 
for some large enough $m^{\prime}$, 
see Lemma~\ref{lem-yetanotherone}. We like to point out that our approach is inspired 
partly by~\cite[Section~8]{lz}.

Suppose $\Char(\K)\neq 2$. 
Denote by $\sigma\colon\Uo\to\Uo$ the involution induced by a graph automorphism of the Dynkin 
diagram of type $\DM$. For $m\geq 4$ the automorphism is given via
\begin{equation}\label{eq-nontrivial}
\xy
(0,0)*{
\begin{tikzpicture}[scale=.75]
	\draw [very thick] (0.25,0.01) to (1.25,0.01);
	\draw [very thick] (1.75,0.01) to (2.75,0.01);
	\draw [very thick] (3.25,0.01) to (3.45,0.01);
	\draw [very thick] (4.05,0.01) to (4.25,0.01);
	\draw [very thick] (4.75,0.01) to (5.75,0.01);
	\draw [very thick] (6.138,0.161) to (6.65,1.1);
	\draw [very thick] (6.138,-0.161) to (6.65,-1.1);
	\draw [very thick, <->] (7,1) to [out=315,in=45] (7,-1);
	\draw [very thick] (0,0) circle (0.25cm);
	\draw [very thick] (1.5,0) circle (0.25cm);
	\draw [very thick] (3,0) circle (0.25cm);
	\draw [very thick] (4.5,0) circle (0.25cm);
	\draw [very thick] (6,0) circle (0.25cm);
	\draw [very thick] (6.75,1.3) circle (0.25cm);
	\draw [very thick] (6.75,-1.3) circle (0.25cm);
	\node at (3.8,0) {\tiny $\cdots$};
	\node at (7.75,0) {$\sigma$};
	\node at (0,0.4) {\tiny $\alpha_1$};
	\node at (1.5,0.4) {\tiny $\alpha_2$};
	\node at (3,0.4) {\tiny $\alpha_3$};
	\node at (4.5,0.4) {\tiny $\alpha_{m-3}$};
	\node at (5.75,0.4) {\tiny $\alpha_{m-2}$};
	\node at (7.6,1.3) {\tiny $\alpha_{m-1}$};
	\node at (7.35,-1.3) {\tiny $\alpha_{m}$};
\end{tikzpicture}
};
\endxy
\end{equation}
Moreover, if $m=1$, then $\soo{2}$ is one-dimensional 
and $\sigma$ is trivial. If $m=2$, then $\soo{4}\cong\mathfrak{sl}_2\times\mathfrak{sl}_2$ 
and $\sigma$ swaps the two components of the Dynkin diagram of type 
$\mathbf{A}_1\times\mathbf{A}_1$. For $m=3$, we have 
$\soo{6}\cong\mathfrak{sl}_4$ and $\sigma$ swaps the two extremal nodes of the Dynkin 
diagram of type $\mathbf{A}_3$.

Suppose $M\in\Mod{\Uo}$. Denote by ${}^{\sigma}\!M$ the $\sigma$\textit{-twist of} $M$, that is, 
${}^{\sigma}\! M=M$ as $\K$-vector spaces with $\Uo$ action via
$x\overline{m}=\overline{\sigma(x)m}$ for all $x\in\Uo,m\in M$.
Here we have written $\overline{m}$ for an element $m\in M$ when considered as an element 
of ${}^{\sigma}\! M$.
Set
\begin{equation}\label{eq-nowidetilde}
\tilde M=M\oplus {}^{\sigma}\! M.
\end{equation}

Denote by $\Uoo$ the \textit{skew group ring} $\Uo\rtimes\K(\Z/2\Z)$ 
and by $\tau$ the generator of $\K(\Z/2\Z)$. The elements of $\Uoo$ are 
of the form $x+y\tau$ for $x,y\in\Uo$ and multiplication in $\Uoo$ is such that 
$\Uo$ and $\K(\Z/2\Z)$ are subalgebras together with
\begin{equation}\label{eq-nexttry}
\tau x=\sigma(x)\tau,\quad\text{for all }x\in\Uo.
\end{equation}
As a semidirect product of Hopf algebras, $\Uoo$ is itself a Hopf algebra 
with Hopf subalgebras $\Uo$ and $\K(\Z/2\Z)$. In particular, $\tau$ is \textit{group-like}
and acts on a tensor product as $\tau\otimes\tau$. Moreover, 
$\tilde M$ from above is a $\Uoo$-module with $\tau$-action given via 
$\tau(m,\overline{n})=(n,\overline{m})$ for all $m,n\in M$ (a computation 
shows that, under this convention,~\eqref{eq-nexttry} is preserved).

Now suppose $\Char(\K)=0$. Then the simple $\Uo$-modules are 
the Weyl modules $\Delta_1(\lambda)$ for $\lambda=\sum_{i=1}^m\lambda_i\varepsilon_i\in X^+$ 
with highest weight vector $v_{\lambda}$. 
As $\Uo$-modules, we have 
\[
\quad{}^{\sigma}\!\Delta_1(\lambda)\cong\Delta_1(\overline{\lambda}),\quad\text{for all }\lambda\in X^+,
\]
where $\overline{\lambda}=\sum_{i=1}^{m-1}\lambda_i\varepsilon_i-\lambda_{m}\varepsilon_m$.
In particular, ${}^{\sigma}\!\Delta_1(\lambda)$ has 
$\overline{v}_{\lambda}$ as highest weight vector. We use the abbreviation $\tilde{\Delta}_1(\lambda)$ 
for the $\Uoo$-module given as in~\eqref{eq-nowidetilde} (for $M=\Delta_1(\lambda)$).

In case $\lambda_m=0$ (thus, $\lambda=\overline{\lambda}$), 
set $w_{\lambda}=(v_{\lambda},\overline{v}_{\lambda})$ and 
$w^{\prime}_{\lambda}=(v_{\lambda},-\overline{v}_{\lambda})$. 
Then the $\Uoo$-module 
$\tilde{\Delta}_1(\lambda)$ 
decomposes into $\pm 1$ eigenspaces of $\tau$ given by
\[
\tilde{\Delta}^+_1(\lambda)=\Uo w_{\lambda},\quad\quad\tilde{\Delta}^-_1(\lambda)=\Uo w^{\prime}_{\lambda}.
\]
For example, the vector representation $V$ of $\Uo$ is a $\Uoo$-module with trivial action of 
$\tau$ and $V\cong \tilde{\Delta}^+_1(\omega_1)$. 
These notations enables us to classify simple $\Uoo$-modules in case $\Char(\K)=0$.

\begin{prop}\label{prop-clifford}
Let $\Char(\K)=0$ and $\lambda\in X^+$. Then $\Mod{\Uoo}$ is semisimple, and:
\begin{enumerate}[(a)]
\item If $\lambda_m\neq 0$, then $\tilde{\Delta}_1(\lambda)\cong \tilde{\Delta}_1(\overline{\lambda})$ 
is a simple $\Uoo$-module.
\item If $\lambda_m=0$, then $\tilde{\Delta}^{\pm}_1(\lambda)$ are simple $\Uoo$-modules.
\item Up to isomorphism, the set
\[
\{\tilde{\Delta}_1(\lambda)\mid \lambda\in X^+,\lambda_m>0\}\cup\{\tilde{\Delta}^{\pm}_1(\lambda)\mid \lambda\in X^+,\lambda_m=0\}
\]
is a complete list of non-isomorphic, simple $\Uoo$-modules.
\end{enumerate}
\end{prop}

\begin{proof}
By the above discussion and standard Clifford theory, see for example~\cite[Section~2]{nrvo} 
or~\cite[Appendix]{rr} (both references treat a more general case).
\end{proof}

Let still $\Char(\K)=0$ and 
recall the following decomposition of $\Delta_1(\lambda)\otimes V$ as a $\Uo$-module:
\begin{equation}\label{eq-decomp}
\Delta_1(\lambda)\otimes V\cong \bigoplus_{i=1}^m\Delta_1(\lambda\pm\varepsilon_i),\quad\text{where }
\Delta_1(\lambda\pm\varepsilon_i)=\Delta_1(\lambda+\varepsilon_i)\oplus\Delta_1(\lambda-\varepsilon_i).
\end{equation}
Here $\Delta_1(\mu)=0$ (and hence, $\tilde{\Delta}_1(\mu)=0$), 
if $\mu\notin X^+$. This leads to the following lemma.

\begin{lem}\label{lem-decomptwist}
Let $\Char(\K)=0$, $\lambda\in X^+$ and $\epsilon\in\{+,-\}$. Then, as $\Uoo$-modules:
\begin{gather*}
\begin{aligned}
\text{If }\lambda_m>1: \quad \tilde\Delta_1(\lambda)\otimes V&\cong \bigoplus_{i=1}^m\tilde\Delta_1(\lambda\pm\varepsilon_i),\\
\text{if }\lambda_m=1: \quad \tilde\Delta_1(\lambda)\otimes V&\cong \bigoplus_{i=1}^{m-1}\tilde\Delta_1(\lambda\pm\varepsilon_i)\oplus\tilde\Delta_1(\lambda+\varepsilon_m)
\oplus\tilde\Delta^+_1(\lambda-\varepsilon_m)\oplus\tilde\Delta^-_1(\lambda-\varepsilon_m),\\
\text{if }\lambda_m=0:\quad \tilde\Delta^{\epsilon}_1(\lambda)\otimes V&\cong \bigoplus_{i=1}^{m-1}\tilde\Delta^{\epsilon}_1(\lambda\pm\varepsilon_i)\oplus\tilde\Delta_1(\lambda+\varepsilon_m).
\end{aligned}
\end{gather*}
\end{lem}

\begin{proof}
We have $\tilde{\Delta}_1(\lambda)\cong\tilde\Delta_1(\overline{\lambda})$ for $\lambda_m\neq 0$ and 
$\tilde{\Delta}_1(\lambda)\cong\Delta^+_1(\lambda)\oplus\Delta^-_1(\lambda)$ for $\lambda_m=0$.
Using Proposition~\ref{prop-clifford} and~\eqref{eq-decomp}, the statement follows.
\end{proof}

This leads to the following multiplicity formulas of $\Uo$-modules.

\begin{prop}\label{prop-decomptwist2}
Let $\Char(\K)=0$, $\lambda\in X^+$ and $\epsilon\in\{+,-\}$. As usual, let $T^d_n=V^{\otimes d}$. Then:
\begin{gather*}
\begin{aligned}
\text{If }\lambda_m>1: \quad (T^d_n:\tilde{\Delta}_1(\lambda))&=\sum_{i=1}^m(T^{d-1}_n:\tilde{\Delta}_1(\lambda\pm \varepsilon_i)),\\
\text{if }\lambda_m=1: \quad (T^d_n:\tilde{\Delta}_1(\lambda))&=\sum_{i=1}^{m-1}(T^{d-1}_n:\tilde{\Delta}_1(\lambda\pm \varepsilon_i))+(T^{d-1}_n:\tilde\Delta_1(\lambda+\varepsilon_m))\\
&\quad+(T^{d-1}_n:\tilde\Delta^+_1(\lambda-\varepsilon_m))+(T^{d-1}_n:\tilde\Delta^-_1(\lambda-\varepsilon_m)),\\
\text{if }\lambda_m=0: \quad (T^{d}_n:\tilde{\Delta}^{\epsilon}_1(\lambda))&=\sum_{i=1}^{m-1}(T^{d-1}_n:\tilde{\Delta}^{\epsilon}_1(\lambda\pm \varepsilon_i))+(T^{d-1}_n:\tilde{\Delta}_1(\lambda+ \varepsilon_m)),
\end{aligned}
\end{gather*}
where any multiplicity is zero if the corresponding weight is not in $X^+$.
\end{prop}

\begin{proof}
For $d=1$, see Lemma~\ref{lem-decomptwist}. The general statement for $d>1$ follows recursively.
\end{proof}

%
%

Let $m^{\prime}\in\Z_{\geq 1}$ and $n^{\prime}=2m^{\prime}$ be 
such that $1\leq d\leq 2m<m^{\prime}$. We consider $\Uo^{\prime}=\Uo(\soo{2m^{\prime}})$ 
and use notations as $V^{\prime}$, $(T^{\prime})^d_{n^{\prime}}$, $X^{\prime}, \lambda^{\prime}$ 
etc. to distinguish 
these from the data for $\Uo$. 
Moreover, let $\lambda^{\prime}\in (X^{\prime})^+$ 
with $|\lambda^{\prime}|\leq d$. Then $\lambda^{\prime}_{m^{\prime}}=0$ and 
$\lambda^{\prime}_{m+1}\leq 1$. Given now a $\Uo^{\prime}$-weight 
$\lambda^{\prime}\in (X^{\prime})^+$, define a $\Uo$-weight
$\tilde\lambda=\sum_{i=1}^m(\lambda^{\prime}_i-\lambda^{\prime}_{2m-i+1})\varepsilon_i\in X^+$.
Note that, if $\lambda^{\prime}_{m+1}=0$, then $\tilde\lambda=\lambda^{\prime}$.

The main step now is to compare $\Uo^{\prime}$-multiplicities to $\Uoo$-multiplicities.

\begin{lem}\label{lem-yetanotherone}
Let $\Char(\K)=0$. With the notation from above, we have the following.
\begin{enumerate}[(a)]
\item If $\lambda^{\prime}_{m+1}=0$, then
\[
((T^{\prime})^{d}_{n^{\prime}}:\Delta^{\prime}_1(\lambda^{\prime}))=\begin{cases}(T^d_n:\tilde{\Delta}_1(\tilde\lambda)), &\text{if }\lambda^{\prime}_m>0,\\
(T^d_n:\tilde{\Delta}^+_1(\tilde\lambda)), &\text{if }\lambda^{\prime}_m=0.
\end{cases}
\]
\item If $\lambda^{\prime}_{m+1}=1$, then
\[
((T^{\prime})^{d}_{n^{\prime}}:\Delta^{\prime}_1(\lambda^{\prime}))=(T^d_n:\tilde{\Delta}^-_1(\tilde\lambda)).
\]
\end{enumerate}
\end{lem}

\begin{proof}
We use induction on $d$. If $d=1$, then 
$(T^{\prime})^{d}_{n^{\prime}}=V^{\prime}\cong\Delta_1^{\prime}(\omega_1^{\prime})$ and 
$T^d_n=V\cong\tilde\Delta_1^+(\omega_1)$. Since the $\Uo$-weight $\tilde\lambda$ associated to 
$\omega_1^{\prime}$ is $\omega_1$, (a) and (b) follow.

Assume $d>1$. Then the 
${}^{\prime}$-version of~\eqref{eq-decomp3} can be used to express the left-hand sides in (a) 
and (b) in terms of $\Uu_1^{\prime}$-multiplicities in $(T^{\prime})^{d-1}_{n^{\prime}}$ 
and Proposition~\ref{prop-decomptwist2} expresses the right-hand sides 
in terms of $\Uoo$-multiplicities in $T^{d-1}_{n}$. It is now a matter 
of bookkeeping to check that the induction hypothesis gives the stated 
equalities. We give details only for the first case of (a). So we 
have $\lambda^{\prime}_{m+1}=0$. Assume first that $\lambda^{\prime}_m>1$. Then the 
${}^{\prime}$-version of~\eqref{eq-decomp3} 
gives
\begin{equation}\label{eq-decomp3}
((T^{\prime})^{d}_{n^{\prime}}:\Delta^{\prime}_1(\lambda^{\prime}))=
\sum_{i=1}^{m+1}((T^{\prime})^{d-1}_{n^{\prime}}:\Delta^{\prime}_1(\lambda^{\prime}\pm\varepsilon_i^{\prime})).
\end{equation}
Note that $\Delta^{\prime}_1(\lambda^{\prime}-\varepsilon_{m+1}^{\prime})=0$ 
and $|\lambda^{\prime}+\varepsilon_{m+1}^{\prime}|\geq 2m+1>d-1$. Thus, 
the $m+1$th summand in~\eqref{eq-decomp3} is zero 
($((T^{\prime})^{d}_{n^{\prime}}:\Delta^{\prime}_1(\lambda^{\prime}))=0$ unless 
$|\lambda^{\prime}|\leq d$). Now, 
by induction hypothesis, $((T^{\prime})^{d-1}_{n^{\prime}}:\Delta^{\prime}_1(\lambda^{\prime}\pm\varepsilon_i^{\prime}))=(T^{d-1}_{n}:\tilde{\Delta}_1(\tilde\lambda\pm\varepsilon_i))$ for all $i=1,\dots,m$ 
(we write $\tilde\lambda\pm\varepsilon_i$ short for $\tilde\mu$ with 
$\mu=\lambda^{\prime}\pm\varepsilon^{\prime}_i$, if $i=1,\dots,m$, and 
$\mu=\lambda^{\prime}\mp\varepsilon^{\prime}_i$, if $i=m+1,\dots,m^{\prime}$). Then 
Proposition~\ref{prop-decomptwist2} gives the desired equality.

Now assume $\lambda^{\prime}_{m+1}=0$ and $\lambda^{\prime}_m=1$. Arguing as before we get
\begin{align*}
((T^{\prime})^{d}_{n^{\prime}}:\Delta^{\prime}_1(\lambda^{\prime}))&=\sum_{i=1}^m
((T^{\prime})^{d-1}_{n^{\prime}}:\Delta^{\prime}_1(\lambda^{\prime}\pm\varepsilon_i^{\prime}))
+((T^{\prime})^{d-1}_{n^{\prime}}:\Delta^{\prime}_1(\lambda^{\prime}+\varepsilon_{m+1}^{\prime}))\\
&=\sum_{i=1}^{m-1}
(T^{d-1}_{n^{\prime}}:\tilde{\Delta}_1(\tilde\lambda\pm\varepsilon_i))+
(T^{d-1}_{n^{\prime}}:\tilde{\Delta}_1(\tilde\lambda+\varepsilon_m))\\
&\quad+
(T^{d-1}_{n^{\prime}}:\tilde{\Delta}^+_1(\tilde\lambda-\varepsilon_m))+
(T^{d-1}_{n^{\prime}}:\tilde{\Delta}^-_1(\tilde\lambda-\varepsilon_m))\\
&=(T^d_n:\tilde{\Delta}_1(\tilde\lambda)).
\end{align*}
Here the second equality uses Lemma~\ref{lem-yetanotherone} for the first and 
the last term (both in case $d-1$). The last equality uses 
Proposition~\ref{prop-decomptwist2}. 
\end{proof}

\begin{cor}\label{cor-important}
Let $\Char(\K)=0$. 
If $1\leq d\leq 2m<m^{\prime}$, then
\[
\dim(\End_{\Uoo}(T^d_n))=\dim(\End_{\Uu_1^{\prime}}((T^{\prime})^{d}_{n^{\prime}})).
\]
\end{cor}

\begin{proof}
Note that $\dim(\End_{\Uoo}(T^d_n))=\sum_{L}(V^{\otimes d}:L)^2$ (the sum 
runs over all simple $\Uoo$-modules $L$ that are composition 
factors of $V^{\otimes d}$). There 
is a similar formula for $\dim(\End_{\Uu_1^{\prime}}((T^{\prime})^{d}_{n^{\prime}})$ 
as well. By Proposition~\ref{prop-clifford}, we can use Lemma~\ref{lem-yetanotherone} to 
see that the dimensions agree.
\end{proof}

We now leave the case $\Char(\K)=0$ and state and 
prove the main result of this subsection, where we only assume that $\Char(\K)\neq 2$.

\begin{thm}\label{thm-brauerimage}
If $1\leq d\leq 2m$ and $\delta_p\equiv 2m\;\text{mod}\;p$, 
then $\cal{B}_d(\delta)\cong\End_{\Uoo}(T^d_n)$.
\end{thm}

\begin{proof}
By Theorem~\ref{thm-schur-weyl3}, we know that the Schur-Weyl-Brauer 
homomorphism
\[
\Phi_{\mathrm{Br}}\colon\cal{B}_{d}(\delta^{\prime})\to
\End_{\Uu_1^{\prime}}((T^{\prime})^{d}_{n^{\prime}})
\]
is injective for $d\leq 2m^{\prime}$ and surjective for $d<m^{\prime}$ 
(where $\delta_p^{\prime}\equiv 2m^{\prime}\;\text{mod}\;p$). 
Hence, for $m^{\prime}>2m$ we have that  
$\dim(\cal{B}_{d}(\delta^{\prime}))=\dim(\End_{\Uu_1^{\prime}}((T^{\prime})^{d}_{n^{\prime}}))$ 
and so 
also $\dim(\cal{B}_{d}(\delta^{\prime}))=\dim(\End_{\Uoo}(T^d_n))$, by Corollary~\ref{cor-important} 
and Proposition~\ref{prop-prop} (since $(T^{\prime})^{d}_{n^{\prime}}$ 
is a $\Uo^{\prime}$-tilting module).
Clearly, we have 
$\mathrm{im}(\Phi_{\mathrm{Br}})\subset\End_{\Uoo}(T^d_n)$. The statement follows,
since $\dim(\cal{B}_{d}(\delta^{\prime}))=\dim(\cal{B}_{d}(\delta))$.
\end{proof}

\begin{rem}\label{rem-oneextra2}
Note that the whole discussion in this subsection goes through in case $m\leq 3$ as 
well (with the corresponding $\sigma$ from above).
\end{rem}
%
%

%% file: res/schuridem.tex
In this section we 
explicitly describe kernels of the epimorphisms 
$\Phi^{\Am}_{q\mathrm{SW}}$, $\Phi_{\mathrm{wBr}}$ 
and $\Phi_{\mathrm{Br}}$ from the 
dualities in Section~\ref{sec-schur}, some of 
which we use in the proofs of our main theorems.

In the case of $\cal{H}^{\Am}_d(q)$ 
all kernels were determined in~\cite[Theorem~4]{hae}.
In our setup, we have for $n$ as in Theorem~\ref{thm-schur-weyl1} and $q=1$
the \textit{anti-symmetrizer} 
\begin{equation}\label{eq-haert}
e_d(n)=\sum_{w\in S_d}(-1)^{l(w)}w\in\ker(\Phi^{\Am}_{q\mathrm{SW}}),
\end{equation}
where $l(w)$ the length of 
$w$. Clearly $e_d(n)e_d(n)=d!e_d(n)$. Thus, 
$e_d(n)$ is a 
quasi-idempotent (an idempotent up to an invertible scalar) 
if and only if $\Char(K)>d$ or $\Char(K)=0$ and nilpotent otherwise. 
By H\"arterich's results, the $\K$-linear span of $e_d(d-1)$ 
equals $\ker(\Phi^{\Am}_{q\mathrm{SW}})$.

For some `boundary cases' 
we can explicitly write down the kernels 
for the other algebras as well as we aim to show next. Note that this
generalizes H\"arterich's results.

\begin{defn}\label{defn-idempotent}
Define 
$e_{r,s}(\delta)\in\cal{B}_{r,s}(\delta)$ 
and $E_d(\delta)\in\cal{B}_{d}(\delta)$ via
\[
e_{r,s}(\delta)=\sum_{x}(-1)^{l(x)}x\in\cal{B}_{r,s}(\delta)
\quad\text{and}\quad E_d(\delta)=\sum_{x}(-1)^{l(x)}x\in\cal{B}_d(\delta).
\]
Here the sums run over all primitive diagrams 
(see Remarks~\ref{rem-walledbrauer} and~\ref{rem-brauer}).
\end{defn}

\begin{ex}\label{ex-idempotent}
If $r=2,s=1$ (walled Brauer case) respectively 
if $d=2$ (Brauer case), then
\begin{gather*}
e_{2,1}(\delta)\;=\;
\xy
(0,0)*{
\begin{tikzpicture}[scale=.3]
	\draw [very thick, ->] (-1,0) to (-1,4);
	\draw [very thick, ->] (1,0) to (1,4);
	\draw [very thick, ->] (3,4) to (3,0);
\end{tikzpicture}
};
\endxy
\;-\;
\xy
(0,0)*{
\begin{tikzpicture}[scale=.3]
	\draw [very thick, ->] (-1,0) to (1,4);
	\draw [very thick, ->] (1,0) to (-1,4);
	\draw [very thick, ->] (3,4) to (3,0);
\end{tikzpicture}
};
\endxy
\;-\;
\xy
(0,0)*{
\begin{tikzpicture}[scale=.3]
	\draw [very thick, ->] (-1,0) to (-1,4);
	\draw [very thick] (3,3.9) to [out=270,in=0] (2,2.9) to [out=180,in=270] (1,3.9);
	\draw [very thick] (1,0.1) to [out=90,in=180] (2,1.1) to [out=0,in=90] (3,0.1);
	\draw [very thick, ->] (1,3.9) to (1,4);
	\draw [very thick] (3,3.9) to (3,4);
	\draw [very thick, ->] (3,0.1) to (3,0);
	\draw [very thick] (1,0.1) to (1,0);
\end{tikzpicture}
};
\endxy
\;+\;
\xy
(0,0)*{
\begin{tikzpicture}[scale=.3]
	\draw [very thick, ->] (1,0) to (-1,4);
	\draw [very thick] (3,3.9) to [out=270,in=0] (2,2.9) to [out=180,in=270] (1,3.9);
	\draw [very thick] (-1,0.1) to [out=90,in=180] (1,1.1) to [out=0,in=90] (3,0.1);
	\draw [very thick, ->] (1,3.9) to (1,4);
	\draw [very thick] (3,3.9) to (3,4);
	\draw [very thick, ->] (3,0.1) to (3,0);
	\draw [very thick] (-1,0.1) to (-1,0);
\end{tikzpicture}
};
\endxy
\;+\;
\xy
(0,0)*{
\begin{tikzpicture}[scale=.3]
	\draw [very thick, ->] (-1,0) to (1,4);
	\draw [very thick] (3,3.9) to [out=270,in=0] (1,2.9) to [out=180,in=270] (-1,3.9);
	\draw [very thick] (1,0.1) to [out=90,in=180] (2,1.1) to [out=0,in=90] (3,0.1);
	\draw [very thick, ->] (-1,3.9) to (-1,4);
	\draw [very thick] (3,3.9) to (3,4);
	\draw [very thick, ->] (3,0.1) to (3,0);
	\draw [very thick] (1,0.1) to (1,0);
\end{tikzpicture}
};
\endxy
\;-\;
\xy
(0,0)*{
\begin{tikzpicture}[scale=.3]
	\draw [very thick, ->] (1,0) to (1,4);
	\draw [very thick] (3,3.9) to [out=270,in=0] (1,2.9) to [out=180,in=270] (-1,3.9);
	\draw [very thick] (-1,0.1) to [out=90,in=180] (1,1.1) to [out=0,in=90] (3,0.1);
	\draw [very thick, ->] (-1,3.9) to (-1,4);
	\draw [very thick] (3,3.9) to (3,4);
	\draw [very thick, ->] (3,0.1) to (3,0);
	\draw [very thick] (-1,0.1) to (-1,0);
\end{tikzpicture}
};
\endxy\in\cal{B}_{2,1}(\delta),\\
E_2(\delta)\;=\;
\xy
(0,0)*{
\begin{tikzpicture}[scale=.3]
	\draw [very thick] (-1,0) to (-1,4);
	\draw [very thick] (1,0) to (1,4);
\end{tikzpicture}
};
\endxy
\;-\;
\xy
(0,0)*{
\begin{tikzpicture}[scale=.3]
	\draw [very thick] (-1,0) to (1,4);
	\draw [very thick] (1,0) to (-1,4);
\end{tikzpicture}
};
\endxy
\;-\;
\xy
(0,0)*{
\begin{tikzpicture}[scale=.3]
	\draw [very thick] (3,4) to [out=270,in=0] (2,3) to [out=180,in=270] (1,4);
	\draw [very thick] (1,0) to [out=90,in=180] (2,1) to [out=0,in=90] (3,0);
\end{tikzpicture}
};
\endxy
\in\cal{B}_{2}(\delta).
\end{gather*}
Moreover, if $s=0$ and $\delta\in\K$ arbitrary, then $e_{r,0}(\delta)$ 
are the elements given in~\eqref{eq-haert}.
\end{ex}

\subsection{The walled Brauer case}

\begin{prop}\label{prop-kernel}
Let $n=r+s-1=\delta_p$. Then the $\K$-linear span 
of $e_{r,s}(n)\in\cal{B}_{r,s}(\delta)$ equals $\ker(\Phi_{\mathrm{wBr}})$.
Furthermore, the following holds.
\begin{enumerate}[(a)]
\item If $\Char(\K)>\max\{r,s\}$ or $\Char(\K)=0$, 
then $e_{r,s}(n)$ is a quasi-idempotent.
\item If $\Char(\K)=p\leq\max\{r,s\}$, then 
$e_{r,s}(n)$ is nilpotent.
\end{enumerate}
\end{prop}

\begin{proof}
The case $r+s=1$ is clear so we may assume now that $r+s\geq 2$.

\noindent\textit{Claim 1.} $e_{r,s}(n)\in\ker(\Phi_{\mathrm{wBr}})$.

\textit{Proof of Claim 1.} We want to use 
the diagrammatic presentation of $\cal{B}_{r,s}(\delta)$ 
from Remark~\ref{rem-walledbrauer}. If we denote 
a basis of $V$ by $\{1,\dots,n\}$
and its dual basis of $V^*$ by $\{\overline{1},\dots,\overline{n}\}$ 
(we assume throughout the proof 
that vectors of the form $\vec{v},\vec{w}\in T_{n}^{r,s}$ that we use below 
have only tensor factors from either $\{1,\dots,n\}$ or $\{\overline{1},\dots,\overline{n}\}$), then 
the action of $\cal{B}_{r,s}(\delta)$ on $T_{n}^{r,s}$ can 
locally be pictured as
\[
\xy
(0,0)*{
\begin{tikzpicture}[scale=.3]
	\draw [very thick, ->] (-1,-1) to (-1,1);
	\draw [very thick, ->] (1,-1) to (1,1);
	\node at (-1,-1.45) {\tiny $v$};
	\node at (1,-1.45) {\tiny $v$};
	\node at (-1,1.45) {\tiny $v$};
	\node at (1,1.45) {\tiny $v$};
	\node at (0,-1.45) {\tiny $\otimes$};
	\node at (0,1.45) {\tiny $\otimes$};
\end{tikzpicture}
};
\endxy,\;\;\;
\xy
(0,0)*{
\begin{tikzpicture}[scale=.3]
	\draw [very thick, ->] (-1,-1) to (-1,1);
	\draw [very thick, ->] (1,-1) to (1,1);
	\node at (-1,-1.45) {\tiny $v$};
	\node at (1,-1.45) {\tiny $w$};
	\node at (-1,1.45) {\tiny $v$};
	\node at (1,1.45) {\tiny $w$};
	\node at (0,-1.45) {\tiny $\otimes$};
	\node at (0,1.45) {\tiny $\otimes$};
\end{tikzpicture}
};
\endxy,\;\;\;
\xy
(0,0)*{
\begin{tikzpicture}[scale=.3]
	\draw [very thick, ->] (-1,-1) to (1,1);
	\draw [very thick, ->] (1,-1) to (-1,1);
	\node at (-1,-1.45) {\tiny $v$};
	\node at (1,-1.45) {\tiny $v$};
	\node at (-1,1.45) {\tiny $v$};
	\node at (1,1.45) {\tiny $v$};
	\node at (0,-1.45) {\tiny $\otimes$};
	\node at (0,1.45) {\tiny $\otimes$};
\end{tikzpicture}
};
\endxy,\;\;\;
\xy
(0,0)*{
\begin{tikzpicture}[scale=.3]
	\draw [very thick, ->] (-1,-1) to (1,1);
	\draw [very thick, ->] (1,-1) to (-1,1);
	\node at (-1,-1.45) {\tiny $v$};
	\node at (1,-1.45) {\tiny $w$};
	\node at (-1,1.45) {\tiny $w$};
	\node at (1,1.45) {\tiny $v$};
	\node at (0,-1.45) {\tiny $\otimes$};
	\node at (0,1.45) {\tiny $\otimes$};
\end{tikzpicture}
};
\endxy,\;\;\;
\xy
(0,0)*{
\begin{tikzpicture}[scale=.3]
	\draw [very thick, ->] (-1,1) to (-1,-1);
	\draw [very thick, ->] (1,1) to (1,-1);
	\node at (-1,-1.45) {\tiny $\overline{v}$};
	\node at (1,-1.45) {\tiny $\overline{v}$};
	\node at (-1,1.45) {\tiny $\overline{v}$};
	\node at (1,1.45) {\tiny $\overline{v}$};
	\node at (0,-1.45) {\tiny $\otimes$};
	\node at (0,1.45) {\tiny $\otimes$};
\end{tikzpicture}
};
\endxy,\;\;\;
\xy
(0,0)*{
\begin{tikzpicture}[scale=.3]
	\draw [very thick, ->] (-1,1) to (-1,-1);
	\draw [very thick, ->] (1,1) to (1,-1);
	\node at (-1,-1.45) {\tiny $\overline{v}$};
	\node at (1,-1.45) {\tiny $\overline{w}$};
	\node at (-1,1.45) {\tiny $\overline{v}$};
	\node at (1,1.45) {\tiny $\overline{w}$};
	\node at (0,-1.45) {\tiny $\otimes$};
	\node at (0,1.45) {\tiny $\otimes$};
\end{tikzpicture}
};
\endxy,\;\;\;
\xy
(0,0)*{
\begin{tikzpicture}[scale=.3]
	\draw [very thick, ->] (-1,1) to (1,-1);
	\draw [very thick, ->] (1,1) to (-1,-1);
	\node at (-1,-1.45) {\tiny $\overline{v}$};
	\node at (1,-1.45) {\tiny $\overline{v}$};
	\node at (-1,1.45) {\tiny $\overline{v}$};
	\node at (1,1.45) {\tiny $\overline{v}$};
	\node at (0,-1.45) {\tiny $\otimes$};
	\node at (0,1.45) {\tiny $\otimes$};
\end{tikzpicture}
};
\endxy,\;\;\;
\xy
(0,0)*{
\begin{tikzpicture}[scale=.3]
	\draw [very thick, ->] (-1,1) to (1,-1);
	\draw [very thick, ->] (1,1) to (-1,-1);
	\node at (-1,-1.45) {\tiny $\overline{v}$};
	\node at (1,-1.45) {\tiny $\overline{w}$};
	\node at (-1,1.45) {\tiny $\overline{w}$};
	\node at (1,1.45) {\tiny $\overline{v}$};
	\node at (0,-1.45) {\tiny $\otimes$};
	\node at (0,1.45) {\tiny $\otimes$};
\end{tikzpicture}
};
\endxy,\;\;\;
\raisebox{.25cm}{\xy
(0,0)*{
\begin{tikzpicture}[scale=.3]
	\draw [very thick, directed=.55] (5,-1) to [out=90,in=180] (6,-.35) to [out=0,in=90] (7,-1);
	\draw [very thick, directed=.55] (7,1) to [out=270,in=0] (6,.35) to [out=180,in=270] (5,1);
	\node at (5,-1.5) {\tiny $v$};
	\node at (7,-1.45) {\tiny $\overline{v}$};
	\node at (4,1.45) {\tiny $\displaystyle\sum_{i=1}^n$};
	\node at (5,1.45) {\tiny $i$};
	\node at (7,1.55) {\tiny $\overline{i}$};
	\node at (6,-1.45) {\tiny $\otimes$};
	\node at (6,1.45) {\tiny $\otimes$};
\end{tikzpicture}
};
\endxy}
,\;\;\;
\xy
(0,0)*{
\begin{tikzpicture}[scale=.3]
	\draw [very thick, directed=.55] (5,-1) to [out=90,in=180] (6,-.35) to [out=0,in=90] (7,-1);
	\draw [very thick, directed=.55] (7,1) to [out=270,in=0] (6,.35) to [out=180,in=270] (5,1);
	\node at (5,-1.5) {\tiny $v$};
	\node at (7,-1.45) {\tiny $\overline{w}$};
	\node at (5,1.45) {\tiny $0$};
	\node at (7,1.45) {\tiny $0$};
	\node at (6,-1.45) {\tiny $\otimes$};
	\node at (6,1.45) {\tiny $\otimes$};
\end{tikzpicture}
};
\endxy
\]
for $v,w\in\{1,\dots,n\}$ with $v\neq w$. For example, the cap-cup generator 
sends a basis vector of $V\otimes V^*$ of the form $v\otimes\overline{v}$ to the 
full sum $\sum_{i=1}^ni\otimes\overline{i}$ 
and all other basis vectors to zero.

We need to show that an 
arbitrary basis vector $\vec{v}\in T_{n}^{r,s}$ is sent to zero by $e_{r,s}(n)$. 
For this purpose, we argue inductively, where the induction is on the total 
number $d=r+s$ of strands.

In case $d=2$, we have either $r=2,s=0$ or $r=0,s=2$ or $r=1,s=1$. Moreover, 
$n=1$ and the only possible basis vectors $\vec{v}\in T_{n}^{r,s}$ in 
these cases are $1\otimes 1$ or $\overline{1}\otimes\overline{1}$ or $1\otimes\overline{1}$. 
Then
\[
e_{2,0}(1)(1\otimes 1)=
\xy
(0,0)*{
\begin{tikzpicture}[scale=.3]
	\draw [very thick, ->] (-1,-1) to (-1,1);
	\draw [very thick, ->] (1,-1) to (1,1);
	\node at (-1,-1.45) {\tiny $1$};
	\node at (1,-1.45) {\tiny $1$};
	\node at (-1,1.45) {\tiny $1$};
	\node at (1,1.45) {\tiny $1$};
	\node at (0,-1.45) {\tiny $\otimes$};
	\node at (0,1.45) {\tiny $\otimes$};
\end{tikzpicture}
};
\endxy
-
\xy
(0,0)*{
\begin{tikzpicture}[scale=.3]
	\draw [very thick, ->] (-1,-1) to (1,1);
	\draw [very thick, ->] (1,-1) to (-1,1);
	\node at (-1,-1.45) {\tiny $1$};
	\node at (1,-1.45) {\tiny $1$};
	\node at (-1,1.45) {\tiny $1$};
	\node at (1,1.45) {\tiny $1$};
	\node at (0,-1.45) {\tiny $\otimes$};
	\node at (0,1.45) {\tiny $\otimes$};
\end{tikzpicture}
};
\endxy,\;\;\;
e_{0,2}(1)(\overline{1}\otimes\overline{1})=
\xy
(0,0)*{
\begin{tikzpicture}[scale=.3]
	\draw [very thick, ->] (-1,1) to (-1,-1);
	\draw [very thick, ->] (1,1) to (1,-1);
	\node at (-1,-1.45) {\tiny $\overline{1}$};
	\node at (1,-1.45) {\tiny $\overline{1}$};
	\node at (-1,1.45) {\tiny $\overline{1}$};
	\node at (1,1.45) {\tiny $\overline{1}$};
	\node at (0,-1.45) {\tiny $\otimes$};
	\node at (0,1.45) {\tiny $\otimes$};
\end{tikzpicture}
};
\endxy
-
\xy
(0,0)*{
\begin{tikzpicture}[scale=.3]
	\draw [very thick, ->] (-1,1) to (1,-1);
	\draw [very thick, ->] (1,1) to (-1,-1);
	\node at (-1,-1.45) {\tiny $\overline{1}$};
	\node at (1,-1.45) {\tiny $\overline{1}$};
	\node at (-1,1.45) {\tiny $\overline{1}$};
	\node at (1,1.45) {\tiny $\overline{1}$};
	\node at (0,-1.45) {\tiny $\otimes$};
	\node at (0,1.45) {\tiny $\otimes$};
\end{tikzpicture}
};
\endxy,\;\;\;
e_{1,1}(1)(1\otimes\overline{1})=
\xy
(0,0)*{
\begin{tikzpicture}[scale=.3]
	\draw [very thick, ->] (-1,-1) to (-1,1);
	\draw [very thick, ->] (1,1) to (1,-1);
	\node at (-1,-1.55) {\tiny $1$};
	\node at (1,-1.45) {\tiny $\overline{1}$};
	\node at (-1,1.45) {\tiny $1$};
	\node at (1,1.55) {\tiny $\overline{1}$};
	\node at (0,-1.45) {\tiny $\otimes$};
	\node at (0,1.45) {\tiny $\otimes$};
\end{tikzpicture}
};
\endxy
-
\xy
(0,0)*{
\begin{tikzpicture}[scale=.3]
    \draw [very thick, directed=.55] (-1,-1) to [out=90,in=180] (0,-.35) to [out=0,in=90] (1,-1);
	\draw [very thick, directed=.55] (1,1) to [out=270,in=0] (0,.35) to [out=180,in=270] (-1,1);
	\node at (-1,-1.55) {\tiny $1$};
	\node at (1,-1.45) {\tiny $\overline{1}$};
	\node at (-1,1.45) {\tiny $1$};
	\node at (1,1.55) {\tiny $\overline{1}$};
	\node at (0,-1.45) {\tiny $\otimes$};
	\node at (0,1.45) {\tiny $\otimes$};
\end{tikzpicture}
};
\endxy
\]
We see that all of these act as zero on a basis of $T_{n}^{r,s}$. Hence, they are all in the kernel.

Let $d>2$ and let $\vec{v}=v_1\otimes\dots\otimes\overline{v}_{r+s}$. 
We need to show that $e_{r,s}(n)(\vec{v})=0$.
We do a case-by-case check depending on the tensor factors of 
$\vec{v}$. For 
simplicity of notation, 
we assume that those tensor factors of $\vec{v}$ 
that we consider are next to each other (otherwise, 
we can permute them next to each other) and we only display 
the relevant part of $\vec{v}$. The cases are:
\begin{itemize}
\item[i.] $\vec{v}$ has tensor factors of the form $v\otimes v$ or 
$\overline{v}\otimes\overline{v}$. Then any primitive diagram $x$ acting 
non-trivially on $\vec{v}$ is locally of the following form.
\[
\text{Case }v\otimes v:
\xy
(0,0)*{
\begin{tikzpicture}[scale=.3]
	\draw [very thick, ->] (-1,-1) to (-1,1);
	\draw [very thick, ->] (1,-1) to (1,1);
	\node at (-1,-1.45) {\tiny $v$};
	\node at (1,-1.45) {\tiny $v$};
	\node at (-1,1.45) {\tiny $v$};
	\node at (1,1.45) {\tiny $v$};
	\node at (0,-1.45) {\tiny $\otimes$};
	\node at (0,1.45) {\tiny $\otimes$};
\end{tikzpicture}
};
\endxy
\quad\text{or}\quad
\xy
(0,0)*{
\begin{tikzpicture}[scale=.3]
	\draw [very thick, ->] (-1,-1) to (1,1);
	\draw [very thick, ->] (1,-1) to (-1,1);
	\node at (-1,-1.45) {\tiny $v$};
	\node at (1,-1.45) {\tiny $v$};
	\node at (-1,1.45) {\tiny $v$};
	\node at (1,1.45) {\tiny $v$};
	\node at (0,-1.45) {\tiny $\otimes$};
	\node at (0,1.45) {\tiny $\otimes$};
\end{tikzpicture}
};
\endxy;\quad\quad\text{Case }\overline{v}\otimes\overline{v}:
\quad\quad
\xy
(0,0)*{
\begin{tikzpicture}[scale=.3]
	\draw [very thick, ->] (-1,1) to (-1,-1);
	\draw [very thick, ->] (1,1) to (1,-1);
	\node at (-1,-1.45) {\tiny $\overline{v}$};
	\node at (1,-1.45) {\tiny $\overline{v}$};
	\node at (-1,1.45) {\tiny $\overline{v}$};
	\node at (1,1.45) {\tiny $\overline{v}$};
	\node at (0,-1.45) {\tiny $\otimes$};
	\node at (0,1.45) {\tiny $\otimes$};
\end{tikzpicture}
};
\endxy
\quad\text{or}\quad
\xy
(0,0)*{
\begin{tikzpicture}[scale=.3]
	\draw [very thick, ->] (-1,1) to (1,-1);
	\draw [very thick, ->] (1,1) to (-1,-1);
	\node at (-1,-1.45) {\tiny $\overline{v}$};
	\node at (1,-1.45) {\tiny $\overline{v}$};
	\node at (-1,1.45) {\tiny $\overline{v}$};
	\node at (1,1.45) {\tiny $\overline{v}$};
	\node at (0,-1.45) {\tiny $\otimes$};
	\node at (0,1.45) {\tiny $\otimes$};
\end{tikzpicture}
};
\endxy
\]
Note that, for each primitive diagram $x\in\cal{B}_{r,s}(\delta)$ 
that is locally as on the left-hand sides above, there 
is precisely one primitive diagram $\tilde x\in\cal{B}_{r,s}(\delta)$ that 
is locally as on the right-hand sides above and otherwise equal to $x$. These 
appear in $e_{r,s}(n)$ with different signs and their contributions cancel. 
This shows that $e_{r,s}(n)(\vec{v})=0$.
\item[ii.] $\vec{v}$ has no entry pairs of the form $v\otimes v$ 
or $\overline{v}\otimes\overline{v}$. We fix a primitive diagram $x$ and do another 
case-by-case check depending on the
matrix entry corresponding to a fixed pair $\vec{v}$ and $\vec{w}=x(\vec{v})$.
We again assume that the tensor factors of $\vec{w}$ 
under consideration are next to each other.
\begin{itemize}
\item $\vec{w}$ has no tensor factors of the form $w\otimes w$ or 
$\overline{w}\otimes\overline{w}$. Then $\vec{w}$ (that is the contribution of $x$) 
is cancelled by a primitive diagram $\tilde x$ obtained from $x$ by applying 
an extra crossing at the corresponding position. Or in pictures 
(for brevity, we only display the upwards oriented version, but the 
other case is completely similar):
\[
x=
\xy
(0,0)*{
\begin{tikzpicture}[scale=.3]
	\draw [very thick] (-2,-8) to (-2,-6);
	\draw [very thick] (0,-8) to (0,-6);
	\draw [very thick] (2,-8) to (2,-6);
	\draw [very thick] (4,-8) to (4,-6);
	\draw [very thick, ->] (-2,-4) to (-2,-2);
	\draw [very thick, ->] (0,-4) to (0,-2);
	\draw [very thick, ->] (2,-4) to (2,-2);
	\draw [very thick, ->] (4,-4) to (4,-2);
	\draw [very thick] (6,-2) to (6,-4);
	\draw [very thick] (8,-2) to (8,-4);
	\draw [very thick, ->] (6,-6) to (6,-8);
	\draw [very thick, ->] (8,-6) to (8,-8);
	\draw [very thick, blue] (-2.3,-6) rectangle (8.3,-4);
	\node at (-2,-8.45) {\tiny $*$};
	\node at (0,-8.45) {\tiny $*$};
	\node at (2,-8.45) {\tiny $*$};
	\node at (4,-8.45) {\tiny $*$};
	\node at (6,-8.45) {\tiny $*$};
	\node at (8,-8.45) {\tiny $*$};
	\node at (-2,-1.55) {\tiny $*$};
	\node at (0,-1.55) {\tiny $w$};
	\node at (2,-1.55) {\tiny $w$};
	\node at (4,-1.55) {\tiny $*$};
	\node at (6,-1.55) {\tiny $*$};
	\node at (8,-1.55) {\tiny $*$};
	\node at (-.9,-3.25) {\tiny $\cdots$};
	\node at (-.9,-7) {\tiny $\cdots$};
	\node at (3.1,-3.25) {\tiny $\cdots$};
	\node at (7.1,-3.25) {\tiny $\cdots$};
	\node at (3.1,-7) {\tiny $\cdots$};
	\node at (7.1,-7) {\tiny $\cdots$};
	\node at (3,-5) {\tiny $x$};
	\node at (9.1,-8.35) {\tiny $=\vec{v}$};
	\node at (9.2,-1.45) {\tiny $=\vec{w}$};
\end{tikzpicture}
};
\endxy,\quad\quad
\tilde x=
\xy
(0,0)*{
\begin{tikzpicture}[scale=.3]
	\draw [very thick] (-2,-8) to (-2,-6);
	\draw [very thick] (0,-8) to (0,-6);
	\draw [very thick] (2,-8) to (2,-6);
	\draw [very thick] (4,-8) to (4,-6);
	\draw [very thick, ->] (-2,-4) to (-2,-2);
	\draw [very thick, ->] (0,-4) to (2,-2);
	\draw [very thick, ->] (2,-4) to (0,-2);
	\draw [very thick, ->] (4,-4) to (4,-2);
	\draw [very thick] (6,-2) to (6,-4);
	\draw [very thick] (8,-2) to (8,-4);
	\draw [very thick, ->] (6,-6) to (6,-8);
	\draw [very thick, ->] (8,-6) to (8,-8);
	\draw [very thick, blue] (-2.3,-6) rectangle (8.3,-4);
	\node at (-2,-8.45) {\tiny $*$};
	\node at (0,-8.45) {\tiny $*$};
	\node at (2,-8.45) {\tiny $*$};
	\node at (4,-8.45) {\tiny $*$};
	\node at (6,-8.45) {\tiny $*$};
	\node at (8,-8.45) {\tiny $*$};
	\node at (-2,-1.55) {\tiny $*$};
	\node at (0,-1.55) {\tiny $w$};
	\node at (2,-1.55) {\tiny $w$};
	\node at (4,-1.55) {\tiny $*$};
	\node at (6,-1.55) {\tiny $*$};
	\node at (8,-1.55) {\tiny $*$};
	\node at (-.9,-3.25) {\tiny $\cdots$};
	\node at (-.9,-7) {\tiny $\cdots$};
	\node at (3.1,-3.25) {\tiny $\cdots$};
	\node at (7.1,-3.25) {\tiny $\cdots$};
	\node at (3.1,-7) {\tiny $\cdots$};
	\node at (7.1,-7) {\tiny $\cdots$};
	\node at (3,-5) {\tiny $x$};
	\node at (9.1,-8.35) {\tiny $=\vec{v}$};
	\node at (9.2,-1.45) {\tiny $=\vec{w}$};
\end{tikzpicture}
};
\endxy
\]
Here the $*$'s represent arbitrary tensor factors (which are the same for $x$ and $\tilde x$). 
Since $x$ and $\tilde x$ appear with different signs in $e_{r,s}(n)$, these two terms 
cancel 
each other.
\item There is a tensor factor $w$ (or $\overline{w}$) 
of $\vec{w}$ that appears isolated, that is no other 
tensor factors of $\vec{w}$ are of the form $w$ or $\overline{w}$. 
Since $\vec{w}=x(\vec{v})$ 
is non-zero and we are not in case i., there exist a unique 
connecting strand in $x$ from a bottom entry $w$ (or $\overline{w}$) to this isolated 
top entry. In pictures 
(where we for simplicity assume that this unique strand is on the left respectively right)
\[
\xy
(0,0)*{
\begin{tikzpicture}[scale=.3]
	\draw [very thick] (0,-8) to (0,-6);
	\draw [very thick] (2,-8) to (2,-6);
	\draw [very thick, ->] (4,-6) to (4,-8);
	\draw [very thick, ->] (-2,-8) to (-2,-2);
	\draw [very thick, ->] (0,-4) to (0,-2);
	\draw [very thick, ->] (2,-4) to (2,-2);
	\draw [very thick] (4,-4) to (4,-2);
	\draw [very thick] (6,-2) to (6,-4);
	\draw [very thick, ->] (6,-6) to (6,-8);
	\draw [very thick, blue] (-0.3,-6) rectangle (6.3,-4);
	\node at (-2,-8.45) {\tiny $w$};
	\node at (0,-8.45) {\tiny $*$};
	\node at (2,-8.45) {\tiny $*$};
	\node at (4,-8.45) {\tiny $*$};
	\node at (6,-8.45) {\tiny $*$};
	\node at (-2,-1.55) {\tiny $w$};
	\node at (0,-1.55) {\tiny $*$};
	\node at (2,-1.55) {\tiny $*$};
	\node at (4,-1.55) {\tiny $*$};
	\node at (6,-1.55) {\tiny $*$};
	\node at (1.1,-3.25) {\tiny $\cdots$};
	\node at (5.1,-3.25) {\tiny $\cdots$};
	\node at (1.1,-7) {\tiny $\cdots$};
	\node at (5.1,-7) {\tiny $\cdots$};
	\node at (3,-5) {\tiny rest of $x$};
	\node at (7.1,-8.35) {\tiny $=\vec{v}$};
	\node at (7.2,-1.45) {\tiny $=\vec{w}$};
\end{tikzpicture}
};
\endxy
\quad\text{or}\quad
\xy
(0,0)*{
\begin{tikzpicture}[scale=.3]
    \draw [very thick] (-2,-8) to (-2,-6);
	\draw [very thick] (0,-8) to (0,-6);
	\draw [very thick, ->] (2,-6) to (2,-8);
	\draw [very thick, ->] (4,-6) to (4,-8);
	\draw [very thick, ->] (-2,-4) to (-2,-2);
	\draw [very thick, ->] (0,-4) to (0,-2);
	\draw [very thick] (2,-4) to (2,-2);
	\draw [very thick] (4,-4) to (4,-2);
	\draw [very thick, ->] (6,-2) to (6,-8);
	\draw [very thick, blue] (-2.3,-6) rectangle (4.3,-4);
	\node at (-2,-8.45) {\tiny $*$};
	\node at (0,-8.45) {\tiny $*$};
	\node at (2,-8.45) {\tiny $*$};
	\node at (4,-8.45) {\tiny $*$};
	\node at (6,-8.45) {\tiny $\overline{w}$};
	\node at (-2,-1.55) {\tiny $*$};
	\node at (0,-1.55) {\tiny $*$};
	\node at (2,-1.55) {\tiny $*$};
	\node at (4,-1.55) {\tiny $*$};
	\node at (6,-1.55) {\tiny $\overline{w}$};
	\node at (-.9,-3.25) {\tiny $\cdots$};
	\node at (3.1,-3.25) {\tiny $\cdots$};
	\node at (-.9,-7) {\tiny $\cdots$};
	\node at (3.1,-7) {\tiny $\cdots$};
	\node at (1,-5) {\tiny rest of $x$};
	\node at (7.1,-8.35) {\tiny $=\vec{v}$};
	\node at (7.2,-1.45) {\tiny $=\vec{w}$};
\end{tikzpicture}
};
\endxy
\]
Here the entries $*$ mean arbitrary tensor factors 
that are neither $w$ nor $\overline{w}$.
Now $e_{r,s}(n)(\vec{v})=0$ if and only if $e^{\prime}_{r,s}(n)(\vec{v}^{\prime})=0$, 
where $\vec{v}^{\prime}$ is obtained from $\vec{v}$ by removing 
the two isolated tensor factors and $e^{\prime}_{r,s}(n)$ is obtained from 
$e_{r,s}(n)$ by first removing all summands which are not of the form as above and then remove 
the unique strand. Hence, we can argue now by induction.

\item $\vec{w}$ has only entry pairs of the form $w\otimes\overline{w}$. Then 
the same is true for $\vec{v}$ (otherwise we are in case i.). Since $n=r+s-1$, 
we know that there is at least one pair $i\otimes\overline{i}$ that 
appears in both $\vec{v}$ and $\vec{w}$. Similarly to the case i., the 
primitive diagram $x$ is locally of the following form.
\[
\xy
(0,0)*{
\begin{tikzpicture}[scale=.3]
	\draw [very thick, ->] (-1,-1) to (-1,1);
	\draw [very thick, ->] (1,1) to (1,-1);
	\node at (-1,-1.55) {\tiny $i$};
	\node at (1,-1.45) {\tiny $\overline{i}$};
	\node at (-1,1.45) {\tiny $i$};
	\node at (1,1.55) {\tiny $\overline{i}$};
	\node at (0,-1.45) {\tiny $\otimes$};
	\node at (0,1.45) {\tiny $\otimes$};
\end{tikzpicture}
};
\endxy
\quad\text{or}\quad
\xy
(0,0)*{
\begin{tikzpicture}[scale=.3]
	\draw [very thick, directed=.55] (-1,-1) to [out=90,in=180] (0,-.35) to [out=0,in=90] (1,-1);
	\draw [very thick, directed=.55] (1,1) to [out=270,in=0] (0,.35) to [out=180,in=270] (-1,1);
	\node at (-1,-1.55) {\tiny $i$};
	\node at (1,-1.45) {\tiny $\overline{i}$};
	\node at (-1,1.45) {\tiny $i$};
	\node at (1,1.55) {\tiny $\overline{i}$};
	\node at (0,-1.45) {\tiny $\otimes$};
	\node at (0,1.45) {\tiny $\otimes$};
\end{tikzpicture}
};
\endxy
\]
Thus, for each such $x$, there is precisely one $\tilde x$ which is locally 
different from $x$ as illustrated above and identical to $x$ otherwise. Since $x$ and 
$\tilde x$ appear with different signs in $e_{r,s}(n)$, their contributions cancel.
\end{itemize}
\end{itemize}
These are all possible cases. Hence, all matrix coefficients of 
$e_{r,s}(n)\in\End_{\Uo}(T^d_n)$ are trivial and so Claim 1 follows.

\noindent\textit{Claim 2.} The $\K$-linear span of $e_{r,s}(n)$ equals $\ker(\Phi_{\mathrm{wBr}})$.

\textit{Proof of Claim 2.} 
By Theorem~\ref{thm-schur-weyl2} we see that $\cal{B}_{r,s}(\delta)$ 
surjects onto $\End_{\Uo}(T_{n}^{r,s})$. The dimension of $\End_{\Uo}(T_{n}^{r,s})$ is independent 
of $\K$ by Proposition~\ref{prop-prop}.
Thus, we may assume $\K=\C$ 
to calculate the dimension. 
Now $\dim(\End_{\Uo}(T_{n}^{r,s}))=\dim(\cal{B}_{r,s}(\delta))-1$ 
by part (c) of Theorem~\ref{thm-schur-weyl2}: for $n=r+s-1$ only 
the pair of Young diagrams with maximal numbers of columns is missing in the 
direct sum decomposition and the missing simple $\cal{B}_{r,s}(\delta)$-module has dimension one 
(in the semisimple case: $D_1^{\lambda,\overline{\mu}}$ has a basis 
parametrized by so-called up-down tableaux, see for example~\cite[Section~6]{bs5}). Hence, $\dim(\ker(\Phi_{\mathrm{wBr}}))=1$ independent of $\K$.
Since 
$0\neq e_{r,s}(n)\in\ker(\Phi_{\mathrm{wBr}})$ (by Claim 1), 
its $\K$-linear span equals $\ker(\Phi_{\mathrm{wBr}})$.

By the Claims 1 and 2 we have 
$e_{r,s}(n)\in\ker(\Phi_{\mathrm{wBr}})$ 
and $\dim(\ker(\Phi_{\mathrm{wBr}}))=1$. Thus, 
$e_{r,s}(n)e_{r,s}(n)=ae_{r,s}(n)$ for some 
$a\in\K$. A direct computation 
shows that the scalar in front of the identity diagram of 
$e_{r,s}(n)e_{r,s}(n)$ is $r!s!$. Thus, we can divide 
by this value to get an `honest' idempotent 
if and only if $\Char(\K)>\max\{r,s\}$ or $\Char(\K)=0$.
\end{proof}

\begin{rem}\label{rem-onedim}
That $\dim(\ker(\Phi_{\mathrm{wBr}}))$ is 
independent of $\K$ was 
already obtained using quite different methods in~\cite[Corollary~7.2]{dds}. The 
explicit description of 
$\ker(\Phi_{\mathrm{wBr}})$ from 
Proposition~\ref{prop-kernel} seems to be new, but is 
already implicitly
contained in~\cite[Section~8]{bs5}.
\end{rem}

\subsection{The Brauer case}

\begin{prop}\label{prop-kernelbrauer}
Let $n=2d-2=-\delta$. Then the $\K$-linear span 
of $E_d(n)\in\cal{B}_{d}(\delta)$ equals $\ker(\Phi_{\mathrm{Br}})$.
Furthermore, the following holds.
\begin{enumerate}[(a)]
\item If $\Char(\K)>d$ or $\Char(\K)=0$, 
then $E_d(n)$ is a quasi-idempotent.
\item If $\Char(\K)=p\leq d$, then 
$E_d(n)$ is nilpotent.
\end{enumerate}
\end{prop}

\begin{proof}
We can argue mutatis mutandis as in 
the proof of Proposition~\ref{prop-kernel}. To be more precise, 
the analogue of Claim 1 works almost word-by-word as 
in the walled Brauer case.
For the analogue of 
Claim 2 we use part (c) of Theorem~\ref{thm-schur-weyl3}, where 
the only summand missing in the $(\Uo,\cal{B}_d(\delta))$-bimodule decomposition 
is the one for the unique Young diagram with maximal number of columns (the 
corresponding simple $\cal{B}_d(\delta)$-module is one dimensional 
which can be deduced from~\cite[Section~4]{gl}). 
For the analogue of the 
proof of (a) and (b) we note that the scalar in front of the identity diagram of 
$E_{d}(n)E_{d}(n)$ can be easily seen to be $d!$.
\end{proof}

\begin{rem}\label{rem-onedimagain}
Lehrer and Zhang describe $\ker(\Phi_{\mathrm{Br}})$ 
for all $d\in\Zg$ and all $\delta\in\Z$ in the case 
$\K=\C$ in~\cite[Theorem~4.3]{lz2}. In particular, they show 
that $\ker(\Phi_{\mathrm{Br}})$ 
is generated (as an ideal) by an idempotent. They also argue 
in~\cite[Proposition~9.2]{lz2} how their 
results generalize to the case of arbitrary $\K$, but with a less 
explicit description as we give above.
\end{rem}

\subsection{Application to semisimplicity}
The description of the kernels will be an important tool in the proof of the semisimplicity 
criteria (see Theorems~\ref{thm-walledbrauer} 
and~\ref{thm-brauer}), because of the following.

\begin{prop}\label{prop-abstract-nonsense}
Let $\cal{A}_1$ and $\cal{A}_2$ be algebras over $\K$ and let 
$\Phi\colon\cal{A}_1\to\cal{A}_2$ be a surjective algebra homomorphism such that $\ker(\Phi)$ 
is spanned as a $\K$-vector space by an idempotent $e$. Then semisimplicity of $\cal{A}_2$ 
implies semisimplicity of $\cal{A}_1$.
\end{prop}

\begin{proof}
Clearly $\cal{A}_1/\ker(\Phi)\cong\cal{A}_2$ as algebras.
Now, for any algebra $\cal{A}$ with ideal $I$ such 
that $\cal{A}/I$ is semisimple, we have $I\supset\mathrm{Rad}(\cal{A})$ 
(here $\mathrm{Rad}(\cal{A})$ 
means the Jacobson radical of $\cal{A}$). 
Assuming that $\cal{A}_2$ is semisimple, we have
$\mathrm{span}_{\K}(\{e\})=\ker(\Phi)\supset\mathrm{Rad}(\cal{A}_1)$. 
Since $e$ is idempotent it follows that $\mathrm{Rad}(\cal{A}_1)=0$.
\end{proof}

%% file: res/hecke.tex
\begin{thm}(\textbf{Semisimplicity criteria for the Hecke algebras of types $\Am$ and $\Bm$})\label{thm-heckesemi}
$\cal{H}^{\Am}_d(q)$ and $\cal{H}^{\Bm}_d(q)$
are semisimple if and only if one of the following conditions hold:
\begin{enumerate}
\item $\Char(\K)>d$ and $q=1$.
\item $\Char(\K)=0$ and $q=1$.
\item $q\in\K^{\ast},q\neq 1$ is a root of unity with $\ord(q^2)>d$.
\item $q\in\K^{\ast},q\neq 1$ is a non-root of unity.
\end{enumerate}
\end{thm}

The proof of Theorem~\ref{thm-heckesemi} requires some preparation.

\subsection{The Schur-Weyl dual story}

Let $\Uq=\Uu_q(\gll{m})$, $V$ 
and $T_{n}^{d}=V^{\otimes d}$ be as before in Theorem~\ref{thm-schur-weyl1}. 
Note that $V$ corresponds to a Young diagram with precisely one node 
(via Conventions~\ref{nota-young}). Thus, by 
Proposition~\ref{prop-classchar}, we can use the classical 
Littlewood-Richardson rule to see that a Weyl factor
$\Dl$ appears in a 
$\Delta_q$-filtration of $T_{n}^{d}$ if and only if $\lambda\in\Lambda^+(d)$ 
(hence, the Young diagram associated to $\lambda$ has $d$ nodes). 
Note that $V\in\T$ 
and hence, also $T_{n}^{d}\in\T$ by Proposition~\ref{prop-prop}.
Thus, by Lemma~\ref{lem-cellsemisimple}, the semisimplicity 
of $T_{n}^{d}$ is equivalent to the condition that 
all of its occurring Weyl factors $\Dl$ 
are simple $\Uq$-modules.

\begin{prop}\label{prop-heckesemi}
We have the following.
\begin{enumerate}[(a)]
\item Let $\Char(\K)>0$ and $q=1$. Then $T_{n}^{d}$ is a semisimple $\Uo$-module if and only if $\Char(\K)>d$.
\item Let $\Char(\K)=0$ and $q=1$. Then $T_{n}^{d}$ is always a semisimple $\Uo$-module.
\item Let $q\in\K^{\ast}$ be a root of unity. 
Then $T_{n}^{d}$ is a semisimple $\Uq$-module if and only if $\ord(q^2)\!>\!d$.
\item Let $q\in\K^{\ast}$ be a non-root of unity.  
Then $T_{n}^{d}$ is always a semisimple $\Uq$-module.
\end{enumerate}
\end{prop}

\begin{proof}
\textit{`If' of (a).} 
Let $\Char(\K)=p>d$, $q=1$ and $\lambda\in\Lambda^+(d)$.
For the positive roots $\alpha\in\Phi^+$ 
of the form $\alpha=\varepsilon_i-\varepsilon_j$ with $1\leq i<j\leq n$ we obtain
\[
\langle\lambda+\rho,(\varepsilon_i-\varepsilon_j)^{\vee}\rangle=n-i+\lambda_i-(n-j+\lambda_j)=j-i+\lambda_i-\lambda_j\leq n+d<n+p.
\]
Hence, \textbf{JSF} from~\eqref{eq-jsumc} for $\Delta_1(\lambda)$ gives
\begin{equation}\label{eq-jantzen}
\sum_{k^{\prime}\geq 1}\mathrm{ch}(\Delta_{1}^{k^{\prime}}(\lambda))=-\sum_{i<j}\sum_{k\in\Zg}v_p(kp)\chi(\lambda-kp(\varepsilon_i-\varepsilon_j)),
\end{equation}
where the right-hand sum runs over all $1\leq i<j\leq n$ and $k\in\Zg$ such that 
$kp<j-i+\lambda_i-\lambda_j$. We claim that the sum in~\eqref{eq-jantzen} is zero.

For this purpose,
fix $1\leq i<j\leq n$ and $k\in\Zg$ and assume that 
$\chi(\lambda-kp(\varepsilon_i-\varepsilon_j))$ 
appears on the right-hand side in~\eqref{eq-jantzen}. We first 
note that $\lambda_j=0$: if $\lambda_j>0$, 
then the Young diagram of $\lambda$ 
contains at least $j-1+\lambda_i$ nodes, that is 
$j-1+\lambda_i\leq d$. But then also 
$j-i+\lambda_i-\lambda_j\leq d<p\leq kp$ and 
$\chi(\lambda-kp(\varepsilon_i-\varepsilon_j))$ does not 
occur in~\eqref{eq-jantzen}, which gives a contradiction.
So we have $kp<j-i+\lambda_i$.

Moreover, 
$(\lambda+\rho-kp(\varepsilon_i-\varepsilon_j))_i=\lambda_i+n-i-kp$. 
Note that $i<i-\lambda_i+kp<j$: the left inequality 
follows from $d<p$, while the right follows from $kp<j-i+\lambda_i$. 
Furthermore, the $i^{\prime}=i-\lambda_i+kp$th 
coordinate of $\lambda+\rho-kp(\varepsilon_i-\varepsilon_j)$ is 
$n-i+\lambda_i-kp$ (note that $\lambda_{i^{\prime}}=0$: as above, $\lambda_{i^{\prime}}>0$ 
would imply $i^{\prime}-1+\lambda_i=i-1+kp\leq d<p\leq kp$, which is clearly impossible).
Thus, it equals the $i$th coordinate 
of $\lambda+\rho-kp(\varepsilon_i-\varepsilon_j)$. Set 
$\mu=\lambda+\rho-kp(\varepsilon_i-\varepsilon_j)$. Then
\[
\mu=(\mu_1,\dots,\mu_{i-1},\colorbox{mycolor}{$\lambda_i+n-i+kp$},\mu_{i+1},
\dots,\mu_{i^{\prime}-1},\colorbox{mycolor}{$\lambda_i+n-i+kp$},\mu_{i^{\prime}+1},\dots,\mu_j,\dots,\mu_n).
\] 
Thus, $\mu$ is a singular $\Uq$-weight. This, by~\eqref{eq-cancel}, implies $\chi(\lambda-kp(\varepsilon_i-\varepsilon_j))=0$.

Altogether, we have proved that the 
right-hand side of~\eqref{eq-jantzen} is zero. Hence, 
$\Delta_1(\lambda)$ is a simple 
$\Uo$-module by Theorem~\ref{thm-jsum} (for all $\lambda\in\Lambda^+(d)$), which 
shows the `if' part of (a).

\noindent\textit{`Only if' of (a).} 
By the above observation, $T_{n}^{d}$ has Weyl factors which are of the form 
$\Delta_1(d\varepsilon_1)$ and $\Delta_1((d-1)\varepsilon_1+\varepsilon_2)$.
If we have $\Char(\K)=p\leq d$ and $q=1$, then either
$\Delta_1(d\varepsilon_1)$ or $\Delta_1((d-1)\varepsilon_1+\varepsilon_2)$ is 
a non-simple $\Uo$-module.
To see this, 
we use \textbf{JSF} and calculate
\[
d\varepsilon_1+\rho-p(\varepsilon_1-\varepsilon_2)=
(\colorbox{colormy}{\color{white}$d+n-1-p$},\colorbox{colormy}{\color{white}$n-2+p$},n-3,\dots,2,1,0).
\]
Since $p\leq d$ implies $d+n-1-p\neq n-j$ for $j\geq 2$ and clearly $n-2+p\neq n-j$, 
this gives a non-trivial contribution (in the case where $d+n-1-p\neq n-2+p$; otherwise take 
$(d-1)\varepsilon_1+\varepsilon_2$ instead of $d\varepsilon_1$) due to the fact that 
cancellation do not occur in type $\AM$, see Remark~\ref{rem-nocan}. 
Alternatively, one can use $\sll{2}$-theory, where 
the combinatorics in the $\sll{2}$ case is 
as in~\cite[Proposition~2.20]{at}.
Thus, we see that the `only if' part holds true in (a).

\noindent\textit{(b).} This follows from the fact that the category of $\Uo$-modules
is semisimple in the classical case, see for example Remark~\ref{rem-roots}.

\noindent\textit{(c).} \textit{Mutatis mutandis} as in 
the proof of (a): we use \textbf{JSF} from~\eqref{eq-jsuma} 
or~\eqref{eq-jsumb} (replacing $p$ by $\ord(q^2)=\ell$) and 
then the same arguments as in (a) work.

\noindent\textit{(d).} This follows again directly 
from the semisimplicity of the
corresponding categories of $\Uq$-modules, 
see for example Remark~\ref{rem-roots}.

We have proved the proposition.
\end{proof}

\subsection{Proof of the semisimplicity criterion for \texorpdfstring{$\cal{H}^{\Am}_d(q)$}{HA_d(q)} and \texorpdfstring{$\cal{H}^{\Bm}_d(q)$}{HB_d(q)}}

\begin{proof}[Proof of Theorem~\ref{thm-heckesemi}]
\textit{Case $\cal{H}^{\Am}_d(q)$.} We choose $n\geq d$ and the 
conclusion follows from Theorem~\ref{thm-schur-weyl1} together with 
Corollary~\ref{cor-cellsemisimple} and Proposition~\ref{prop-heckesemi}.

\noindent\textit{Case $\cal{H}^{\Bm}_d(q)$.} By Theorem~\ref{thm-schur-weylb}, we 
choose ${\textstyle \frac{1}{2}}n\geq d$. 
We consider the 
Schur-Weyl dual situation with $\Uq=\Uq(\gll{m}\oplus\gll{m})$ 
acting on $T_{n}^{d}$. Note that 
$V=\Delta_q(\omega_1,0)\oplus\Delta_q(0,\omega_1)$. 
Hence, $\Delta_q(\lambda,\mu)$ 
is a Weyl factor of $T_{n}^{d}$ if and only if 
$(\lambda,\mu)\in\Lambda^+(d_1)\times\Lambda^+(d_2)$ 
with $d_1+d_2=d$. 
Moreover, $\Delta_q(\lambda,\mu)$ is a simple $\Uq$-module 
if and only if $\Delta_q(\lambda)$ and $\Delta_q(\mu)$ are simple $\Uq(\gll{m})$-modules. 
Thus, by Corollary~\ref{cor-cellsemisimple}, 
$\cal{H}^{\Bm}_d(q)$ is semisimple if and only if 
$\Delta_q(\lambda)$ is a simple $\Uq(\gll{m})$-module 
for all $\lambda\in\Lambda^+(d^{\prime})$ 
with $1\leq d^{\prime}\leq d$. By 
Theorem~\ref{thm-heckesemi}, this is precisely 
the case when $\cal{H}^{\Am}_{d^{\prime}}(q)$ is semisimple 
for all $1\leq d^{\prime}\leq d$ 
which, by Theorem~\ref{thm-heckesemi} again, is equivalent to the 
semisimplicity of $\cal{H}^{\Am}_{d}(q)$.

The criterion follows.
\end{proof}

\begin{rem}\label{rem-semi1}
These semisimplicity criteria are not new, but 
were found using different methods: for $\cal{H}^{\Am}_d(q)$ 
it can be deduced from the work of Gyoja and Uno~\cite{gu} (they work over $\C$, but their 
arguments can be generalized to any field $\K$, 
see also~\cite[Page~12, Exercise~10]{mathas}).
For $\cal{H}_d^{\Bm}(q)$ it was 
first found in~\cite[Theorem~5.5]{dj}.
\end{rem}

\begin{rem}\label{rem-akalgebra}
Similar as in the case of $\cal{H}_d^{\Bm}(q)$, one could also 
prove semisimplicity criteria for Ariki-Koike algebras using 
the Schur-Weyl dualities mentioned in 
Remark~\ref{rem-arikikoike} (of course, these criteria are known, see for example~\cite[Main~Theorem]{ar}, but they again fit into the same framework). 
For brevity and to avoid 
some technicalities, we do not discuss this in more detail here. We point out 
that the \textbf{JSF} (in the related, but slightly different framework of 
cyclotomic $q$-Schur algebras) was already successfully applied in~\cite{lm} 
in the study of blocks of Ariki-Koike algebras.
\end{rem}

\begin{rem}\label{rem-spider}
Our methods also apply to tensor products of arbitrary 
fundamental representations. For example, 
given $\vec{k}=(k_1,\dots,k_d)$ with $k_i\in\{1,\dots,n-1\}$, we could consider
algebras of the form 
$\End_{\Uq}(T_{n}^{\vec{k}})=\End_{\Uq}(\Delta_q(\omega_{k_1})\otimes\dots\otimes\Delta_q(\omega_{k_d}))$.
These algebras are known as \textit{spider algebras} 
in the sense of Kuperberg~\cite{kup}. The semisimplicity 
criterion of $\End_{\Uq}(T_{n}^{\vec{k}})$ is not known, but it should 
be possible to deduce it from our setup.
\end{rem}

%% file: res/walled-brauer.tex
For the whole section let $r,s\in\N$, not both zero. Choose 
$\delta$ and $\delta_p$ (recalling $\delta_0=|\delta|$) in accordance with 
Conventions~\ref{nota-delta}.

\begin{thm}(\textbf{Semisimplicity criterion for the walled Brauer algebra})\label{thm-walledbrauer}
\newline
$\cal{B}_{r,s}(\delta)$ 
is semisimple if and only if one of the following conditions hold:
\begin{enumerate}
\item $\delta_p\neq 0$, $\Char(\K)=p$ and $r+s\leq\min\{\delta_p+1,p-\delta_p+1\}$.
\item $\delta_0\neq 0$, $\Char(\K)=0$ and $r+s\leq\delta_0+1$.
\item $\delta_p=0$, $\Char(\K)=p\geq 5$ and $(r,s)\in\{(2,1),(1,2),(3,1),(1,3)\}\cup\{(a,0),(0,a)\mid a<p\}$.
\item $\delta_3=0$, $\Char(\K)=3$ and $(r,s)\in\{(2,1),(1,2)\}\cup\{(a,0),(0,a)\mid a<3\}$.
\item $\delta_2=0$, $\Char(\K)=2$ and $(r,s)\in\{(1,0),(0,1)\}$.
\item $\delta_0=0$, $\Char(\K)=0$ and $(r,s)\in\{(2,1),(1,2),(3,1),(1,3)\}\cup\{(a,0),(0,a)\mid a\in\Zg\}$.
\end{enumerate}
\end{thm}

The proof of Theorem~\ref{thm-walledbrauer} again requires some preparation and is 
split into several lemmas.

\subsection{The Schur-Weyl dual story: from $(r,s)$ to $(r+1,s+1)$}

Let $\Uo=\Uo(\gll{m})$, $V$ and
$T_{n}^{r,s}$ be as in Theorem~\ref{thm-schur-weyl2}.  
As before, $V,V^*\in\T$ and 
so is $T_{n}^{r,s}$ by Proposition~\ref{prop-prop}.
Recall that we can 
calculate the Weyl factors of $T_{n}^{r,s}$ as in the classical case.

\begin{prop}\label{prop-lem1}
If $T_{n}^{r,s}$ is a non-semisimple $\Uo$-module, 
then so is $T_{n}^{r+1,s+1}$.
\end{prop}

\begin{proof}
A direct computation shows that $\Delta_1(0)\cong\K$ is a Weyl factor of $T_{n}^{1,1}$. 
Because of this and $T_{n}^{r+1,s+1}\cong T_{n}^{r,s}\otimes T_{n}^{1,1}$, we have that 
any Weyl factor of $T_{n}^{r,s}$ is also a Weyl factor of $T_{n}^{r+1,s+1}$. 
The conclusion follows then from Lemma~\ref{lem-cellsemisimple}.
\end{proof}

\begin{cor}\label{cor-picture2}
Let $\Char(\K)=p$. Then
$\cal{B}_{r,s}(\delta)$ is semisimple if and only if 
$\cal{B}_{s,r}(\delta)$ is semisimple.
\end{cor}

\begin{proof}
Note that $(T_{n}^{r,s})^*\cong T_{n}^{s,r}$ as $\Uo$-modules. Thus, 
$T_{n}^{r,s}$ is a semisimple $\Uo$-module if and only if $T_{n}^{s,r}$ 
is a semisimple $\Uo$-module. Choose $n\geq r+s+2$ with 
$n\equiv \delta_p\;\text{mod}\;p$. Then the statement 
follows directly from Theorem~\ref{thm-schur-weyl2}, Proposition~\ref{prop-lem1} 
and Theorem~\ref{thm-cellsemisimple}.
\end{proof}

\begin{cor}\label{cor-picture}
Let $\Char(\K)=p$. If 
$\cal{B}_{r,s}(\delta)$ is non-semisimple, then so is $\cal{B}_{r+1,s+1}(\delta)$.
\end{cor}

\begin{proof}
We can then proceed 
similarly as in the proof of Corollary~\ref{cor-picture2}.
\end{proof}

Corollary~\ref{cor-picture} fails for $\Char(\K)=0$, because 
we can not choose $n$ `big enough'. But for $\Char(\K)=p$
this allows us 
to prove non-semisimplicity of $\cal{B}_{r,s}(\delta)$ by 
proving non-semisimplicity for certain `boundary values'. 
For example, if $\delta_p\neq 0$, then 
this can be illustrated as
\[
\xy
(0,0)*{
\begin{tikzpicture}[scale=.75]
	\draw [very thick, ->] (0,0) to (0,6);
	\draw [very thick, ->] (0,0) to (6,0);
	\draw [very thick, red] (0,4) to (4,0);
	\draw [very thick, green] (1,4) to (4,1);
	\draw [very thick, red] (0,4) to (0,5.9);
	\draw [very thick, red] (4,0) to (5.9,0);
	\draw [very thick, green] (1,4) to (1,5.9);
	\draw [very thick, green] (4,1) to (5.9,1);
	\draw [very thick, dashed, ->] (4.25,.5) to (3.6,.5);
	\draw [very thick, dashed, ->] (4.25,1.5) to (3.6,1.5);
	\draw [very thick, blue, ->] (2,2) to (3,3);
	\draw [very thick, blue, ->] (3,3) to (4,4);
	\draw [very thick, blue, ->] (2,3) to (3,4);
	\draw [very thick, blue, ->] (3,4) to (4,5);
	\node at (-.25,6) {\tiny $s$};
	\node at (6,-.25) {\tiny $r$};
	\node at (7.25,.5) {\tiny $r+s=\min\{\delta_p+1,p-\delta_p+1\}+1$};
	\node at (7.25,1.5) {\tiny $r+s=\min\{\delta_p+1,p-\delta_p+1\}+2$};
	\node at (1.25,1) {\tiny semisimple};
	\node at (5,3.5) {\tiny non-semisimple};
	\node at (0,-0.02) {\tiny $\bullet$};
	\node at (2.02,1.98) {\tiny $\bullet$};
	\node at (3.02,2.98) {\tiny $\bullet$};
	\node at (4.02,3.98) {\tiny $\bullet$};
	\node at (2.02,2.98) {\tiny $\bullet$};
	\node at (3.02,3.98) {\tiny $\bullet$};
	\node at (4.02,4.98) {\tiny $\bullet$};
	\node at (4.425,4.5) {\reflectbox{\tiny $\ddots$}};
	\node at (4.425,5.5) {\reflectbox{\tiny $\ddots$}};
	\node at (0,-.25) {\tiny $(0,0)$};
\end{tikzpicture}
};
\endxy
\]
The `boundary line' (bottom, red) where the semisimplicity fails 
is illustrated above. 
We also displayed the passage from $(r,s)$ to $(r+1,s+1)$  
provided by Corollary~\ref{cor-picture}.
Note that we have to check 
additionally points on a line `above the boundary line' (top, green).

\subsection{The \texorpdfstring{$\delta_p\neq 0$}{delta_p not 0} case}

We assume in this subsection that $\Char(\K)=p$ and $\delta_p\neq 0$.

\begin{lem}\label{lem-walledsecond}
If $r+s<\delta_p+1$ and $r+s\leq p-\delta_p+1$, then $\cal{B}_{r,s}(\delta)$ is semisimple.
\end{lem}

\begin{proof}
We consider $T_{n}^{r,s}$ for $n=\delta_p$. Any Weyl factor $\Delta_1(\lambda)$ 
of $T_{n}^{r,s}$ satisfies
\[
\langle \lambda+\rho,(\varepsilon_i-\varepsilon_j)^{\vee}\rangle\leq
\langle r\varepsilon_1-s\varepsilon_{n}+\rho,(\varepsilon_1-\varepsilon_{n})^{\vee}\rangle\leq
\delta_p-1+r+s\leq p.
\]
Here the last inequality follows from the assumption that $r+s\leq p-\delta_p+1$. This means that 
all Weyl factors of $T_{n}^{r,s}$ are simple $\Uo$-modules, since there is no 
positive root $\alpha\in\Phi^+$ which gives a contribution to \textbf{JSF}.
As usual, the statement follows from 
Theorem~\ref{thm-schur-weyl2} (note that we have $r+s<\delta_p+1$ and the 
needed isomorphism holds) and Corollary~\ref{cor-cellsemisimple}.
\end{proof}

\begin{lem}\label{lem-walledsecondc}
If $r+s>p-\delta_p+1$ with $r,s\geq 1$, then $\cal{B}_{r,s}(\delta)$ is non-semisimple.
\end{lem}

\begin{proof}
Set $n=\delta_p$ and consider again $T_{n}^{r,s}$.
Let first $s=1$ and assume $r>p-\delta_p$ with $r<p$. 
As before, we only need to give one Weyl factor $\Delta_1(\lambda)$ which is a non-simple 
$\Uo$-module. 
We take 
$\lambda=(p-\delta_p+1)\varepsilon_1+\varepsilon_2+\dots+\varepsilon_{r-p+\delta_p}-\varepsilon_n$. 
Then $\alpha=\varepsilon_1-\varepsilon_{n}$ contributes to \textbf{JSF} because
\[
\lambda+\rho-p(\varepsilon_1-\varepsilon_n)=
(\colorbox{colormy}{\color{white}$0$},\delta_p-1,\delta_p-2,\dots,p-r,p-r-2,\dots,2,1,\colorbox{colormy}{\color{white}$p-1$})
\]
(note that $r<p$). This is a regular $\Uo$-weight because $\delta_p<p$. 
Again, cancellation can not occur, see Remark~\ref{rem-nocan}.
Thus, 
$T_{n}^{r,s}$ is a non-semisimple $\Uo$-module.

We now verify non-semisimple for the `boundary values'.
Assume $r,s>1$. Set
\[
\flat_1=\{(r,s)\mid r+s=(p-\delta_p+1)+1\},\quad\quad \flat_2=\{(r,s)\mid r+s=(p-\delta_p+1)+2\}.
\]
A direct computation using \textbf{JSF} 
shows that $\lambda=r\varepsilon_1-s\varepsilon_n$ is a 
Weyl factor $\Delta_1(\lambda)$ of $T_{n}^{r,s}$ that is a non-simple $\Uo$-module for all pairs 
$(r,s)\in \flat_1\cup \flat_2$ (the 
positive root making \textbf{JSF} non-zero 
is $\alpha=\varepsilon_1-\varepsilon_n$). For those $(r,s)$ we have that 
$T_{n}^{r,s}$ is a non-semisimple $\Uo$-module.

By  
Theorem~\ref{thm-schur-weyl2} and Corollary~\ref{cor-cellsemisimple}, we see that 
$\cal{B}_{r,s}(\delta)$ is non-semisimple under the same conditions 
(surjectivity in Theorem~\ref{thm-schur-weyl2} suffices since 
semisimple algebras have semisimple quotients). 
By Lemma~\ref{lem-walledsecondb} 
we additionally see that $\cal{B}_{r,1}(\delta)$ is non-semisimple for $r\geq p$. 
Thus, the 
statement follows from 
Corollaries~\ref{cor-picture} and~\ref{cor-picture2}.
\end{proof}

\begin{lem}\label{lem-walledsecondd}
If $r+s>\delta_p+1$ with $r,s\geq 1$, then $\cal{B}_{r,s}(\delta)$ is non-semisimple.
\end{lem}

\begin{proof}
Very similar to the proof of Lemma~\ref{lem-walledsecondc}. 
This time we take $n=p+\delta_p$ 
and we consider $T_{n}^{r,s}$. The `boundary values' for which we need to 
check non-semisimplicity are 
\begin{gather*}
\flat_1=\{(r,1)\mid \delta_p+1\leq r<p\},\\
\flat_2=\{(r,s)\mid r+s=(\delta_p+1)+1,\;s\geq 2\},\quad\quad \flat_3=\{(r,s)\mid r+s=(\delta_p+1)+2,\;s\geq 2\}.
\end{gather*}
For these we directly verify, using again \textbf{JSF}, 
that $T_{n}^{r,s}$ is a non-semisimple $\Uo$-module and the 
statement follows similarly as before. 
Since the arguments are straightforward, we only 
list a Weyl factor $\Delta_1(\lambda)$ 
that is a non-simple $\Uo$-module for each case (together with the positive 
roots giving a non-zero contribution to \textbf{JSF}). 
\begin{gather*}
\begin{aligned}
\text{For }\flat_1:&\;\; \lambda=(r-\delta_p)\varepsilon_1+\varepsilon_2+\dots+\varepsilon_{\delta_p+1}-\varepsilon_n
,\quad\quad\quad\;\;\;\;\text{positive root: }\alpha=\varepsilon_{\delta_p+1}-\varepsilon_n,\\
\text{For }\flat_2:&\;\; \lambda=\varepsilon_1+\varepsilon_2+\dots+\varepsilon_{r}-\varepsilon_{n-s+1}-\dots-2\varepsilon_{n}
,\quad\;\;\,\text{positive root: }\alpha=\varepsilon_{r}-\varepsilon_{n-s+1},\\
\text{For }\flat_3:&\;\; \lambda=2\varepsilon_1+\varepsilon_2+\dots+\varepsilon_{r-1}-\varepsilon_{n-s+1}-\dots-2\varepsilon_{n}
,\;\text{positive root: }\alpha=\varepsilon_{r-1}-\varepsilon_{n-s+1}.
\end{aligned}
\end{gather*}
Again, no cancellations occur by Remark~\ref{rem-nocan} and the statement follows as usual.
\end{proof}

\subsection{The \texorpdfstring{$\delta_p=0$}{delta_p=0} case}

We assume in this subsection that $\Char(\K)=p$ and $\delta_p=0$.

\begin{lem}\label{lem-walledfirst}
Let $p\geq 3$. Then $\cal{B}_{2,1}(\delta)$ is semisimple. 
If $p\geq 5$, then 
$\cal{B}_{3,1}(\delta)$ is also semisimple.
\end{lem}

\begin{proof}
Set $n=p$ and consider $T_{n}^{r,s}$ for $r=2,3$ and $s=1$. 
By Theorem~\ref{thm-schur-weyl2} and Corollary~\ref{cor-cellsemisimple}, 
it suffices to check that $T_n^{r,s}$ has 
only Weyl factors $\Delta_1(\lambda)$ which are simple $\Uo$-modules.

\noindent\textit{Case $r=2$.} The Weyl factors of $T_n^{r,s}$ 
are $\Delta_1(2\varepsilon_1-\varepsilon_p)$, 
$\Delta_1(\varepsilon_1+\varepsilon_2-\varepsilon_p)$ and $\Delta_1(\varepsilon_1)$. 
The third is clearly a simple $\Uo$-module and it remains to verify the same for the other two factors. 
As before, we want to use 
\textbf{JSF}:
\begin{itemize}
\item If $\lambda=2\varepsilon_1-\varepsilon_p$, then 
the only possible positive root $\alpha\in\Phi^+$ that contributes 
to the corresponding \textbf{JSF} 
is $\alpha=\varepsilon_1-\varepsilon_p$ (and it contributes only once). But
\[
\lambda+\rho-p(\varepsilon_1-\varepsilon_p)=(\colorbox{mycolor}{$1$},p-2,p-3,\dots,2,\colorbox{mycolor}{$1$},p-1)
\]
is a singular $\Uo$-weight (because $p\geq 3$). Thus, \textbf{JSF} of 
$\Delta_1(2\varepsilon_1-\varepsilon_p)$ is zero.
\item Similarly, if 
$\lambda=\varepsilon_1+\varepsilon_2-\varepsilon_p$, 
then, as before, the only positive root 
$\alpha\in\Phi^+$ we need to consider is 
$\alpha=\varepsilon_1-\varepsilon_p$. But
\[
\lambda+\rho-p(\varepsilon_1-\varepsilon_p)=(0,\colorbox{mycolor}{$p-1$},p-3,\dots,2,1,\colorbox{mycolor}{$p-1$})
\]
is a singular $\Uo$-weight. Hence, \textbf{JSF} of 
$\Delta_1(\varepsilon_1+\varepsilon_2-\varepsilon_p)$ is zero.
\end{itemize}
Thus, $T_n^{r,s}$ is a semisimple $\Uo$-module.

\noindent\textit{Case $r=3$.} We get the Weyl factors 
$\Delta_1(3\varepsilon_1-\varepsilon_p)$, $\Delta_1(2\varepsilon_1+\varepsilon_2-\varepsilon_p)$, 
$\Delta_1(\varepsilon_1+\varepsilon_2+\varepsilon_3-\varepsilon_p)$, 
$\Delta_1(\varepsilon_1+\varepsilon_2)$ and $\Delta_1(2\varepsilon_1)$. 
We proceed as above. 
For $\lambda=3\varepsilon_1-\varepsilon_p$ the only positive roots $\alpha\in\Phi^+$ 
we need to consider for
\textbf{JSF} are 
$\alpha=\varepsilon_1-\varepsilon_p$ and $\alpha=\varepsilon_1-\varepsilon_{p-1}$. 
Both contribute only one term and we get
\begin{gather*}
\begin{aligned}
\lambda+\rho-p(\varepsilon_1-\varepsilon_p)&=(\colorbox{mycolor}{$2$},p-2,p-3,\dots,\colorbox{mycolor}{$2$},1,p-1),\\
\lambda+\rho-p(\varepsilon_1-\varepsilon_{p-1})&=(\colorbox{mycolor}{$2$},p-1,p-3,\dots,\colorbox{mycolor}{$2$},p+1,-1).
\end{aligned}
\end{gather*}
Since $p\geq 5$, both are singular $\Uo$-weights. 
Similar for the remaining Weyl factors and we omit the calculation for brevity. 
We only note that one has to 
consider $\alpha=\varepsilon_1-\varepsilon_p$ for 
$\lambda=2\varepsilon_1+\varepsilon_2-\varepsilon_p$,  $\lambda=\varepsilon_1+\varepsilon_2+\varepsilon_3-\varepsilon_p$ 
and $\lambda=2\varepsilon_1$, 
while for $\lambda=\varepsilon_1+\varepsilon_2$ 
no positive root $\alpha\in\Phi^+$ 
needs to be considered.

The lemma follows.
\end{proof}

\begin{lem}\label{lem-walledfirstb}
Assume $r,s\geq 1$.
\begin{enumerate}[(a)]
\item If $p\geq 5$, then $\cal{B}_{r,s}(\delta)$ is semisimple if and only if 
$(r,s)\in\{(2,1),(1,2),(3,1),(1,3)\}$.
\item If $p=3$, then $\cal{B}_{r,s}(\delta)$ is semisimple if and only if 
$(r,s)\in\{(2,1),(1,2)\}$.
\item If $p=2$, then $\cal{B}_{r,s}(\delta)$ is never semisimple.
\end{enumerate}
\end{lem}

\begin{proof}
We only prove (a) and leave the other (completely similar) cases to the reader.

Because of the 
Corollaries~\ref{cor-picture} and~\ref{cor-picture2} it 
suffices to check that $\cal{B}_{r,s}(\delta)$ is non-semisimple 
for $(r,s)=(1,1)$ (difference $0$), 
$(r,s)=(3,2)$ (difference $1$), $(r,s)=(4,2)$ (difference $2$) and 
$(r,s)=(r,1)$ for $4\leq r$ (difference $\geq 3$).

As before, let $n=p$ and consider $T_{n}^{r,s}$. Hence,
it remains to find a Weyl factor 
$\Delta_1(\lambda)$ of $T_{n}^{r,s}$ which is a non-simple $\Uo$-module. 
We list such factors in the following. Since 
cancellations do not occur, see Remark~\ref{rem-nocan}, this 
suffices to show that the corresponding \textbf{JSF} is non-zero.
\begin{itemize}
\item The Weyl factor $\Delta_1(\varepsilon_1-\varepsilon_p)$ 
of $T_{n}^{r,s}$ for $(r,s)=(1,1)$ is a non-simple $\Uo$-module:
\[
\varepsilon_1-\varepsilon_p+\rho-p(\varepsilon_1-\varepsilon_p)=(\colorbox{colormy}{\color{white}$0$},p-2,p-3,\dots,2,1,\colorbox{colormy}{\color{white}$p-1$}).
\]
\item The Weyl factor $\Delta_1(2\varepsilon_1+\varepsilon_2-2\varepsilon_p)$ of 
$T_{n}^{r,s}$ for $(r,s)=(3,2)$ is a non-simple $\Uo$-module:
\begin{gather*}
2\varepsilon_1+\varepsilon_2-2\varepsilon_p+\rho-p(\varepsilon_2-\varepsilon_{p})=(p+1,\colorbox{colormy}{\color{white}$-1$},p-3,\dots,2,1,\colorbox{colormy}{\color{white}$p-2$}).
\end{gather*}
\item The Weyl factor $\Delta_1(3\varepsilon_1+\varepsilon_2-2\varepsilon_p)$ of 
$T_{n}^{r,s}$ for $(r,s)=(4,2)$ is a non-simple $\Uo$-module:
\begin{gather*}
3\varepsilon_1+\varepsilon_2-2\varepsilon_p+\rho-p(\varepsilon_2-\varepsilon_{p})=(p+2,\colorbox{colormy}{\color{white}$-1$},p-3,\dots,2,1,\colorbox{colormy}{\color{white}$p-2$}).
\end{gather*}
\item The Weyl factor $\Delta_1((r-2)\varepsilon_1+2\varepsilon_2-\varepsilon_p)$ of 
$T_{n}^{r,s}$ for $(r,s)=(r,1)$ with $4\leq r$ is a non-simple $\Uo$-module:
\begin{gather*}
(r-2)\varepsilon_1+2\varepsilon_2-\varepsilon_p+\rho-p(\varepsilon_2-\varepsilon_{p})=(p+r-3,\colorbox{colormy}{\color{white}$0$},p-3,\dots,2,1,\colorbox{colormy}{\color{white}$p-1$}).
\end{gather*}
Note that $4\leq r$ ensures that $(r-2)\varepsilon_1+2\varepsilon_2-\varepsilon_p$ 
occurs in $T_{n}^{r,s}$ (the $2$ in front of $\varepsilon_2$ is needed 
for $\alpha=\varepsilon_2-\varepsilon_{p}$ to give a 
contribution to \textbf{JSF}).
\end{itemize}
As in the proof of Lemma~\ref{lem-walledsecondb}, 
semisimple algebras have semisimple quotients. Hence, 
surjectivity in Theorem~\ref{thm-schur-weyl2} and 
Corollary~\ref{cor-cellsemisimple} provide 
the `only if' part of (a).
Thus, we have proven the statement,
because the `if' part of (a) follows from 
Lemma~\ref{lem-walledfirst} and Corollary~\ref{cor-picture2}.
\end{proof}

\subsection{Proof of the semisimplicity criterion for \texorpdfstring{$\cal{B}_{r,s}(\delta)$}{B_r,s(delta)}}

\begin{proof}[Proof of Theorem~\ref{thm-walledbrauer}]
\textit{(1).} The `only if' part of (a) follows from 
Lemmas~\ref{lem-walledsecondb},~\ref{lem-walledsecondc} 
and~\ref{lem-walledsecondd}

By Lemma~\ref{lem-walledsecond}, the only missing
case for the `if' part is the case $r+s=\delta_p+1$ and $p>\max\{r,s\}$, since in this case 
the corresponding 
Schur-Weyl duality gives only a surjection. As in the proof of Lemma~\ref{lem-walledsecond} 
we see that $T_{n}^{r,s}$ for $n+1=r+s=\delta_p+1$ is a semisimple $\Uo$-module. Thus, 
by  
Theorem~\ref{thm-schur-weyl2} and Theorem~\ref{thm-cellsemisimple}, 
we have that the algebra $\cal{B}_{r,s}(\delta)$ has $\End_{\Uo}(T_{n}^{r,s})$ 
as a semisimple quotient. We have calculated $\ker(\Phi_{\mathrm{wBr}})$ 
in Proposition~\ref{prop-kernel} from above: $\ker(\Phi_{\mathrm{wBr}})$ is one 
dimensional and spanned by the 
idempotent $e_{r,s}(n)$. The conclusion follows 
from Proposition~\ref{prop-abstract-nonsense}.

\noindent\textit{(2).} We can use Theorem~\ref{thm-transfer2}. That is, 
the statement in (2) can be 
obtained from the statement (1) by `taking the limit $p\to\infty$'.

\noindent\textit{(3), (4) and (5).} Directly from Lemma~\ref{lem-walledfirstb} 
and Theorem~\ref{thm-heckesemi} (for the 
cases where we have either $r=0$ or $s=0$).

\noindent\textit{(6).} Analogous to (2) by Theorem~\ref{thm-transfer2}, but using  
(3) instead of (1).

This finishes the proof.
\end{proof}

\begin{rem}\label{rem-semi2}
The semisimplicity criterion for $\cal{B}_{r,s}(\delta)$ is again of course not new: 
it was already discussed in~\cite[Theorem~6.3]{cddm}, but using a 
very different approach.
\end{rem}

\begin{rem}\label{rem-qauntizedwb}
Our approach works perfectly fine for the quantized walled Brauer algebras as well. The only 
difference is that one has to consider the versions~\eqref{eq-jsuma} and~\eqref{eq-jsumb} 
of \textbf{JSF} instead of~\eqref{eq-jsumc}. For brevity, we do not discuss the details here.
\end{rem}

%% file: res/brauer.tex
Let $d\in\Z_{>0}$. Choose 
$\delta$ and $\delta_p$ (recalling 
$\delta_0=|\delta|$) again as in 
Conventions~\ref{nota-delta}.

\begin{thm}(\textbf{Semisimplicity criterion for the Brauer algebra})\label{thm-brauer}
\newline
$\cal{B}_{d}(\delta)$ 
is semisimple if and only if one of the following conditions hold:
\begin{enumerate}
\item $\delta_p\neq 0$ odd, $\Char(\K)=p>2$ and $d\leq\min\{\delta_p+1,\frac{1}{2}(p-\delta_p+2)\}$.
\item $\delta_p\neq 0$ even, $\Char(\K)=p>2$ and $d\leq\min\{\delta_p+1,p-\delta_p+3,p-1\}$.
\item $\delta_0\neq 0$, $\Char(\K)=0$ and $d\leq\delta_0+1$.
\item $\delta_p=0$, $\Char(\K)=p>2$, $d\in\{1,3,5\}$ and $d<p$.
\item $\delta_0=0$, $\Char(\K)=0$ and $d\in\{1,3,5\}$.
\item $\Char(\K)=2$ and $d=1$.
\end{enumerate}
\end{thm}

We split the proof of Theorem~\ref{thm-brauer} into several lemmas.

\begin{rem}\label{rem-evenodd}
Note the difference between (1) and (2): 
the restriction $d\leq\frac{1}{2}(p-\delta_p+2)$ in the odd case is in general 
stronger than the restriction $d\leq p-\delta_p+3$ in the even case.
The reason for this 
is that odd $\delta$ 
corresponds via Schur-Weyl-Brauer duality to $\Uo(\soo{2m+1})$ 
whereas even $\delta$ corresponds to $\Uo=\Uo(\soo{2m})$. 
In the latter case the Brauer algebra $\cal{B}_d(\delta)$ does not 
control $\End_{\Uo}(T_{n}^{d})$ well enough  
since there is a non-trivial automorphism of the Dynkin diagram, see Section~\ref{sec-schur} (which  
we use below in the proof of Lemma~\ref{lem-brauersecondb}).
In particular, the semisimplicity in the even case is much harder to 
prove than in the odd case.
\end{rem}

\subsection{A summary of the proof}

The proof of Theorem~\ref{thm-brauer} is slightly involved.
For the convenience of the reader we summarize its proof.
We like to note that the proof itself is mostly smooth -- except 
for a short list of special cases coming 
from the fact that the types $\BM$, $\CM$ and $\DM$ 
are `special' for small $m$ (we 
tend to omit the calculations for these for brevity).

First we assume $\Char(\K)=p>2$ where we 
use Lemma~\ref{lem-brauersecondnew} to further assume $p>d$. We can deduce the 
case $\Char(\K)=0$ from it by trace form arguments, see~\ref{sec-app2}. 
In the remaining case, 
$\Char(\K)=2$ and $d=1$, semisimplicity 
of $\cal{B}_d(\delta)$ is immediate.

In the situation $\Char(\K)=p>2$, 
we start by deducing a general argument that enables us to go from $d$ to $d+2$ (similarly 
as in the walled Brauer case).
We then 
separate three cases: $\delta_p\neq 0$ odd, $\delta_p\neq 0$ even and $\delta_p=0$. In all three 
cases there is a $d_0$ such that 
$\cal{B}_d(\delta)$ is semisimple for 
$d < d_0$ and non-semisimple for $d \geq d_0$.
We verify these cases separately.
For example, our argumentation in the first case ($\delta_p\neq 0$ odd)
can be illustrated as follows.
\begin{gather*}
\xy
(0,0)*{
\begin{tikzpicture}[scale=.75]
\draw [very thick, colormy] (3,0) to (8,0);
\draw [very thick, colormy, dashed] (0,0) to (3,0);
\draw [very thick, mycolor] (0,0) to (-5,0);
\draw [very thick, colormy] (0,-0.7) to (8,-0.7);
\draw [very thick, colormy] (0,0) to (0,-0.7);
\draw [very thick, dashed, ->] (3,0.6) to (3,0.05);
\draw [very thick, dashed, ->] (0,0.6) to (0,0.05);
	\node at (0,-0.025) {\tiny $\bullet$};
	\node at (3,-0.025) {\tiny $\bullet$};
	\node at (0,.7) {\tiny $d=\delta_p+1$};
	\node at (3,.7) {\tiny $d=\frac{1}{2}(p-\delta_p+2)$};
	\node at (-2.5,.2) {\tiny semisimple};
	\node at (5.5,.2) {\tiny non-semisimple};
	\node at (-2.5,-.2) {\tiny Lemma~\ref{lem-brauerseconda}};
	\node at (5.5,-.2) {\tiny Lemma~\ref{lem-brauersecondc}};
	\node at (1.5,-.5) {\tiny Lemma~\ref{lem-brauerseconde}};
\end{tikzpicture}
};
\endxy\\
\xy
(0,0)*{
\begin{tikzpicture}[scale=.75]
\draw [very thick, colormy] (3,0) to (8,0);
\draw [very thick, colormy, dashed] (0,0) to (3,0);
\draw [very thick, mycolor] (0,0) to (-5,0);
\draw [very thick, colormy] (0,-0.7) to (8,-0.7);
\draw [very thick, colormy] (0,0) to (0,-0.7);
\draw [very thick, dashed, ->] (3,0.6) to (3,0.05);
\draw [very thick, dashed, ->] (0,0.6) to (0,0.05);
	\node at (0,-0.025) {\tiny $\bullet$};
	\node at (3,-0.025) {\tiny $\bullet$};
	\node at (0,.7) {\tiny $d=\frac{1}{2}(p-\delta_p+2)$};
	\node at (3,.7) {\tiny $d=\delta_p+1$};
	\node at (-2.5,.2) {\tiny semisimple};
	\node at (5.5,.2) {\tiny non-semisimple};
	\node at (-2.5,-.2) {\tiny Lemma~\ref{lem-brauerseconda}};
	\node at (5.5,-.2) {\tiny Lemma~\ref{lem-brauerseconde}};
	\node at (1.5,-.5) {\tiny Lemma~\ref{lem-brauersecondc}};
\end{tikzpicture}
};
\endxy
\end{gather*}
Here the top case is $\delta_p+1\leq\frac{1}{2}(p-\delta_p+2)$, while the 
bottom is $\frac{1}{2}(p-\delta_p+2)\leq\delta_p+1$. We have illustrated the 
`boundary value' where $\cal{B}_d(\delta)$ stops to be semisimple and what lemmas we use to 
deduce (non-)semisimplicity. Note that, as in the walled Brauer case, one `boundary case' 
remains to be verified. We do this, as before, by using our explicit description of 
the kernel of the Schur-Weyl-Brauer action from Proposition~\ref{prop-kernelbrauer}.
Similarly in the other cases.

\subsection{The Schur-Weyl-Brauer dual story: from $d$ to $d+2$}

Let $\Uo=\Uo(\fg)$ and 
$\fg$ be either $\mathfrak{so}_{2m+1}$, $\mathfrak{sp}_{2m}$ or $\mathfrak{so}_{2m}$ 
(types $\BM,\CM$ and $\DM$ respectively). We set $n=2m+1$ for 
$\fg=\mathfrak{so}_{2m+1}$ and $n=2m$ otherwise. 
Moreover, let $V$ and
$T_{n}^{d}$ be as in Theorem~\ref{thm-schur-weyl3}.  
Again, $V,T_{n}^{d}\in\T$ by Proposition~\ref{prop-prop}.
As usual, we can 
calculate the Weyl factors of 
$T_{n}^{d}$ as in the classical case.

\begin{prop}\label{prop-lem1brauer}
If $T_{n}^{d}$ is a non-semisimple $\Uo$-module, 
then so is $T_{n}^{d+2}$.
\end{prop}

\begin{proof}
Note that $\Delta_1(0)\cong\K$ is a Weyl factor of $T_{n}^{2}$. As in the proof 
of Proposition~\ref{prop-lem1}: any Weyl factor 
of $T_{n}^{d}$ is also a Weyl factor of $T_{n}^{d+2}$.
The conclusion follows from Lemma~\ref{lem-cellsemisimple}.
\end{proof}

\begin{cor}\label{cor-picturebrauer}
Let $\Char(\K)=p$. If 
$\cal{B}_{d}(\delta)$ is non-semisimple, then so is $\cal{B}_{d+2}(\delta)$.
\end{cor}

\begin{proof}
Choose $n\geq 2d$ with $n\equiv \delta_p\;\text{mod}\;p$. Then the statement 
follows directly from Theorem~\ref{thm-schur-weyl3}, Proposition~\ref{prop-lem1brauer} 
and Theorem~\ref{thm-cellsemisimple}.
\end{proof}

Corollary~\ref{cor-picturebrauer} again 
fails in general for $\Char(\K)=0$ for the same 
reasons as in Corollary~\ref{cor-picture}.

Analogously to the case of walled Brauer algebras, we use Corollary~\ref{cor-picturebrauer} 
to check certain boundary values. For example, if $\delta_p\neq 0$ is odd, 
then this can be illustrated as
\[
\xy
(0,0)*{
\begin{tikzpicture}[scale=.75]
\draw [very thick, mycolor] (0,0) to (4.5,0);
\draw [very thick, colormy, ->] (4.5,0) to (13.5,0);
\draw [very thick, dashed, ->] (4.5,-.65) to (4.5,-.1);
\draw [very thick, dashed] (4.5,0) to (4.5,1);
\draw [very thick, blue, ->] (4.5,0.05) [out=45, in=180] to (6,1) to [out=0, in=135] (7.5,0.05);
\draw [very thick, blue, ->] (7.5,0.05) [out=45, in=180] to (9,1) to [out=0, in=135] (10.5,0.05);
\draw [very thick, blue, ->] (6,-0.05) [out=315, in=180] to (7.5,-1) to [out=0, in=225] (9,-0.05);
\draw [very thick, blue, ->] (9,-0.05) [out=315, in=180] to (10.5,-1) to [out=0, in=225] (12,-0.05);
	\node at (0,-0.025) {\tiny $\bullet$};
	\node at (1.5,-0.025) {\tiny $\bullet$};
	\node at (3,-0.025) {\tiny $\bullet$};
	\node[red] at (4.5,-0.025) {$\bullet$};
	\node[green] at (6,-0.025) {\tiny $\blacksquare$};
	\node at (7.5,-0.025) {\tiny $\bullet$};
	\node at (9,-0.025) {\tiny $\bullet$};
	\node at (10.5,-0.025) {\tiny $\bullet$};
	\node at (12,-0.025) {\tiny $\bullet$};
	\node at (13.5,-.25) {\tiny $d$};
	\node at (4,-1.00) {\tiny $d=\min\{\delta_p+1,\frac{1}{2}(p-\delta_p+2)\}+1$};
	\node at (2.25,.2) {\tiny semisimple};
	\node at (9,.2) {\tiny non-semisimple};
	\node at (11,.5) {\tiny $\cdots$};
	\node at (12.5,-.5) {\tiny $\cdots$};
	\node at (0,-.25) {\tiny $0$};
\end{tikzpicture}
};
\endxy
\]
Again, there are two boundary values (displayed as a red dot respectively green box).

We want to point out that the relation of $\cal{B}_{d}(\delta)$ and 
$\cal{B}_{d+2}(\delta)$ underlying Corollary~\ref{cor-picturebrauer} 
was already observed in~\cite[Section~1.3]{martin} 
and~\cite[Section~5]{dwh}, while the 
`trick' to add $p$, used in Corollary~\ref{cor-picturebrauer} and below, 
appeared in~\cite[Section~5]{dt}, but both 
in a different setting.

\subsection{The case \texorpdfstring{$\delta_p\neq 0$}{delta_p not 0} is odd or even}

Let $\Char(\K)=p>2$, $p>d$ and let $\delta_p\neq 0$ be odd or even.

\begin{lem}\label{lem-brauerseconde}
If $d>\delta_p+1$, then $\cal{B}_{d}(\delta)$ is non-semisimple.
\end{lem}

\begin{proof}
By Corollary~\ref{cor-picturebrauer}, it 
suffices to check the boundary values $d=\delta_p+2$ and $d=\delta_p+3$. 
Consider 
$T_{n}^{d}$ for $n=p+\delta_p$. First let us assume 
that $\delta_p$ is even (type $\BM$ with $m=\frac{1}{2}(p+\delta_p-1)$). 
Consider 
$\lambda=2\varepsilon_1+\varepsilon_2+\dots+\varepsilon_{\delta_p+1}$ (for $d=\delta_p+2$)
and $\mu=3\varepsilon_1+2\varepsilon_2+\dots+\varepsilon_{\delta_p}$ (for $d=\delta_p+3$). We get
\begin{gather*}
\lambda+\rho-p(\varepsilon_1+\varepsilon_{\delta_p-1})=
{\textstyle\frac{1}{2}}\cdot(\colorbox{colormy}{\color{white}$-p+\delta_p+2$},p+\delta_p-2,\dots,\colorbox{colormy}{\color{white}$-p-\delta_p$},
p-\delta_p-4,\dots,3,1),\\
\mu\!+\!\rho\!-\!p(\varepsilon_2+\varepsilon_{\delta_p})\!=\!
{\textstyle\frac{1}{2}}\!\cdot\!(p+\delta_p+4,\colorbox{colormy}{\color{white}$-p+\delta_p$},p+\delta_p-4,\dots,\colorbox{colormy}{\color{white}$-p-\delta_p+2$},
p-\delta_p-2,\dots,3,1),
\end{gather*}
for the two boundary values.
Because we assume $d<p$, we have that $\delta_p\leq p-3$. Thus, the 
maximal $k$ in \textbf{JSF} is $k=1$ and
a direct computation verifies that the contribution of the positive roots 
$\alpha\in\Phi^+$ of the forms 
$\alpha=\varepsilon_1+\varepsilon_{\delta_p-1}$ and 
$\alpha=\varepsilon_2+\varepsilon_{\delta_p}$ to the \textbf{JSF} of 
$\Delta_1(\lambda)$ and $\Delta_1(\mu)$ from above are not cancelled.
Hence, \textbf{JSF} of $\Delta_1(\lambda)$ and of $\Delta_1(\mu)$ are non-zero.

Now assume $\delta_p$ is odd. We take again the same $\lambda$ and $\mu$ and the same 
reasoning as above works (which takes place in type $\DM$ for $m=\frac{1}{2}(p+\delta_p)$ now). 
The surjectivity in Theorem~\ref{thm-schur-weyl3} and 
Corollary~\ref{cor-cellsemisimple} provide the statement, 
since semisimple algebras have semisimple quotients. Note that the surjectivity fails in type $\DM$ for 
the cases $\delta_p=p-4$ and $d=\delta_p+2=p-2$ or $d=\delta_p+3=p-1$, or 
$\delta_p=p-6$ and $d=\delta_p+3=p-3$ (note that $m\geq 4$ in the remaining cases). 
These have to be $p$-shifted twice to
$\BM$ by taking $m=\frac{1}{2}(2p+\delta_p-1)$, but the argument is again similar 
(that is, using \textbf{JSF}), but slightly involved due to 
the `size' of the numbers in question and omitted 
for brevity.
\end{proof}

\subsection{The case \texorpdfstring{$\delta_p\neq 0$}{delta_p not 0} is odd}

Let $\Char(\K)=p>2$, $p>d$ and let $\delta_p\neq 0$ be odd.

\begin{lem}\label{lem-brauerseconda}
If $d\leq\delta_p+1$ and $d<\frac{1}{2}(p-\delta_p+2)$, 
then $\cal{B}_{d}(\delta)$ is semisimple.
\end{lem}

\begin{proof}
Note that $p$ is odd and $p>\delta_p$. Hence, $n=p-\delta_p\geq 2d$ is a positive, 
even number. 
Consider $T^d_n$ (type $\CM$ with $m={\textstyle \frac{1}{2}}(p-\delta_p)$). 
Then Theorem~\ref{thm-schur-weyl3} gives
\[
\cal{B}_d(\delta)\cong\cal{B}_d(\delta_p)\cong\cal{B}_d(\delta_p-p)\cong\End_{\Uo}(T_{n}^{d}).
\] 
Since $d\leq \delta_p+1$, all Weyl factors $\Delta_1(\lambda)$ of $T_{n}^{d}$ satisfy
\[
\langle\lambda+\rho,\alpha^{\vee}\rangle\leq\langle d\varepsilon_1+\rho,(\varepsilon_1+\varepsilon_2)^{\vee}\rangle
=d+p-\delta_p-1\leq p,
\]
for all $\alpha\in\Phi^+$.
Hence, the corresponding \textbf{JSF}s are all zero and the statement 
follows from 
Corollary~\ref{cor-cellsemisimple} as long as $m\geq 3$. The case $m=1$ only occurs if 
$\delta_p=p-2$ and thus, $d<\frac{1}{2}(p-\delta_p+2)$ gives $d<2$ where 
semisimplicity is clear. The case $m=2$ only occurs if 
$\delta_p=p-4$ for $p\geq 5$ and thus, $d<\frac{1}{2}(p-\delta_p+2)$ gives $d<3$. Semisimplicity 
of $\cal{B}_d(\delta)$ for $d=2$ and $p\geq 5$ follows because 
the following
pairwise orthogonal, primitive idempotents
\[
{\textstyle\frac{1}{2}}\cdot
\xy
(0,0)*{
\begin{tikzpicture}[scale=.3]
	\draw [very thick] (-1,-1) to (-1,1);
	\draw [very thick] (1,-1) to (1,1);
\end{tikzpicture}
};
\endxy
-{\textstyle\frac{1}{2}}\cdot
\xy
(0,0)*{
\begin{tikzpicture}[scale=.3]
	\draw [very thick] (-1,-1) to (1,1);
	\draw [very thick] (1,-1) to (-1,1);
\end{tikzpicture}
};
\endxy\,,\quad\quad
{\textstyle\frac{1}{\delta}}\cdot
\xy
(0,0)*{
\begin{tikzpicture}[scale=.3]
	\draw [very thick] (5,-1) to [out=90,in=180] (6,-.35) to [out=0,in=90] (7,-1);
	\draw [very thick] (7,1) to [out=270,in=0] (6,.35) to [out=180,in=270] (5,1);
\end{tikzpicture}
};
\endxy\,,\quad\quad
{\textstyle\frac{1}{2}}\cdot
\xy
(0,0)*{
\begin{tikzpicture}[scale=.3]
	\draw [very thick] (-1,-1) to (-1,1);
	\draw [very thick] (1,-1) to (1,1);
\end{tikzpicture}
};
\endxy
+{\textstyle\frac{1}{2}}\cdot
\xy
(0,0)*{
\begin{tikzpicture}[scale=.3]
	\draw [very thick] (-1,-1) to (1,1);
	\draw [very thick] (1,-1) to (-1,1);
\end{tikzpicture}
};
\endxy
-{\textstyle\frac{1}{\delta}}\cdot
\xy
(0,0)*{
\begin{tikzpicture}[scale=.3]
	\draw [very thick] (5,-1) to [out=90,in=180] (6,-.35) to [out=0,in=90] (7,-1);
	\draw [very thick] (7,1) to [out=270,in=0] (6,.35) to [out=180,in=270] (5,1);
\end{tikzpicture}
};
\endxy
\]
form a basis of $\cal{B}_2(\delta)$. 
Hence, $\cal{B}_2(\delta)\cong\K\oplus\K\oplus\K$.
\end{proof}

\begin{lem}\label{lem-brauersecondc}
If $d>\frac{1}{2}(p-\delta_p+2)$, then $\cal{B}_{d}(\delta)$ is non-semisimple.
\end{lem}

\begin{proof}
As above, by Corollary~\ref{cor-picturebrauer} it 
remains to verify non-semisimplicity for the boundary values $d=\frac{1}{2}(p-\delta_p+2)+1$ 
and $d=\frac{1}{2}(p-\delta_p+2)+2$. We first assume that $\delta_p\geq 3$ 
and we take $T_{n}^{d}$ for $n=p+\delta_p$ (thus, 
we are in type $\DM$ with $m=\frac{1}{2}(p+\delta_p)\geq 4$). 
By surjectivity in Theorem~\ref{thm-schur-weyl3} and by
Corollary~\ref{cor-cellsemisimple}, it remains to find 
Weyl factors $\Delta_1(\lambda)$
for both $T_{n}^{d}$ which are non-simple $\Uo$-modules. 
We take $\lambda=d\varepsilon_1$ 
and $\alpha=\varepsilon_1+\varepsilon_{m-1}$ (for $d=\frac{1}{2}(p-\delta_p+2)+1$) 
and $\alpha=\varepsilon_1+\varepsilon_{m-2}$ (for $d=\frac{1}{2}(p-\delta_p+2)+2$):
\begin{gather*}
\begin{aligned}
\lambda+\rho-p(\varepsilon_1+\varepsilon_{m-1})&=
(\colorbox{colormy}{\color{white}$1$},{\textstyle \frac{1}{2}}(p+\delta_p)-2,{\textstyle \frac{1}{2}}(p+\delta_p)-3,\dots,3,2,\colorbox{colormy}{\color{white}$p+1$},0),\\
\lambda+\rho-(\varepsilon_1+\varepsilon_{m-2})&=
(\colorbox{colormy}{\color{white}$2$},{\textstyle \frac{1}{2}}(p+\delta_p)-2,{\textstyle \frac{1}{2}}(p+\delta_p)-3,\dots,3,\colorbox{colormy}{\color{white}$p+2$},1,0).
\end{aligned}
\end{gather*}
These are regular $\Uo$-weights since $\delta_p<p$ which are 
not cancelled in \textbf{JSF}: only positive roots 
$\alpha\in\Phi^+$ of the form 
$\alpha=\varepsilon_1+\varepsilon_j$ for $j\neq 1$ can contribute to \textbf{JSF} 
(and at most once) and all 
of these yield singular $\Uo$-weights. Thus, \textbf{JSF}s 
of these $\Delta_1(\lambda)$'s are non-zero.

It remains to verify the case $\delta_p=1$.
First note that we do not have to consider 
$p=3$, since 
$d>\frac{1}{2}(p-\delta_p+2)=3=p$ by assumption.
For the remaining cases first assume $p\geq 7$. We take 
$\Uo=\Uo(\spo{2m})$ with 
$m=\frac{1}{2}(p-\delta_p)$ (hence, $m\geq 3$) 
and $T_{n}^{d}$ with $n=p-\delta_p$. Then we proceed as before, 
but in type $\CM$ and with $\alpha=\varepsilon_1+\varepsilon_{m}$ 
instead of $\alpha=\varepsilon_1+\varepsilon_{m-1}$ and 
$\alpha=\varepsilon_1+\varepsilon_{m-1}$ 
instead of $\alpha=\varepsilon_1+\varepsilon_{m-2}$. The remaining case $p=5,\delta_p=1$ and 
$d=4$ can be done by going to type $\BM$ with $m=5$.
\end{proof}

\subsection{The case \texorpdfstring{$\delta_p\neq 0$}{delta_p not 0} is even}

Let $\Char(\K)=p>2$, $p>d$ and let $\delta_p\neq 0$ be even.

\begin{lem}\label{lem-brauersecondb}
If $d\leq\min\{\delta_p,p-\delta_p+3\}$, then $\cal{B}_{d}(\delta)$ is semisimple.
\end{lem}

\begin{proof}
We take the Schur-Weyl-Brauer data as in Theorem~\ref{thm-brauerimage}, 
that is $\Uo=\Uo(\soo{2m})$ with $n=\delta_p$ (type $\DM$ with $m=\frac{1}{2}\delta_p$) and the 
$d$-fold tensor product $T^d_n$ of its vector representation (we note that 
our arguments go through in case $m\leq 3$ as well, see Remark~\ref{rem-oneextra2}).

We claim that $T^d_n$ is a semisimple $\Uo$-module:
it, as usual, remains to check that all Weyl factors $\Delta_1(\lambda)$ of $T^d_n$ 
have zero as their \textbf{JSF}. This follows almost directly, since, for all positive roots 
$\alpha\in\Phi^+$, we have (recall that $d\leq p-\delta_p+3$)
\[
\langle\lambda+\rho,\alpha^{\vee}\rangle\leq\langle d\varepsilon_1+\rho,(\varepsilon_1+\varepsilon_2)^{\vee}\rangle
=d+2m-3=d+\delta_p-3\leq p.
\]
Thus, $T^d_n$ is a semisimple $\Uo$-module and hence, 
a direct sum of simple Weyl modules. By Proposition~\ref{prop-clifford} and 
Lemma~\ref{lem-decomptwist} 
(which is valid in this specific $\Char(\K)=p$ case since $T^d_n$ is 
a direct sum of simple Weyl modules), we have that $T^d_n$ is also a semisimple $\Uoo$-module.
Hence, $\End_{\Uoo}(T^d_n)$ is semisimple.
Since we assume $d\leq\delta_p$, we 
have $\cal{B}_{d}(\delta)\cong\End_{\Uoo}(T^d_n)$ by Theorem~\ref{thm-brauerimage}. 
The statement follows.
\end{proof}

\begin{lem}\label{lem-brauersecondbb}
If $d\leq\min\{\delta_p+1,p-\delta_p+3\}$, then $\cal{B}_{d}(\delta)$ is semisimple.
\end{lem}

\begin{proof}
By Lemma~\ref{lem-brauersecondb} it suffices to check 
that $\cal{B}_d(\delta)$ is semisimple for $d=\delta_p+1<p-\delta_p+3$.

In order to do so, we fix $n=p+\delta_p$ and the 
statement follows by Theorem~\ref{thm-schur-weyl3}, if we show that 
$T_{n}^{d}$ is a semisimple 
$\Uo=\Uo(\soo{2m+1})$-module (type $\BM$ with $m=\frac{1}{2}(p+\delta_p-1)$).

Our argument below 
uses $p\geq 7$. Since $\delta_p+1< p-\delta_p+3$ 
gives $\delta_p<\frac{1}{2}p+1$, only $\delta_p=2$ and $p=3,5$ 
are the cases with $p<7$ for which we need to check semisimplicity of $\cal{B}_d(\delta)$. 
In case $p=3$ we would have to check $d=\delta_p+1=3$. Since we assume $p>d$, 
this case does not occur. The case $p=5$ and $d=3$ can be verified as 
usual using \textbf{JSF} and is in particular very similar to the case $p\geq 7$
discussed below (the stated inequalities below are not true anymore and there are a few extra 
cases to check). We leave the details to the reader.

Assume now $p\geq 7$.
Following 
our usual recipe, we have to show that all Weyl factors $\Delta_1(\lambda)$ 
of $T_{n}^{d}$ have zero \textbf{JSF}.
Note now that such $\lambda$'s satisfy $\lambda\in\Lambda^+(d-2i^{\prime})$ for 
some $i^{\prime}=0,\dots,\lfloor\frac{1}{2}d\rfloor$. It turns out that 
there are two different cases:
the first case is $\lambda_{t}=0$ and the second
is $\lambda_{t}=1$ (for $t=\frac{1}{2}(d+1)$).
Note that these are all cases since $\lambda_{t}>1$ can not 
occur for $\lambda\in\Lambda^+(d-2i^{\prime})$ (in particular, we always have 
$\lambda_{t^{\prime}}\in\{0,1\}$ for all $t\leq t^{\prime}\leq d-2i^{\prime}$).

Fix now $\lambda\in\Lambda^+(d-2i^{\prime})$. A direct verification shows 
that positive roots $\alpha\in\Phi^+$ of the form $\alpha=\varepsilon_i-\varepsilon_j$ 
(for $1\leq i<j\leq m$) will 
never yield contributions to \textbf{JSF} of $\Delta_1(\lambda)$ (this can 
be seen similarly as in the proof of Theorem~\ref{thm-heckesemi}). Thus, we only 
need to check positive roots $\alpha\in\Phi^+$ of the form 
$\alpha=\varepsilon_i+\varepsilon_j$ (for $1\leq i<j\leq m$) or of the form
$\alpha=\varepsilon_i$ (for $i=1,\dots,m$).
Note now that 
$d=\delta_p+1$, $\delta_p< \frac{1}{2}p+1$ and $p\geq 7$ gives
\begin{gather*}
\begin{aligned}
\langle\lambda+\rho,(\varepsilon_i+\varepsilon_j)^{\vee}\rangle &\leq\langle d\varepsilon_1+\rho,(\varepsilon_1+\varepsilon_2)^{\vee}\rangle=p+\delta_p+d-3=p+2\delta_p-2< 2p,\\
\langle\lambda+\rho,(\varepsilon_i)^{\vee}\rangle &\leq\langle d\varepsilon_1+\rho,(\varepsilon_1)^{\vee}\rangle=p+\delta_p+2d-2=p+3\delta_p<{\textstyle\frac{5}{2}}p+3\leq 3p.
\end{aligned}
\end{gather*}
Thus, it suffices to consider $k=1$ or $k=1,2$ in \textbf{JSF} of $\Delta_1(\lambda)$.

\noindent\textit{Case $\lambda_{t}=0$.} 
There is a `tail' in 
$\lambda+\rho$: every value of the form $\frac{1}{2}(2k^{\prime}+1)$ (for $k^{\prime}\in\N$) 
appears after $(\lambda+\rho)_{t}=\rho_{t}=\frac{1}{2}p-1$. 
Assume
$p<\langle\lambda+\rho,(\varepsilon_i)^{\vee}\rangle< 3p$ for some $i=1,\dots,m$ 
(it follows that $i\leq t$). Thus, 
$\frac{1}{2}p<\lambda_i+m-\frac{1}{2}-i<\frac{3}{2}p$. Then $(\lambda+\rho-p\varepsilon_i)_i$ 
will always be in the `tail' and hence, $\lambda+\rho-kp\varepsilon_i$ is 
a singular $\Uo$-weight for $k=1$ and such $\lambda$:
\[
0<|(\lambda+\rho-p\varepsilon_i)_i|=|\lambda_i+m+\tfrac{1}{2}-i-p|\leq {\textstyle\frac{1}{2}}p-1 =(\lambda+\rho)_{t}=\rho_t.
\]
Similarly, 
$\lambda+\rho-p(\varepsilon_i+\varepsilon_j)$ (for $1\leq i<j\leq m$) is 
a singular $\Uo$-weight except if we have
$(\lambda+\rho-p\varepsilon_i)_i=(\lambda+\rho)_j$. The latter occurs if and only if 
$2p<\langle\lambda+\rho,(\varepsilon_i)^{\vee}\rangle< 3p$. These 
$\Uo$-weights give contributions to \textbf{JSF} of 
$\Delta_1(\lambda)$, but are cancelled by the contribution for $\varepsilon_i$ and $k=2$.

In summary, \textbf{JSF} of $\Delta_1(\lambda)$ is zero : either $1\leq i<j\leq m$ 
does not have any contributions for 
$\alpha=\varepsilon_i+\varepsilon_j$ or $\alpha=\varepsilon_i$
(in the case $0<\langle\lambda+\rho,(\varepsilon_i)^{\vee}\rangle\leq p$) 
or only singular $\Uo$-weights appear 
(this happens in the case $p< \langle\lambda+\rho,(\varepsilon_i)^{\vee}\rangle\leq 2p$) or 
there will be cancellations 
(this happens in the case $2p< \langle\lambda+\rho,(\varepsilon_i)^{\vee}\rangle< 3p$).

\noindent\textit{Case $\lambda_{t}=1$.} Similarly as before. We omit the 
details for brevity and only note that the assumption $\lambda_{t}=1$ 
ensures that there will be only one `gap in the tail'. Hence, the same 
argumentation as above goes through with the extra case that $(\lambda+\rho-p\varepsilon_i)_i$ 
can precisely land in this gap. In this case $\lambda+\rho-p\varepsilon_i$ 
is a regular $\Uo$-weight, but it is again cancelled in \textbf{JSF} of $\Delta_1(\lambda)$ 
(this time by $\lambda+\rho-p(\varepsilon_i+\varepsilon_j)$ where $j$ is the entry of the gap).

Thus, $T_{n}^{d}$ is a semisimple $\Uo$-module which shows the statement.
\end{proof}

\begin{rem}\label{re-stupidproof}
The boundary case in Lemma~\ref{lem-brauersecondbb} could also 
be done by analyzing the kernel of the Schur-Weyl-Brauer action 
$\Phi_{\mathrm{Br}}$ as in the 
proof of Theorem~\ref{thm-walledbrauer} and as 
in the proof of Theorem~\ref{thm-brauer} below. 
But this would require going to the reductive group $O_{2m}$ (Brauer 
already observed in~\cite[Page~870]{bra} that surjectivity of $\Phi_{\mathrm{Br}}$ fails 
in general for $SO_{2m}$). In order to keep the paper reasonably 
self-contained, we avoid using the reductive group setting here.
\end{rem}

\begin{lem}\label{lem-brauersecondd}
If $d>p-\delta_p+3$, then $\cal{B}_{d}(\delta)$ is non-semisimple.
\end{lem}

\begin{proof}
We take $n=p+\delta_p$ (type $\BM$ with $m=\frac{1}{2}(p+\delta_p-1)$ again). 
As usual, by the surjectivity 
in Theorem~\ref{thm-schur-weyl3} and 
Corollary~\ref{cor-cellsemisimple}, it suffices to give 
a Weyl factor of $T_{n}^{d}$ that is a non-simple $\Uo$-module. By 
Corollary~\ref{cor-picturebrauer}, it remains to give such 
factors in the cases $d=(p-\delta_p+3)+1$ and $d=(p-\delta_p+3)+2$. Take 
$\lambda=(d-1)\varepsilon_1+\varepsilon_2$ and 
$\mu=(d-2)\varepsilon_1+2\varepsilon_2$:
\begin{gather*}
\begin{aligned}
\lambda+\rho-p(\varepsilon_1+\varepsilon_{\delta_p-2}) &={\textstyle\frac{1}{2}}(\colorbox{colormy}{\color{white}$p-\delta_p+4$},p+\delta_p-2,p+\delta_p-6,\dots,\colorbox{colormy}{\color{white}$-p-\delta_p+4$},\dots,3,1),\\
\mu+\rho-p(\varepsilon_1+\varepsilon_{\delta_p-2}) &={\textstyle\frac{1}{2}}(\colorbox{colormy}{\color{white}$p-\delta_p+4$},p+\delta_p,p+\delta_p-6,\dots,\colorbox{colormy}{\color{white}$-p-\delta_p+4$},\dots,3,1).
\end{aligned}
\end{gather*}
Here we assume that $\delta_p\geq 4$ (for $\delta_p=2$  
we have $d>p$ and we can use Lemma~\ref{lem-brauersecondnew}). 
Only the positive roots $\alpha\in\Phi^+$ 
with $\alpha=\varepsilon_i\pm\varepsilon_j$ 
or $\alpha=\varepsilon_i$ for $i=1$ can 
give other non-zero contributions to \textbf{JSF}s. 
These remaining positive roots $\alpha\in\Phi^+$ 
do not cancel the contributions above 
($k\leq 1$ for all of these). Hence, \textbf{JSF}s for $\Delta_1(\lambda)$ 
and $\Delta_1(\mu)$ are non-zero.
\end{proof}

\subsection{The \texorpdfstring{$\delta_p=0$}{delta_p=0} case}

Let $\Char(\K)=p>2$, $p>d$ and $\delta_p=0$.

\begin{lem}\label{lem-brauerfirst}
Let $p\geq 5$. Then $\cal{B}_{3}(\delta)$ is semisimple. 
If $p\geq 7$, then 
$\cal{B}_{5}(\delta)$ is also semisimple.
\end{lem}

\begin{proof}
Again, we check the corresponding \textbf{JSF}. To this end, 
we consider $T_{n}^{d}$ for $n=p$ (we are in type $\BM$ 
with $m=\frac{1}{2}(p-1)$) and $d=3$. 
Then $T_{n}^{d}$ 
has only Weyl factors which are simple $\Uo$-modules: its Weyl factors are 
$\Delta_1(3\varepsilon_1)$, $\Delta_1(2\varepsilon_1+\varepsilon_2)$, $\Delta_1(\varepsilon_1+\varepsilon_2+\varepsilon_3)$ and $\Delta_1(\varepsilon_1)$, 
all of which have zero \textbf{JSF}. This can be seen as usual and we only 
do the first case explicitly here.

A direct computations shows that 
$\langle 3\varepsilon_1+\rho,\alpha^{\vee}\rangle\leq p$ for 
all positive roots $\alpha\in\Phi^+$ except of $\alpha=\varepsilon_1$, where 
$\langle3\varepsilon_1+\rho,\varepsilon_1^{\vee}\rangle=p+4<2p$ 
(recall that $p\geq 5$). Then, because $p\geq 5$,
\begin{gather*}
3\varepsilon_1+\rho-p\varepsilon_1={\textstyle\frac{1}{2}}\cdot(\colorbox{mycolor}{$4-p$},\colorbox{mycolor}{$p-4$},p-6,\dots,3,1),
\end{gather*}
is a singular $\Uo$-weight. Thus, 
\textbf{JSF} of $\Delta_1(3\varepsilon_1)$ 
is zero.

Similarly for $d=5$, $T_{n}^{d}$ 
has only Weyl factors that are simple $\Uo$-modules:
\begin{gather*}
\Delta_1(5\varepsilon_1),\quad\quad\Delta_1(4\varepsilon_1+\varepsilon_2),\quad\quad\Delta_1(3\varepsilon_1+2\varepsilon_2),\quad\quad\Delta_1(3\varepsilon_1+\varepsilon_2+\varepsilon_3),\\
\Delta_1(2\varepsilon_1+2\varepsilon_2+\varepsilon_3),\quad\quad\Delta_1(2\varepsilon_1+\varepsilon_2+\varepsilon_3+\varepsilon_4),\quad\quad\Delta_1(\varepsilon_1+\varepsilon_2+\varepsilon_3+\varepsilon_4+\varepsilon_5),\quad\quad\Delta_1(3\varepsilon_1),\\
\Delta_1(2\varepsilon_1+\varepsilon_2),\quad\quad\Delta_1(\varepsilon_1+\varepsilon_2+\varepsilon_3),\quad\quad\Delta_1(\varepsilon_1),
\end{gather*} 
all of which have zero as their \textbf{JSF} when $p\geq 7$ 
(we omit the calculation which works as above).
The statement follows from Theorem~\ref{thm-schur-weyl3} (note that 
$n=p>d$) and Corollary~\ref{cor-cellsemisimple}.
\end{proof}

\begin{lem}\label{lem-brauerfirstb}
$\cal{B}_{d}(\delta)$  is non-semisimple for 
$d\in\{2,4,6,8,\dots\}\cup\{7,9,11,13,\dots\}$.
\end{lem}

\begin{proof}
By Corollary~\ref{cor-picturebrauer} it remains to 
verify the boundary values $d=2$ and $d=7$.

Assume first that $d=2$ and $p\geq 5$. 
We consider $T_{n}^{d}$ for $n=p$ (we are in type $\BM$ 
with $m=\frac{1}{2}(p-1)$). We have that $T_n^d$ 
has a Weyl factor of the form $\Delta_1(2\varepsilon_1)$. We calculate
\begin{gather*}
2\varepsilon_1+\rho-p\varepsilon_1=
{\textstyle\frac{1}{2}}\cdot(\colorbox{colormy}{\color{white}$2-p$},p-4,p-6\dots,3,1).
\end{gather*}
Thus, the positive root $\alpha\in\Phi^+$ of the form 
$\alpha=\varepsilon_1$ gives a non-zero 
contribution to \textbf{JSF} of $\Delta_1(2\varepsilon_1)$ 
(no other positive root $\alpha\in\Phi^+$ contributes and hence, we have 
no cancellations). 
As before, $\cal{B}_d(\delta)$ is non-semisimple by Theorem~\ref{thm-schur-weyl3} 
and by Corollary~\ref{cor-cellsemisimple}.

Assume now that $d=2$ and $p=3$.
This case can be done analogously by 
considering $T_{n}^{d}$ for $n=9=3p$ (type $\BM$ with 
$m=4$). We leave the details to the reader.

Next, let $d=7$ and $p\geq 7$. Then we proceed similar as above: 
we take $T_{n}^{d}$ for $n=p$ and 
the Weyl factor 
$\Delta_1(4\varepsilon_1+2\varepsilon_2+\varepsilon_3)$. 
Hence, we only need to check 
positive roots $\alpha\in\Phi^+$ of the form $\alpha=\varepsilon_1$, 
$\alpha=\varepsilon_1+\varepsilon_2$ and 
$\alpha=\varepsilon_1+\varepsilon_3$. 
All of them contribute at most once. We get 
(for $\lambda=4\varepsilon_1+2\varepsilon_2+\varepsilon_3$)
\begin{gather*}
\begin{aligned}
\lambda+\rho
-p\varepsilon_1&={\textstyle\frac{1}{2}}\cdot(\colorbox{colormy}{\color{white}$6-p$},p,p-4,p-8,p-10,\dots,3,1),\\
\lambda+\rho
-p(\varepsilon_1+\varepsilon_2)&={\textstyle\frac{1}{2}}\cdot(\colorbox{colormy}{\color{white}$6-p$},\colorbox{colormy}{\color{white}$-p$},p-4,p-8,p-10,\dots,3,1),\\
\lambda+\rho
-p(\varepsilon_1+\varepsilon_3)&={\textstyle\frac{1}{2}}\cdot(\colorbox{colormy}{\color{white}$6-p$},p,\colorbox{colormy}{\color{white}$-p-4$},p-8,p-10,\dots,3,1).
\end{aligned}
\end{gather*}
All of these are regular $\Uo$-weights and the first two cancel each other, but 
the last remains.
Thus, \textbf{JSF} of 
$\Delta_1(4\varepsilon_1+2\varepsilon_2+\varepsilon_3)$ is non-zero. 
The conclusion that $\cal{B}_{7}(\delta)$ is non-semisimple
follows again from Theorem~\ref{thm-schur-weyl3} 
and Corollary~\ref{cor-cellsemisimple}.
\end{proof}

\subsection{Proof of the semisimplicity criterion for \texorpdfstring{$\cal{B}_{d}(\delta)$}{B_d(delta)}}

\begin{proof}[Proof of Theorem~\ref{thm-brauer}]
\textit{(1).} 
By Lemma~\ref{lem-brauerseconda}, only 
the case $d=\frac{1}{2}(p-\delta_p+2)\leq\delta_p+1$ is missing 
for the `if' part, because in this case 
the corresponding 
Schur-Weyl-Brauer duality gives only a surjection. As in the 
proof of Theorem~\ref{thm-walledbrauer}, we 
can use Propositions~\ref{prop-kernelbrauer} and \ref{prop-abstract-nonsense} 
to handle this missing case. The `only if' part follows from 
Lemmas~\ref{lem-brauerseconde} 
and~\ref{lem-brauersecondc}.

\noindent\textit{(2).} This follows from Lemma~\ref{lem-brauersecondbb} respectively from the  
Lemmas~\ref{lem-brauerseconde} 
and~\ref{lem-brauersecondd}.

\noindent\textit{(3).} As in the 
proof of Theorem~\ref{thm-walledbrauer} 
we can again use Theorem~\ref{thm-transfer2}.

\noindent\textit{(4).} Directly from the Lemmas~\ref{lem-brauerfirst} 
and~\ref{lem-brauerfirstb}, since 
$\cal{B}_1(\delta)$ is always semisimple.

\noindent\textit{(5).} Again by Theorem~\ref{thm-transfer2}.

\noindent\textit{(6).} $\cal{B}_1(\delta)$ is 
clearly semisimple, while $\cal{B}_d(\delta)$ for 
$d\geq 2$ is not because of Lemma~\ref{lem-brauersecondnew}.

The theorem follows.
\end{proof}

\begin{rem}\label{rem-semi3}
Of course, the
semisimplicity criterion from Theorem~\ref{thm-brauer} 
was already observed 
before. In particular, the case $\K=\C$ and $\delta\in\Z$
goes back to a paper of Brown~\cite[Theorem~8D]{brown} and, in case $\delta$ is not an integer 
or $\K$ is an arbitrary field of characteristic zero 
and arbitrary $\delta\in\K$, to work of 
Wenzl~\cite[Corollary~3.3]{wenzl}. The case for $\K$ being a field of
arbitrary characteristic is treated by Rui  
in~\cite[Theorem~1.2]{rui}. To see that Rui's criterion matches ours,
we note that a slight reformulation of Rui's criterion was given by Rui together with Si later
in~\cite[Corollary~2.5]{rs}. The latter is easily seen to coincide with the one we obtain.
\end{rem}

\begin{rem}\label{rem-qauntizedb}
Using our approach we could reprove the semisimplicity criterion for the BMW algebra found 
in~\cite[Theorem~5.9]{rs2}, but decided to stay in the 
$q=1$ case.
\end{rem}

%% file: res/appendix2.tex
Here we recall some algebraic notions 
which we use to transfer our results from positive characteristic to characteristic zero. 
To this end, 
given a $\Z$-algebra $\cal{A}^{\Z}$ and any fixed field $\K$, we denote 
by $\cal{A}^{\K}=\cal{A}^{\Z}\otimes_{\Z}\K$ the \textit{scalar extension} of 
$\cal{A}$. Moreover, 
we assume 
throughout that all $\Z$-algebras 
are finitely generated and free.

Fix a $\Z$-algebra $\cal{A}^{\Z}$.
Recall that there is a \textit{trace form} 
$\langle\cdot,\cdot\rangle\colon \cal{A}^{\Z}\otimes \cal{A}^{\Z}\to \Z$ on $\cal{A}^{\Z}$ given as 
follows. Denote by $R_a\in\End_{\Z}(\cal{A}^{\Z})$ the right multiplication 
with $a\in \cal{A}^{\Z}$. By choosing a basis, we can identify $R_a$ with a 
matrix in $M_{\dim(\cal{A}^{\Z})}(\Z)$ and define
$
\langle a,b\rangle=\mathrm{tr}(R_b\circ R_a)\in\Z.
$
One can easily show that this assignment is independent of 
the choice of basis.

\begin{prop}\label{prop-trace}
Let $\{a_1,\dots,a_{\dim(\cal{A}^{\Z})}\}$ be any basis of $\cal{A}^{\Z}$. Then 
$\cal{A}^{\K}$ is semisimple if and only if
$
\det(M_{\Z})\in\K-\{0\}$, where $M_{\Z}=\left(\langle a_i,a_j\rangle\right)_{i,j=1}^{\dim(\cal{A}^{\Z})}.
$
\end{prop}

\begin{proof}
This is proven in~\cite[Proposition~4.46]{katu}: the 
vanishing of the determinant as above is equivalent to the degeneracy of 
the trace form.
\end{proof}

Recall that we denote by $\F$ the finite field with $p$ 
elements.

\begin{prop}\label{prop-transfer}
Let $\K$ be a field with $\Char(\K)=0$. Then $\cal{A}^{\K}$ is 
semisimple if and only if $\cal{A}^{\F}$ is semisimple for infinitely many primes $p$.
\end{prop}

\begin{proof}
Assume 
there is a prime $p$ such that the algebra $\cal{A}^{\F}$ is semisimple. 
Then, by Proposition~\ref{prop-trace}, we have that $p$ does not 
divide $\det(M_{\Z})$. In particular, $\det(M_{\Z}) \in \Z- \{0\} \subset \K- \{0\}$ 
and $\cal{A}^{\K}$ is therefore semisimple according to the same proposition.

If $\cal{A}^{\K}$ is semisimple, 
then $\det(M_{\Z}) \in \K- \{0\}$ by Proposition~\ref{prop-trace}. 
By choosing a $\Z$-basis of $\cal{A}^{\K}$, this 
in turn implies $\det(M_{\Z}) \in \Z- \{0\}$. Hence, by Proposition~\ref{prop-trace}, 
$\mathcal {A}^{\F}$ is semisimple for all primes $p > |\det(M_{\Z})|$.
\end{proof}

\begin{thm}\label{thm-transfer2}
Let $\Char(\K)=0$ and $\delta\in\Z$.
\begin{enumerate}[(a)]
\item $\cal{B}_{r,s}(\delta)$ is semisimple over $\K$ if and only if 
$\cal{B}_{r,s}(\delta)$ is semisimple over $\F$ for infinitely many primes $p$.
\item $\cal{B}_d(\delta)$ is semisimple over $\K$ if and only if 
$\cal{B}_d(\delta)$ is semisimple over $\F$ for infinitely many primes $p$.
\end{enumerate}
\end{thm}

\begin{proof}
Note that $\cal{B}_{r,s}(\delta)$ and 
$\cal{B}_d(\delta)$ given in 
Definitions~\ref{defn-walledbrauer} 
and~\ref{defn-brauer} and considered over $\K$ are obtained 
from integral versions $\cal{B}^{\Z}_{r,s}(\delta)$ 
and $\cal{B}^{\Z}_d(\delta)$ via scalar extension 
(the integral versions of these algebras can be found 
in~\cite[Section~2]{dds} and in~\cite[Section~4]{gl} 
respectively). Hence, 
the two statements follow from Proposition~\ref{prop-transfer}.
\end{proof}

%% file: res/appendix.tex
For the convenience of the reader we list here the root and weight data of types 
$\AM$ attached to $\fg=\gll{m}$, of type $\BM$ attached to $\fg=\soo{2m+1}$ for $m\geq 2$, 
of type $\CM$ attached to $\fg=\spo{2m}$ for $m\geq 3$, and of 
type $\DM$ attached to $\fg=\soo{2m}$ for $m\geq 4$.  
We assume $1\leq i,j\leq m$, $i\not=j$  and $1\leq i^{\prime}\leq m-1$.
The root system and its dual is realized
inside the Euclidean space $E=\R^m$ with standard 
basis $\varepsilon_1,\dots,\varepsilon_m$ and inner product 
determined by 
$\langle\varepsilon_i,\varepsilon_j\rangle=\delta_{i,j}$. 
Then the data is given as follows:
\newline
\renewcommand{\arraystretch}{1.3}
\begin{array}[t]{c|c|c}
&\AM&\BM\\
\hline
\hline
\Phi 
&\{\varepsilon_i-\varepsilon_j\mid 1\leq i\neq j\leq m\}
&\{\pm\varepsilon_i\pm\varepsilon_j,\pm\varepsilon_i\mid 1\leq i\neq j\leq m\}
\\
\hline
\Phi^+
&\{\varepsilon_i-\varepsilon_j\mid 1\leq i<j\leq m\}
&\{\varepsilon_i\pm\varepsilon_j,\varepsilon_i\mid 1\leq i<j\leq m\}
\\
\hline
\Pi
&\{\alpha_{i^{\prime}}=\varepsilon_{i^{\prime}}-\varepsilon_{i^{\prime}+1}\}
&\{\alpha_{i^{\prime}}=\varepsilon_{i^{\prime}}-\varepsilon_{i^{\prime}+1}\}\cup\{\alpha_m=\varepsilon_m\}
\\
\hline
\Pi^\vee
&\{\alpha^{\vee}_{i^{\prime}}=\varepsilon_{i^{\prime}}-\varepsilon_{i^{\prime}+1}\}
&\{\alpha^{\vee}_{i^{\prime}}=\varepsilon_{i^{\prime}}-\varepsilon_{i^{\prime}+1}\}\cup\{\alpha^\vee_m=2\varepsilon_{m}\}
\\
\hline
\raisebox{-.35cm}{$\rho$} &\raisebox{-.35cm}{$(m-1,m-2,\dots,1,0)$}&
\begin{array}[t]{c}(m-{\textstyle\frac{1}{2}},m-{\textstyle\frac{3}{2}},\dots,{\textstyle\frac{3}{2}},{\textstyle\frac{1}{2}})\\
={\textstyle\frac{1}{2}}\cdot (2m-1,2m-3,\dots,3,1)
\end{array}
\\
\hline
X&\hspace*{0.18cm}\{\lambda=\sum_{i=1}^{m}\lambda_i\varepsilon_i\in \R^m\mid \lambda_i\in\Z\}\hspace*{0.18cm}&\{\lambda=\sum_{i=1}^{m}\lambda_i\varepsilon_i\in \R^m\mid \lambda_i\in\frac{1}{2}\Z,\;\lambda_i-\lambda_j\in\Z\}\\
\hline
X^+&\{\lambda\in X\mid \lambda_1\geq\dots\geq\lambda_m\}&\{\lambda\in X\mid \lambda_1\geq\dots\geq\lambda_m\geq0\}\\
\hline
\omega_i
&\omega_{i^{\prime}}=\sum_{j=1}^{i^{\prime}}\varepsilon_j
&\omega_{i^{\prime}}=\sum_{j=1}^{i^{\prime}}\varepsilon_j,\quad \omega_m=  \frac{1}{2}(\varepsilon_1+\dots+\varepsilon_m)
\\
\hline
W&S_m&S_m\ltimes(\Z/2\Z)^m
\end{array}
\vspace{0.1cm}

\begin{center}
\begin{array}[t]{c|c|c}
&\CM&\DM\\
\hline
\hline
\Phi 
&\{\pm\varepsilon_i\pm\varepsilon_j,\pm2\varepsilon_i\mid 1\leq i\neq j\leq m\}
&\{\pm\varepsilon_i\pm\varepsilon_j\mid 1\leq i\neq j\leq m\}\\
\hline
\Phi^+
&\{\varepsilon_i\pm\varepsilon_j,2\varepsilon_i\mid 1\leq i<j\leq m\}
&\{\varepsilon_i\pm\varepsilon_j\mid 1\leq i<j\leq m\}\\
\hline
\Pi
&\{\alpha_{i^{\prime}}=\varepsilon_{i^{\prime}}-\varepsilon_{i^{\prime}+1}\}\cup\{\alpha_m=2\varepsilon_m\}
&\{\alpha_{i^{\prime}}=\varepsilon_{i^{\prime}}-\varepsilon_{i^{\prime}+1}\}\cup\{\alpha_m=\varepsilon_{m-1}+\varepsilon_m\}
\\
\hline
\Pi^\vee
&\{\alpha_{i^{\prime}}^{\vee}=\varepsilon_{i^{\prime}}-\varepsilon_{i^{\prime}+1}\}\cup\{\alpha^\vee_m=\varepsilon_m\}
&\{\alpha_{i^{\prime}}^{\vee}=\varepsilon_{i^{\prime}}-\varepsilon_{i^{\prime}+1}\}\cup\{\alpha^\vee_m=\varepsilon_{m-1}+\varepsilon_m\}
\\
\hline
\rho
&(m,m-1,\dots,2,1)&(m-1,m-2,\dots,1,0)
\\
\hline
X&\{\lambda=\sum_{i=1}^{m}\lambda_i\varepsilon_i\in \R^m\mid \lambda_i\in\Z\}&\{\lambda=\sum_{i=1}^{m}\lambda_i\varepsilon_i\in \R^m\mid \lambda_i\in\frac{1}{2}\Z,\;\lambda_i-\lambda_j\in\Z\}\\
\hline
 X^+&\{\lambda\in X\mid \lambda_1\geq\dots\geq\lambda_m\geq0\}&\{\lambda\in X\mid \lambda_1\geq\dots\geq\lambda_{m-1}\geq |\lambda_m|\geq 0\}\\
 \hline
\raisebox{-.7cm}{$\omega_i$}&\raisebox{-.7cm}{$\omega_i=\sum_{j=1}^i\varepsilon_j$}&
\begin{array}[t]{c}
\omega_{i^{\prime\prime}}=\sum_{j=1}^{i^{\prime\prime}}\varepsilon_{i^{\prime}},\quad (1\leq i^{\prime\prime}\leq m-2),\\
 \omega_{m-1}=  \frac{1}{2}(\varepsilon_1+\dots+\varepsilon_{m-1}-\varepsilon_m),\\
 \omega_{m}=  \frac{1}{2}(\varepsilon_1+\dots+\varepsilon_{m-1}+\varepsilon_m)\\
\end{array}
\\
\hline
W&S_m\ltimes(\Z/2\Z)^m&S_m\ltimes(\Z/2\Z)^{m-1}
\end{array}
\end{center}
\vspace*{0.1cm}

In type $\AM$, the simple transpositions $s_r$ of $S_m=W$ act on $X$ 
via permutation. The (dot-)singular type $\AM$ weights 
in the sense of Definition~\ref{defn-singular} 
and Convention~\ref{nota-rho} are:
\begin{gather*}
\begin{aligned}
\lambda\in X\text{ is dot-singular }&\Leftrightarrow\text{ there exist }i\neq j\text{ such that }(\lambda+\rho)_i=(\lambda+\rho)_j,
\\
\lambda\in X\text{ is singular }&\Leftrightarrow\text{ there exist }i\neq j\text{ such that }\lambda_i=\lambda_j.
\end{aligned}
\end{gather*}
For types $\BM$ and $\CM$ and
$i=1,\dots,m-1$ the elements $s_i\in S_m$ act as in type $\AM$, while $s_m\colon(\lambda_1,\dots,\lambda_m)\mapsto
(\lambda_1,\dots,-\lambda_m)$. The (dot-)singular type 
$\BM$ and $\CM$ weights are:
\begin{gather*}
\begin{aligned}
\lambda\in X\text{ is dot-singular }&\Leftrightarrow\text{ there exist }i\neq j\text{ such that }(\lambda+\rho)_i=\pm(\lambda+\rho)_j\text{ or }(\lambda+\rho)_i=0,
\\
\lambda\in X\text{ is singular }&\Leftrightarrow\text{ there exist }i\neq j\text{ such that }\lambda_i=\pm\lambda_j\text{ or }\lambda_i=0.
\end{aligned}
\end{gather*}
For type $\DM$, the action of $W$ on $X$ is as in types $\BM$ and $\CM$, but 
$s_m$ changes two signs instead of one. 
The (dot-)singular type 
$\DM$ weights are given as accordingly.